\documentclass[12pt]{amsart}
\usepackage{amsmath,amssymb,latexsym,cancel,rotating}
\usepackage{graphicx,amssymb,mathrsfs,amsmath,color,fancyhdr,amsthm}
\usepackage[all]{xy}
\usepackage{diagrams}

\newtheorem{theorem}{Theorem}[section]
\newtheorem{lemma}[theorem]{Lemma}
\newtheorem{corollary}[theorem]{Corollary}
\newtheorem{definition}[theorem]{Definition}
\newtheorem{proposition}[theorem]{Proposition}

\newtheorem{remark}[theorem]{Remark}

\def\cM{{\mathcal M}}\def\cP{{\mathcal P}} 
\def\cC{{\mathcal C}} \def\cD{{\mathcal D}}    \def\cL{{\mathcal L}} \def\cF{{\mathcal F}} 
\def\cH{{\mathcal H}}      \def\cU{{\mathcal U}} \def\cV{{\mathcal V}}
 \def\cO{{\mathcal O}} 
\def\cI{{\mathcal I}}      \def\cR{{\mathcal R}} \def\cG{{\mathcal G}} \def\cN{{\mathcal N}}
\def\cA{{\mathcal A}}    
\def\cK{{\mathcal K}} \def\rE{{\rm E}} \def\rG{{\rm G}}\def\rH{{\textrm H}} \def\rC{{\rm C}}  \def\rP{{\rm P}}
\def\bbN{{\mathbb N}}  \def\bbZ{{\mathbb Z}}  
    \def\bbF{{\mathbb F}}
\def\bbC{{\mathbb C}}    \def\bbP{{\mathbb P}}

\def\fh{\mathfrak h}

\def\Hom{\mbox{\rm Hom}}  
       \def\im{\mbox{\rm Im}\,}
\def\Ext{\mbox{\rm Ext}\,}   \def\Ker{\mbox{\rm Ker}\,}
\def\dim{\mbox{\rm dim}\,}   \def\End{\mbox{\rm End}}
\def\Aut{\mbox{\rm Aut}}  \def\Stab{\mbox{\rm Stab}} \def\tw{{\rm tw}}  \def\red{{\rm red}} \def\Htp{\mbox{\rm Htp}}
\def\red{{\rm red}}
\def\udim{{\mathbf dim\,}}  \def\oline{\overline}  \def\uline{\underline}
\def\mod{\mbox{\rm mod}\,}  
\def\rad{{\rm rad}\,}    \def\ind{{\rm ind}\,}
 \def\udim{\mbox{\underline{\rm dim}}\,}
\def\rep{\mbox{\rm rep}\,}   
  
 \def\Spec{\mbox{\rm Spec}\,}

\def\St{\mbox{\rm St}}
\def\Var{\mbox{\rm Var}}

\def\supp{{\rm supp}}

\newcommand \bd{{\bf d}}

\newcommand \ualpha{{\uline{\alpha}}}

\newcommand\fg{\mathfrak g}
\newcommand\fn{\mathfrak n}

\newcommand\ue{{\underline{e}}}

\begin{document}

\title[]{Lie algebras arising from two-periodic projective complex and derived categories}

\author{Jiepeng Fang, Yixin Lan, Jie Xiao}

\address{School of Mathematical Sciences, Peking University, Beijing 100871, P.R. China}
\email{fangjp@math.pku.edu.cn (J. Fang)}
\address{Academy of Mathematics and Systems Science, Chinese Academy of Sciences, Beijing 100190, P.R. China}
\email{lanyixin@amss.ac.cn (Y. Lan)}
\address{School of Mathematical Sciences, Beijing Normal University, Beijing 100875, P.R. China}
\email{jxiao@bnu.edu.cn (J. Xiao)}

\renewcommand{\subjclassname}{\textup{2020} Mathematics Subject Classification}


\bibliographystyle{abbrv}

\maketitle 

\begin{center}
\small\emph{Dedicated to Professor Claus Michael Ringel for his 80th birthday}
\end{center}

\begin{abstract}
Let $A$ be a finite-dimensional $\mathbb{C}$-algebra of finite global dimension and $\mathcal{A}$ be the category of finitely generated right $A$-modules. By using of the category of two-periodic projective complexes $\mathcal{C}_2(\mathcal{P})$, we construct the motivic Bridgeland's Hall algebra for $\mathcal{A}$, where structure constants are given by Poincar\'{e} polynomials in $t$, then construct a $\mathbb{C}$-Lie subalgebra $\mathfrak{g}=\mathfrak{n}\oplus \mathfrak{h}$ at $t=-1$, where $\mathfrak{n}$ is constructed by stack functions about indecomposable radical complexes, and $\mathfrak{h}$ is by contractible complexes. For the stable category $\mathcal{K}_2(\mathcal{P})$ of $\mathcal{C}_2(\mathcal{P})$, we construct its moduli spaces and a $\mathbb{C}$-Lie algebra $\tilde{\mathfrak{g}}=\tilde{\mathfrak{n}}\oplus \tilde{\mathfrak{h}}$, where $\tilde{\mathfrak{n}}$ is constructed by support-indecomposable constructible functions, and $\tilde{\mathfrak{h}}$ is by the Grothendieck group of $\mathcal{K}_2(\mathcal{P})$. We prove that the natural functor $\mathcal{C}_2(\mathcal{P})\rightarrow \mathcal{K}_2(\mathcal{P})$ together with the natural isomorphism between Grothendieck groups of $\mathcal{A}$ and $\mathcal{K}_2(\mathcal{P})$ induces a Lie algebra isomorphism $\mathfrak{g}\cong\tilde{\mathfrak{g}}$. This makes clear that the structure constants at $t=-1$ provided by Bridgeland in \cite{Bridgeland-2013} in terms of exact structure of $\mathcal{C}_2(\mathcal{P})$ precisely equal to that given in \cite{Peng-Xiao-2000} in terms of triangulated category structure of $\mathcal{K}_2(\mathcal{P})$.
\end{abstract}

\setcounter{tocdepth}{1}\tableofcontents

\section{Introduction}

\subsection{Background}\

Let $Q=(I,H,s,t)$ be a finite acyclic quiver, then $(2\delta_{ij}-n_{ij})_{i,j\in I}$ is a symmetric generalized Cartan matrix, where $n_{ij}$ is the number of edges connecting $i,j\in I$ in the underline graph of $Q$. The corresponding Kac-Moody Lie algebra $\cG$ is closely related to the category of finite-dimensional representations of $Q$. Since Ringel introduced Hall algebra for quiver representations to realize the nilpotent part $U^+_v(\fn)$ and $\fn$ of quantum group $U_v(\cG)$ and $\cG$, respectively in \cite{Ringel-1990,Ringel-1990-1,Ringel-1992}, there are two important results realizing $\cG$ globally by representation theory of $Q$. 

In \cite{Peng-Xiao-2000}, Peng-Xiao constructed a Lie algebra $\tilde{\fg}_{(q-1)}$ over $\bbZ/(q-1)\bbZ$ for any Hom-finite, two-periodic triangulated category over a finite field $\bbF_q$. In particular, for the root category of finite-dimensional representations of $Q$ over $\bbF_q$, they used the resulting Lie algebra to realize $\cG$ globally. Later, in \cite{Bridgeland-2013}, Bridgeland constructed his Hall algebra $\cD\cH^{\red}_q$ for any finitary abelian category of finite global dimension over $\bbF_q$. In particular, in the hereditary case, for the category of finite-dimensional quiver representations, he used the resulting algebra to realize $U_{v=\sqrt{q}}(\cG)$ globally. By the canonical process recovering $\cG$ from $U_v(\cG)$, Bridgeland also provided a method to realize $\cG$ globally.

It is interesting that Peng-Xiao's method and Bridgeland's method reach the same goal from two different ways, where Peng-Xiao used the root category, equivalently, the stable category of two-periodic projective complexes of quiver representations, while Bridgeland relied on the category of two-periodic projective complexes of quiver representations. Obviously, only for the realizations of Lie algebras, we can change the finite field $\bbF_q$ to the complex field $\bbC$, in analogue to \cite{Riedtmann-1994,Xiao-Xu-Zhang-2006}. In present paper, we reveal the connections of Peng-Xiao \cite{Peng-Xiao-2000} and Bridgeland \cite{Bridgeland-2013} intrinsically. More generally, our main concern is about the triangulated categories which are derived from abelian category of finite global dimension, since for such an abelian category $\cA$, Bridgeland's construction is valid and Peng-Xiao's construction also works for the stable category of two-periodic projective complexes of $\cA$. We not only consider the hereditary case.

\subsection{Main result}\

For any finite-dimensional $\bbC$-algebra $A$ of finite global dimension, let $\cA$ be the category of finitely generated right $A$-module, $\cC_2(\cP)$ be the category of two-periodic projective complexes, and $\cK_2(\cP)$ be its stable category. We denote by $K(\cA)$ the Grothendieck group of $\cA$ and denote by $K_0$ the Grothendieck group of $\cK_2(\cP)$. Our work and main results are divided into three parts as follows.

\subsubsection{} We construct a $\bbC$-Lie algebra $\fg$ from $\cC_2(\cP)$.

Firstly, we define a moduli stack $\cM$ to parametrize isomorphism classes of two-periodic complexes of $\cA$, and define two locally closed substacks $\cN^{\rad}\subset \cN\subset \cM$ corresponding to radical projective and projective complexes, respectively. 

Then, we generalize Bridgeland's construction in the framework of motivic Hall algebra defined by Joyce in \cite{Joyce-2007}, and construct a motivic form of Bridgeland's Hall algebra $\cD\cH_t^{\red}(\cA)$ over $\bbC(t)$ whose structure constants are given by (virtual) Poincar\'{e} polynomials in $t$. Inspired by the canonical process recovering $U(\cG)$ and $\cG$ from $U_v(\cG)$, see \cite[Section 3.4]{Jin-Jin-2002} or \cite{Murphy-2018}, we define a $\bbC_{-1}$-subalgebra $\cD\cH_t(\cA)_{\bbC_{-1}}$ of $\cD\cH_t^{\red}(\cA)$ generated by elements of the form
$$[[\cO/\rG_{\ue}]\hookrightarrow \cN^{\rad}\hookrightarrow \cM],\ b_{\alpha},\ \frac{b_{\alpha}-1}{-t-1},$$
where $\bbC_{-1}=\bbC[t]_{(t+1)}=\{\frac{f(t)}{g(t)}\in \bbC(t)|g(-1)\not=0\}$ is a local ring with residue field $\bbC$, $[\cO/\rG_{\ue}]\subset  \cN^{\rad}$ is a constructible substack and the elements $b_\alpha, \frac{b_{\alpha}-1}{-t-1}$ are constructed by contractible complexes for $\alpha\in K(\cA)$, see Subsection \ref{Classical limit} for details. Then we define the classical limit to be the $\bbC$-algebra $\cD\cH_{-1}(\cA)=\bbC_{-1}/(t+1)\bbC_{-1}\otimes_{\bbC_{-1}}\cD\cH_t(\cA)_{\bbC_{-1}}$ which has a natural Lie bracket given by commutator as an associative algebra. 

Finally, we define a $\bbC$-subspace $\fg=\fn\oplus \fh\subset \cD\cH_{-1}(\cA)$, where $\fn$ is spanned by elements of the form 
$$1\otimes [[\cO/\rG_{\ue}]\hookrightarrow \cN^{\rad}\hookrightarrow \cM],$$
where $[\cO/\rG_{\ue}]\subset  \cN^{\rad}$ is a constructible substack consisting of geometric points corresponding to indecomposable radical complexes, and $\fh$ is spanned by $1\otimes \frac{b_{\alpha}-1}{-t-1}$ for $\alpha\in K(\cA)$. We denote by $h_\alpha=1\otimes \frac{b_{\alpha}-1}{-t-1}$ for $\alpha\in K(\cA)$, then we prove these elements satisfy $[h_\alpha,h_\beta]=0$ and $h_\alpha+h_\beta=h_{\alpha+\beta}$. As a consequence, $\alpha\mapsto h_\alpha$ extends to an isomorphism $\bbC\otimes_{\bbZ}K(\cA)\cong \fh$. Moreover, we prove the following theorem.

\begin{theorem}\label{Theorem 1}
The $\bbC$-subspace $\fg\subset \cD\cH_{-1}(\cA)$ is closed under the Lie bracket, and so it is a Lie subalgebra.
\end{theorem}

We remark why we need to construct the motivic Bridgeland's Hall algebra. For constructing a $\bbC$-Lie algebra from $\cC_2(\cP)$, the most natural option maybe generalizing Bridgeland's Hall algebra via constructible functions on varieties or stacks, which is similar to the generalization of Ringel-Hall algebra for quiver representations via constructible functions in \cite{Riedtmann-1994, Joyce-2007}. During our study, we realize that it is difficult to construct desired elements $h_\alpha$ for $\alpha\in K(\cA)$ via constructible functions, and we realize that the twist $\sqrt{q}^{\langle X^1,Y^1\rangle+\langle X^0,Y^0\rangle}$ given by Bridgeland, see \cite[(15)]{Bridgeland-2013}, plays an important role in his construction over $\bbF_q$. In order to overcome these difficulties, we construct the motivic Bridgeland's Hall algebra, where we have to choose Poincar\'{e} polynomial $\Upsilon$ to define the structure constants because of the fact that $\Upsilon(\bbC^n)=t^{2n}$ and Bridgeland's twist.

\subsubsection{} We construct a $\bbC$-Lie algebra $\tilde{\fg}$ from $\cK_2(\cP)$, which is a generalization of Peng-Xiao's Lie algebra $\tilde{\fg}_{(q-1)}$ over $\bbZ/(q-1)\bbZ$.

Firstly, for any $\bd\in K_0$, we define an ind-constructible stack $\rP_2(A,\bd)$ together with a group $\rG_{\bd}$-action to parametrize the homotopy equivalence classes of objects in $\cK_2(\cP)$ whose images in the Grothendieck group are $\bd$. We prove that $\rP_2(A,\bd)$ together with $\rG_{\bd}$ is determined by the derived category $\cD^b(A)$.

Then, on the $\bbC$-vector space $\tilde{\fg}=\tilde{\fn}\oplus\tilde{\fh}$, where $\tilde{\fn}=\bigoplus_{\bd\in K_0}\tilde{\fn}_{\bd}$, each $\tilde{\fn}_{\bd}$ is the $\bbC$-vector space of $\rG_{\bd}$-invariant support-indecomposable constructible functions on $\rP_2(A,\bd)$, and $\tilde{\fh}=\bbC\otimes_{\bbZ}K_0$, we define an anti-symmetric bracket $[-,-]$ by the triangulated category structure of $\cK_2(\cP)$.

Finally, we prove the following theorem.

\begin{theorem}\label{Theorem 2}
The $\bbC$-vector space $\tilde{\fg}$ is a Lie algebra with the bracket $[-,-]$.
\end{theorem}

We emphasize that this part is actually a revise of Xiao-Xu-Zhang's work \cite{Xiao-Xu-Zhang-2006}. We add the $*$-saturated condition to refine their definition for $\rP_2(A,\bd)$, see Proposition \ref{saturated isomorphism} for the reason, and rewrite a self-contained proof for Theorem \ref{Theorem 2} which uses the octahedral axiom to prove the Jacobi identity. Our proof is a reformulation and a refinement of Xiao-Xu-Zhang \cite{Xiao-Xu-Zhang-2006}, and the basic idea is of course the same as Peng-Xiao \cite{Peng-Xiao-2000} with some improvement by Hubery \cite{Hubery-2006}.

\subsubsection{}

There is a natural functor $\Phi:\cC_2(\cP)\rightarrow \cK_2(\cP)$ which preserves objects and maps morphisms in $\cC_2(\cP)$ to their homotopy classes. For any indecomposable radical complex $X$ in $\cC_2(\cP)$, the image $\Phi(X)=X$ is an indecomposable object in $\cK_2(\cP)$. Hence $\Phi$ naturally induces a linear map $\varphi_{\fn}:\fn\rightarrow \tilde{\fn}$. The Grothendieck groups $K(\cA)$ of $\cA$ and $K_0$ of $\cK_2(\cP)$ are naturally isomorphic, see \cite[Proposition 2.11]{Fu-2012}, thus there is a linear isomorphism 
$\varphi_{\fh}:\fh\cong \bbC\otimes_{\bbZ}K(\cA)\rightarrow \bbC\otimes_{\bbZ}K_0=\tilde{\fh}$. We prove the following theorem.

\begin{theorem}\label{Theorem 3}
The linear map $\varphi=\begin{pmatrix}\varphi_{\fn}&\\&\varphi_{\fh}\end{pmatrix}:\fg=\fn\oplus \fh\rightarrow \tilde{\fn}\oplus \tilde{\fh}=\tilde{\fg}$ is a Lie algebra isomorphism.
\end{theorem}

We prove above theorem by verifying the structure constants of $\fg$ coincide with the structure constants of $\tilde{\fg}$. Our idea comes from the calculation over $\bbZ/(q-1)\bbZ$, for example, let $X,Y,Z$ be indecomposable radical two-periodic projective complexes over $\bbF_q$, consider
\begin{align*}
g^Z_{XY}&=|\{Z'\subset Z|Z/Z'\cong X,Z'\cong Y\}|,\\
F^Z_{XY}&=|W(X,Y;Z)/\Aut_{\cK_2(\cP_{\bbF_q})}(X)\times \Aut_{\cK_2(\cP_{\bbF_q})}(Y)|,
\end{align*}
the structure constants appearing in Bridgeland's Hall algebra \cite[Section 2.3]{Bridgeland-2013} and Peng-Xiao's Lie algebra \cite[Section 3.2]{Peng-Xiao-2000}, respectively, where $W(X,Y;Z)$ is a subset of $\Hom_{\cK_2(\cP)}(Y,Z)\times\Hom_{\cK_2(\cP)}(Z,X)\times\Hom_{\cK_2(\cP)}(X,Y^*)$ consisting of $(f,g,h)$ such that $Y\xrightarrow{f}Z\xrightarrow{g}X\xrightarrow{h}Y^*$ is a distinguished triangle, and the group $\Aut_{\cK_2(\cP_{\bbF_q})}(X)\times \Aut_{\cK_2(\cP_{\bbF_q})}(Y)$ acts on it via $(\alpha, \beta).(f,g,h)=(f\beta^{-1},\alpha g,\beta^*h\alpha^{-1})$. On the one hand, by Riedtmann-Peng's formula, we have
$$g^Z_{XY}=\frac{|\Ext^1_{\cC_2(\cP_{\bbF_q})}(X,Y)_Z||\Aut_{\cC_2(\cP_{\bbF_q})}(Z)|}{|\Hom_{\cC_2(\cP_{\bbF_q})}(X,Y)||\Aut_{\cC_2(\cP_{\bbF_q})}(X)||\Aut_{\cC_2(\cP_{\bbF_q})}(Y)|},$$
where $\Ext^1_{\cC_2(\cP_{\bbF_q})}(X,Y)_Z\subset \Ext^1_{\cC_2(\cP_{\bbF_q})}(X,Y)$ is a subset consisting of equivalence classes of short exact sequences whose middle terms are isomorphic to $Z$, see \cite{Ringel-1996, Schiffmann-2012-1}. On the other hand, we divide $W(X,Y;Z)=\bigsqcup_{h\in \Hom_{\cK_2(\cP_{\bbF_q})}(X,Y^*)_{Z^*}}W(X,Y;Z)_h$, where $\Hom_{\cK_2(\cP_{\bbF_q})}(X,Y^*)_{Z^*}\subset  \Hom_{\cK_2(\cP_{\bbF_q})}(X,Y^*)$ is a subset consisting of morphisms whose mapping cones are homotopy equivalent to $Z^*$, and $W(X,Y;Z)_h$ is a subset of $\Hom_{\cK_2(\cP)}(Y,Z)\times\Hom_{\cK_2(\cP)}(Z,X)$ consisting of $(f,g)$ such that $Y\xrightarrow{f}Z\xrightarrow{g}X\xrightarrow{h}Y^*$ is a distinguished triangle. The group $\Aut_{\cK_2(\cP_{\bbF_q})}(Z)$ acts transitively on $W(X,Y;Z)_h$ via $\gamma.(f,g)=(\gamma f,g\gamma^{-1})$ and the stabilizer is bijective to a vector space, see $(***)$ in the proof of Theorem \ref{two Lie algebra isomorphism}, hence
\begin{align*}
|W(X,Y;Z)|&=\!\!\!\!\sum_{h\in \Hom_{\cK_2(\cP_{\bbF_q})}(X,Y^*)_{Z^*}}\!\!\!\!|W(X,Y;Z)_h|
\equiv \!\!\!\!\sum_{h\in \Hom_{\cK_2(\cP_{\bbF_q})}(X,Y^*)_{Z^*}}\!\!\!\!|\Aut_{\cK_2(\cP_{\bbF_q})}(Z)|\\
&\equiv |\Hom_{\cK_2(\cP_{\bbF_q})}(X,Y^*)_{Z^*}||\Aut_{\cK_2(\cP_{\bbF_q})}(Z)|\ \mod (q-1).
\end{align*}
Moreover, the stabilizers of any $\Aut_{\cK_2(\cP_{\bbF_q})}(X)\times \Aut_{\cK_2(\cP_{\bbF_q})}(Y)$-orbits in $W(X,Y;Z)$ are bijective to vector spaces,  see $(***\ *)$ in the proof of Theorem \ref{two Lie algebra isomorphism}, hence 
\begin{align*}
F^Z_{XY}&\equiv\frac{|W(X,Y;Z)|}{|\Aut_{\cK_2(\cP_{\bbF_q})}(X)|| \Aut_{\cK_2(\cP_{\bbF_q})}(Y)|}\\
&\equiv\frac{|\Hom_{\cK_2(\cP_{\bbF_q})}(X,Y^*)_{Z^*}||\Aut_{\cK_2(\cP_{\bbF_q})}(Z)|}{|\Aut_{\cK_2(\cP_{\bbF_q})}(X)|| \Aut_{\cK_2(\cP_{\bbF_q})}(Y)|}\ \mod (q-1).
\end{align*}
If $\Ext^1_{\cC_2(\cP_{\bbF_q})}(X,Y)_Z\not=\varnothing$, then $\Ext^1_{\cC_2(\cP_{\bbF_q})}(X,Y)_Z\cong\! \Hom_{\cK_2(\cP_{\bbF_q})}(X,Y^*)_{Z^*}$, by Lemma \ref{bijection between Ext and Hom}. For any radical complex $L$, we have $|\Aut_{\cC_2(\cP_{\bbF_q})}(L)|\equiv|\Aut_{\cK_2(\cP_{\bbF_q})}(L)|\ \mod (q-1)$, by Corollary \ref{automorphism groups coincide}. Therefore, 
$$g^Z_{XY}\equiv F^Z_{XY}\ \mod (q-1).$$
This is one of (the most important) relations between the structure constants of  Bridgeland's Hall algebra and of Peng-Xiao's Lie algebra over $\bbZ/(q-1)\bbZ$. Compare with (\ref{g=F}) in the proof of Theorem \ref{two Lie algebra isomorphism}, which is an analogue equation over $\bbC$.

It is worth remarking that the proof of Theorem \ref{Theorem 3} does not depend on Theorem \ref{Theorem 1} nor Theorem \ref{Theorem 2}. We point out two applications of our main theorems. 

(I) From Theorem \ref{Theorem 2} and Theorem \ref{Theorem 3}, since the Lie algebra $\tilde{\fg}$ is completely defined by $\cK_2(\cP)$ which is determined by the derived category $\cD^b(A)$, we intrinsically prove that Bridgeland's Hall algebra contains a Lie algebra $\fg$ determined by $\cD^b(A)$. This partially reflects the relation between Bridgeland's Hall algebra and the derived category, although it has been revealed by earlier study. In the hereditary case, Yanagida \cite{Yanagida-2016} proved that Bridgeland's Hall algebra is isomorphic to the Drinfeld double Ringel-Hall algebra which is derived invariant by Cramer \cite{Cramer-2010}. In general, Gorsky \cite{Gorsky-2016} related Bridgeland's Hall algebra as a kind of semi-derived Hall algebra with derived Hall algebra by To\"{e}n \cite{Toen-2006}, also Xiao-Xu \cite{Xiao-Xu-2008}.

(II) From Theorem \ref{Theorem 1} and Theorem \ref{Theorem 3}, the $\bbC$-space $\tilde{\fg}$ is a Lie algebra can be regarded as a corollary. In other word, we give a much simpler proof of Theorem \ref{Theorem 2} (main theorems in \cite{Peng-Xiao-2000,Xiao-Xu-Zhang-2006}) than which uses the octahedral axiom to prove the Jacobi identity.

\subsection{Assumption}\

Throughout this paper, we assume that $A$ is a finite-dimensional $\bbC$-algebra of finite global dimension. Up to Morita equivalence, we may assume $A=\bbC Q/J$ is the quotient of a path algebra of a finite quiver $Q=(I, H, s,t)$ with relations $R$, where $J$ is the admissible ideal generated by $R$. We identify the category of finitely generated right $A$-modules with the category of finite-dimensional $(Q,J)$-representations
$$\cA=\mod A\simeq \rep(Q,J),$$
see \cite{Assem-Simson-Skowronski-2006}. We denote by $\cP$ the full subcategory consisting of projective modules, and fix a complete set of indecomposable projective modules $\{P_i\mid i\in I\}$ up to isomorphism.

\section{Categories of two-periodic projective complexes}\label{Categories of two-periodic projective complexes}

\subsection{The categories $\cC_2(\cP)$ and $\cK_2(\cP)$}\label{The categories cC_2(cP) and cK_2(cP)}\

Let $\cC_2(\cP)$ be the category of two-periodic projective complexes of $\cA$, whose objects are of the form 
\begin{diagram}[midshaft]
{X=(X^1,X^0,d^1,d^0)=X^1} &\pile{\rTo^{d^1}\\ \lTo_{d^0}} &X^0,
\end{diagram}
where $X^j\in \cP$ and $d^j\in \Hom_{\cA}(X^j,X^{j+1})$, such that $d^{j+1}d^j=0$ for $j\in \bbZ_2$. From now on, the upper symbol $j$ denotes an element in $\bbZ_2$, unless we illustrate explicitly.

A morphism $f\in \Hom_{\cC_2(\cP)}(X,Y)$ is a pair $f=(f^1,f^0)$, where $f^j\in \Hom_{\cA}(X^j,Y^j)$, such that $d^j_Yf^j=f^{j+1}d^j_X$, that is, there is a commutative diagram
\begin{diagram}[midshaft,size=2em]
X^1 &\pile{\rTo^{d^1_X}\\ \lTo_{d^0_X}} &X^0\\
\dTo^{f^1} & &\dTo_{f^0}\\
Y^1 &\pile{\rTo^{d^1_Y}\\ \lTo_{d^0_Y}} &Y^0.
\end{diagram}

Two morphisms $f,g\in \Hom_{\cC_2(\cP)}(X,Y)$ are said to be homotopic, denoted by $f\sim g$, if there exists $h^j\in \Hom_{\cA}(X^j,Y^{j+1})$ such that $f^j-g^j=d^{j+1}_Yh^j+h^{j+1}d^j_X$. 
\begin{diagram}[midshaft,size=2em]
X^1 &\pile{\rTo^{d^1_X}\\ \lTo_{d^0_X}} &X^0\\
\dTo^{f^1-g^1}_{\ \ \ h^1} &\rdDashto \ldDashto&\dTo^{h^0\ \ \ }_{f^0-g^0}\\
Y^1 &\pile{\rTo^{d^1_Y}\\ \lTo_{d^0_Y}} &Y^0.
\end{diagram}

Let $\cK_2(\cP)$ be the stable category of two-periodic projective complexes of $\cA$, whose objects are the same as $\cC_2(\cP)$ and 
$$\Hom_{\cK_2(\cP)}(X,Y)=\Hom_{\cC_2(\cP)}(X,Y)/\Htp(X,Y),$$
where $\Htp(X,Y)=\{f\in \Hom_{\cC_2(\cP)}(X,Y)|f\sim 0\}$. For any $f\in \Hom_{\cC_2(\cP)}(X,Y)$, we still denote its image in the quotient $\Hom_{\cK_2(\cP)}(X,Y)$ by $f$ rather than $\overline{f}$.

The shift functor on $\cC_2(\cP)$ and $\cK_2(\cP)$ are denoted by
\begin{diagram}[midshaft]
(-)^*:{(X=X^1} &\pile{\rTo^{d^1}\\ \lTo_{d^0}} &X^0) &\rMapsto &{(X^*=X^0} &\pile{\rTo^{-d^0}\\ \lTo_{-d^1}} &X^1)
\end{diagram}
which is an involution.

An object $X$ in $\cC_2(\cP)$ or $\cK_2(\cP)$ is said to be contractible, if $1_X\sim 0$.

An object $X$ in $\cC_2(\cP)$ or $\cK_2(\cP)$ is said to be radical, if $\im d^j\subseteq \rad X^{j+1}$.

A morphism $f\in \Hom_{\cC_2(\cP)}(X,Y)$ is called an isomorphism, if it is an isomorphism in $\cC_2(\cP)$. In this case, we write $X\cong Y$. And $f$ is called a homotopy equivalence, if its homotopy class in $\Hom_{\cK_2(\cP)}(X,Y)$ is an isomorphism in $\cK_2(\cP)$. In this case, we write $X\simeq Y$.

\begin{lemma}[\cite{Bridgeland-2013}, Lemma 3.2]\label{contractible complex}
If $X$ is contractible, then there exist unique $P,Q\in \cP$ up to isomorphism such that $X\cong K_P\oplus K_Q^*$, where 
\begin{diagram}
{K_P=P} &\pile{\rTo[midshaft]^1\\ \lTo[midshaft]_0} &P, &{K_Q^*=Q} &\pile{\rTo[midshaft]^0\\ \lTo[midshaft]_1} &Q.
\end{diagram}
\end{lemma}

\begin{proposition}[\cite{Peng-Xiao-1997}, Proposition 7.1]
The category $\cC_2(\cP)$ is a Frobenius category whose projective-injective objects are contractible complexes, and its stable category is $\cK_2(\cP)$. Hence $\cK_2(\cP)$ is a triangulated category.
\end{proposition}
The following two lemmas are well known for bounded projective complexes, a proof can be founded in \cite[Lemma 1]{Jensen-Su-Zimmermann-2005}. We write a proof for the two-periodic case.

\begin{lemma}\label{decomposition}
Any object $X\in\cC_2(\cP)$ can be uniquely decomposed as $X\cong X_r\oplus X_c$ up to isomorphism such that $X_r$ is radical and $X_c$ is contractible.
\end{lemma}
\begin{proof}
For any $X=(X^1,X^0,d^1,d^0)\in\cC_2(\cP)$, since $d^j(\rad X^j)\subseteq \rad X^{j+1}$, we have a two-periodic complex $(\oline{X^1},\oline{X^0},\oline{d^1},\oline{d^0})$ of semisimple modules, where $\oline{X^j}=X^j/\rad X^j$, and $\oline{d^j}$ is induced by $d^j$. If both $\oline{d^j}=0$, then $X$ is radical. Otherwise, without loss of generality, we assume that $\oline{d^1}\not=0$, then $\im \oline{d^1}$ is a non-zero direct summand of the semisimple module $\oline{X^0}$. Let $\pi'^0:X'^0\rightarrow \im \oline{d^1}$ be the projective cover of $\im \oline{d^1}$, we claim that $X'^0$ is isomorphic to a direct summand of $X^0$. Indeed, consider the diagram
\begin{diagram}[midshaft,size=2em]
X^0 &\rDashto^{\varphi^0}   &X'^0\\
\dOnto^{\pi^0} & &\dOnto_{\pi'^0}\\
\oline{X^0} &\rOnto^{\oline{\pi^0}} &{\im \oline{d^1}},
\end{diagram}
where $\pi^0,\oline{\pi^0}$ are natural projections. Since $X^0$ is projective and $\pi'^0$ is surjective, there exists $\varphi^0:X^0\rightarrow {X^0}'$ such that above diagram commutes. Since $\pi'^0$ is a projective cover, the composition $\pi'^0\varphi^0=\oline{\pi^0}\pi^0$ is surjective implies that $\varphi^0$ is surjective. Hence $\varphi^0$ splits, since $X'^0$ is projective. This finishes the proof of our claim.

By similar argument, we have the following commutative diagram
\begin{diagram}[midshaft,size=2em]
X^1 &\rDashto^{\varphi^1}   &X'^0\\
\dOnto^{\pi^1} & &\dOnto_{\pi'^0}\\
\oline{X^1} &\rOnto^{\oline{d^1}} &{\im \oline{d^1}},
\end{diagram}
where $\varphi^1$ is a splitting surjective morphism, so $X'^0$ is also isomorphic to a direct summand of $X^1$. We assume that $X^1\cong X''^1\oplus X'^0$ and $X^0\cong X''^0\oplus X'^0$, then $d^1$ is transformed into 
$\begin{pmatrix}\begin{smallmatrix}
d''^1 &0\\
0    &1_{X'^0}
\end{smallmatrix}\end{pmatrix}$ and $d^0$ is transformed into 
$\begin{pmatrix}\begin{smallmatrix}
d''^0 &0\\
0 &0
\end{smallmatrix}\end{pmatrix}$ satisfying $d''^{j+1}d''^j=0$ such that  
\begin{diagram}[midshaft]
X\cong(X''^1 &\pile{\rTo^{d''^1}\\ \lTo_{d''^0}} &X''^0) \oplus (X'^0 &\pile{\rTo^{1_{X'^0}}\\ \lTo_0} &X'^0),
\end{diagram}
where $\im d''^1\subseteq \rad X''^0$. If $\im d''^0\subseteq \rad X''^1$, it is as desired. Otherwise, we may repeat above progress until the conclusion holds. 
\end{proof}

\begin{lemma}\label{isomorphism and homotopy equivalence}
Let $X,Y$ be radical complexes, then any morphism $f\in \Hom_{\cC_2(\cP)}(X,Y)$ is an isomorphism if and only if it is a homotopy equivalence.
\end{lemma}
\begin{proof}
If $f$ is an isomorphism, it is clear that $f$ is a homotopy equivalence.

Conversely, if $f$ is a homotopy equivalence, that is, there exists  $g:Y\rightarrow X$ such that $gf\sim 1_X$ and $fg\sim 1_Y$, that is, there exist morphisms $h^j,k^j$ in $\cA$ such that
\begin{align*}
g^j f^j-1_{X^j}&=d^{j+1}_X h^j+h^{j+1} d^j_X,\\
f^j g^j-1_{Y^j}&=d^{j+1}_Y k^j+k^{j+1} d^j_Y.
\end{align*}
Since $X,Y$ are radical complexes, that is, $\im d^j_X\subseteq \rad X^{j+1}, \im d^j_Y\subseteq \rad Y^{j+1}$, we have $\im (g^j f^j-1_{X^j})\subseteq \rad X^j,\im (f^j g^j-1_{Y^j})\subseteq \rad Y^j$. By \cite[Proposition 3.7(e)]{Assem-Simson-Skowronski-2006},
\begin{align*}
\im (g^jf^j-1_{X^j})+\im (g^jf^j)=X^j &\Rightarrow \im (g^jf^j)=X^j.\\
\im (f^jg^j-1_{Y^j})+\im (f^jg^j)=Y^j&\Rightarrow \im(f^jg^j)=Y^j.
\end{align*}
Hence the endomorphisms $g^jf^j\in \End_{\cA}(X^j), f^jg^j\in \End_{\cA}(Y^j)$ are isomorphisms. We assume that $g^jf^jk^j=k^jg^jf^j=1_{X^j},f^jg^jk'^j=k'^jf^jg^j=1_{Y^j}$, then
\begin{align*}
k^jg^j&=k^jg^jf^jg^jk'^j=g^jk'^j,\\
d_X^jk^jg^j=k^{j+1}g^{j+1}f^{j+1}d_X^jk^jg^j&=k^{j+1}d_X^jg^jf^jk^jg^j=k^{j+1}d_X^jg^j=k^{j+1}g^{j+1}d_Y^j,
\end{align*}
where the second equation implies that $(k^1g^1,k^0g^0)\in \Hom_{\cC_2(\cP)}(Y,X)$, and then in addition the first equation implies that $(k^1g^1,k^0g^0)$ is the inverse of $f$.
\end{proof}

\begin{corollary}\label{automorphism groups coincide}
For any radical complex $X$, the natural projection from $\End_{\cC_2(\cP)}(X)$ to $\End_{\cK_2(\cP)}(X)$ restricts to a surjective group homomorphism 
$$p:\Aut_{\cC_2(\cP)}(X)\twoheadrightarrow \Aut_{\cK_2(\cP)}(X)$$
with kernel $\{1_X+f|f\in \Htp(X,X)\}$ which is bijective to the vector space $\Htp(X,X)$.
\end{corollary}
\begin{proof}
By Lemma \ref{isomorphism and homotopy equivalence}, the surjectivity is clear. Moreover, for any $f\in \Htp(X,X)$, the morphism $1_X+f\sim 1_X$ is a homotopy equivalence, then it is a isomorphism and contained in the kernel $\Ker$. Conversely, for any $g\in \Ker$, we have $g\sim 1_X$, and so $g-1_X\sim 0$, that is, $g-1_X\in \Htp(X,X)$.
\end{proof}

\begin{definition}
A complex $X$ is said to be $*$-saturated, if $X_c^*\cong X_c$.
\end{definition}

The projective dimension vector pair $\ue$ of an object $X=(X^1,X^0,d^1,d^0)$ in $\cC_2(\cP)$ or $\cK_2(\cP)$ is defined to be $\ue=(e^1,e^0)\in \bbN I\times \bbN I$ such that $X^j\cong \bigoplus_{i\in I} e^j_iP_i$, where $e^j_iP_i$ is the direct sum of $e^j_i$-copies of $P_i$. 

\begin{proposition}\label{saturated isomorphism}
Let $X,Y$ be two $*$-saturated complexes having the same projective dimension vector pair, then $X$ is isomorphic to $Y$ if and only if $X$ is homotopy equivalent to $Y$.
\end{proposition}
\begin{proof}
If $X\cong Y$, then it is clear that $X\simeq Y$. 

Conversely, if $X\simeq Y$, by Lemma \ref{decomposition}, we assume that $X\cong X_r\oplus X_c$ and $Y\cong Y_r\oplus Y_c$, then $X_r\simeq Y_r$. By Lemma \ref{isomorphism and homotopy equivalence}, $X_r\cong Y_r$, in particular, $X_r^j\cong Y_r^j$, and so $X_c^j\cong Y_c^j$. Since $X,Y$ are $*$-saturated, by Lemma \ref{contractible complex}, $X_c\cong K_P\oplus K_P^*$ and $Y_c\cong K_Q\oplus K_Q^*$ for some $P,Q\in \cP$. Comparing their projective dimension vector pairs, we have $P\cong Q$, and so $X_c\cong Y_c$. Therefore, $X\cong Y$.
\end{proof}

For $*$-saturated complexes having fixed projective dimension vector pair, their isomorphism classes are the same as the homotopy equivalence classes. We remark that the conclusion fails without the $*$-saturated condition, for example, $K_P\simeq K_P^*$, but $K_P\ncong K_P^*$

\subsection{Derived invariance}\label{derived invariance}\

Let $\cD^b(A)$ be the bounded derived category of $\cA$ with the shift functor denoted by $[1]$. Then $\cK_2(\cP)$ is determined by $\cD^b(A)$, see \cite{Fu-2012,Stai-2018,Zhao-2014}. We sketch the main results in this subsection.

Since $\cA$ has enough projective objects and finite-global dimension, there exists a triangulated equivalence $\cD^b(A)\simeq \cK^b(\cP)$ by taking projective resolutions, where $\cK^b(\cP)$ is the bounded homotopy category of projective complexes. Let $\cK^b(\cP)/[2]$ be the orbit category of $\cK^b(\cP)$ with respect to the functor $[2]$, whose objects are the same as $\cK^b(\cP)$ and 
$$\Hom_{\cK^b(\cP)/[2]}(X,Y)=\bigoplus_{s\in \bbZ}\Hom_{\cK^b(\cP)}(X,Y[2s]).$$
There is a fully-faithful functor
\begin{align*}
\pi:\cK^b(\cP)/[2]&\rightarrow \cK_2(\cP)\\
(X^n,d^n)_{n\in \bbZ}&\mapsto (\oplus_{s\in \bbZ}X^{2s+1},\oplus_{s\in \bbZ}X^{2s},\oline{d^1},\oline{d^0}),
\end{align*}
where $\oline{d^1},\oline{d^0}$ are determined by $(d^n)_{n\in\bbZ}$ in a natural way, see \cite[Lemma 3.1]{Bridgeland-2013}. In the case $\cA$ is hereditary, $\pi$ is dense and then is an equivalence, see \cite[Section 7]{Peng-Xiao-1997}. But in general, $\pi$ is not dense or $\cK^b(\cP)/[2]$ is not a triangulated category.

Let $\cR_A$ be the triangulated hull of $\cK^b(\cP)/[2]$ defined by Keller in \cite{Keller-2005}, which is a triangulated category admitting a triangulated functor $\pi_A:\cK^b(\cP)\rightarrow \cR_A$ satisfying a universal property, see \cite[Subsection 2.3]{Fu-2012}. The relation between $\cR_A$ and $\cK_2(\cP)$ was studied by \cite{Fu-2012,Stai-2018,Zhao-2014}. We summarize the result in the following proposition.
\begin{proposition}
(a) If $B$ is another algebra which is derived equivalent to $A$, that is, $\cD^b(B)\simeq \cD^b(A)$, then there is a triangulated equivalence $\cR_B\simeq \cR_A$.\\
(b) There is a triangulated equivalence $\cR_A\simeq \cK_2(\cP)$.
\end{proposition}

By above proposition, if $\cD^b(A)\simeq \cD^b(B)$, then $\cK_2(\cP_A)\simeq \cK_2(\cP_B)$. Moreover, we want to obtain a standard equivalence in the sense of Rickard.

\begin{proposition}[\cite{Rickard-1991}]
Let $F:\cD^b(A)\rightarrow\cD^b(B)$ be a derived equivalence, then $F$ induces two derived equivalences 
\begin{align*}
&F_1:\cD^b(A^{op}\otimes_{\bbC}A)\rightarrow \cD^b(B^{op}\otimes_{\bbC}A),\\
&F_2:\cD^b(B^{op}\otimes_{\bbC}B)\rightarrow \cD^b(A^{op}\otimes_{\bbC}B).
\end{align*}
Let $\Delta=F_1(_AA_A)\in \cD^b(B^{op}\otimes_{\bbC}A), \Delta'=F_2(_BB_B)\in \cD^b(A^{op}\otimes_{\bbC}B)$, then the derived tensor functor 
$$-\otimes^L_B\Delta:\cD^b(B)\rightarrow \cD^b(A)$$
is an equivalence with the quasi-inverse
$$R\Hom_A(\Delta,-)\simeq -\otimes^L_A\Delta':D^b(A)\rightarrow D^b(B).$$
This equivalence is called a standard equivalence and $\Delta$ is called a two-sided tilting complex. Moreover, $\Delta\otimes^L_A\Delta'\cong R\Hom_A(\Delta,\Delta)\cong B.$
\end{proposition}

The standard equivalence $\cD^b(B)\xrightarrow{\simeq} \cD^b(A)$ can be translated to be an equivalence $\cK^b(\cP_B)\xrightarrow{\simeq} \cK^b(\cP_A)$ as follows. Since $\cD^b(B^{op}\otimes_{\bbC}A)\simeq \cK^b(\cP_{B^{op}\otimes_{\bbC}A})$, we can regard $\Delta=(\Delta^n,d^n_\Delta)_{n\in \bbZ}$ as an object of $\cK^b(\cP_{B^{op}\otimes_{\bbC}A})$ by taking projective resolutions. Then for any object $Y=(Y^n,d^n_Y)_{n\in\bbZ}\in \cK^b(\cP_B)$, we have $X=Y\otimes^L_B\Delta\in \cK^b(\cP_A)$, where $X^n=\prod_{s+t=n}Y^s\otimes_B\Delta^t\in \cP_A$, and $d_X^n|_{Y^s\otimes_B\Delta^t}=d^s_Y\otimes 1_{\Delta^t}+(-1)^s1_{Y^s}\otimes d_\Delta^t$.

We want to construct a triangulated equivalence $\cK_2(\cP_B)\rightarrow \cK_2(\cP_A)$.
Firstly, for any $\Theta\in \cK^b(\cP_{B^{op}\otimes_{\bbC}A})$, there is a functor 
\begin{align*}
-\boxtimes_B\Theta:\cC_2(\cP_B)&\rightarrow \cC_2(\cP_A)\\
Y=(Y^j,d^j_Y)_{j\in \bbZ_2}&\mapsto X=Y\boxtimes_B\Theta,
\end{align*}
where $X^j=\prod_{s+t\equiv j (\textrm{mod}\, 2)}Y^s\otimes_B\Theta^t$ and $d^j_X|_{Y^s\otimes_B\Theta^t}=d^s_Y\otimes 1_{\Theta^t}+(-1)^s1_{Y^s}\otimes d^t_{\Theta}$. Then it is routine to check that the functor is compatible with the homotopic relations, and so it induces a functor between stable categories
$$-\boxtimes_B\Theta: \cK_2(\cP_B)\rightarrow \cK_2(\cP_A).$$
Similarly, for any $\Theta'\in \cK^b(\cP_{A^{op}\otimes_{\bbC}B})$, there is a functor
$$-\boxtimes_A\Theta': \cK_2(\cP_A)\rightarrow \cK_2(\cP_B).$$
Taking $\Theta=\Delta$, we claim that $-\boxtimes_B\Delta: \cK_2(\cP_B)\rightarrow \cK_2(\cP_A)$ is an equivalence, which has been proven in \cite[Lemma 3.3]{Stai-2018} for the one-periodic case. We repeat the proof for the two-periodic case.
\begin{proposition}
The functor $-\boxtimes_B\Delta: \cK_2(\cP_B)\rightarrow \cK_2(\cP_A)$ is an equivalence.
\end{proposition}
\begin{proof}
By Rickard's theorem, the quasi-inverse of $-\otimes^L_B\Delta:\cD^b(B)\rightarrow \cD^b(A)$ is given by $R\Hom_A(\Delta,-)\simeq-\otimes^L_A\Delta':\cD^b(A)\rightarrow \cD^b(B)$, and we have 
$$\Delta\otimes^L_A\Delta'\cong R\Hom_A(\Delta,\Delta)\cong B.$$
Since $\cD^b(B^{op}\otimes_{\bbC}B)\simeq \cK^b(\cP_{B^{op}\otimes_{\bbC}B})$, we regard $_BB_B\in \cD^b(B^{op}\otimes_{\bbC}B)$ as the stalk complex, and then regard $_BB_B\in \cK^b(\cP_{B^{op}\otimes_{\bbC}B})$ by taking projective resolution, then we have the functor $-\boxtimes_BB:\cK_2(\cP_B)\rightarrow \cK_2(\cP_B)$ such that for any $Y\in \cK_2(\cP_B)$, there are natural isomorphisms
$$(Y\boxtimes_B\Delta)\boxtimes_A\Delta'\cong Y\boxtimes_B(\Delta\otimes^L_A\Delta')\cong Y\boxtimes_BB\cong Y.$$
Similarly, for any $X\in \cK_2(\cP_A)$, there is a natural isomorphism 
$$(X\boxtimes_A\Delta')\boxtimes_B\Delta\cong X.$$
Hence, the functor $-\boxtimes_B\Delta$ is an equivalence with the quasi-inverse $-\boxtimes_A\Delta'$.
\end{proof}

\section {Moduli varieties of two-periodic projective complexes}

With the same notations in Section \ref{Categories of two-periodic projective complexes}. In this section, we construct the moduli spaces for categories $\cA,\cC_2(\cP)$ and $\cK_2(\cP)$.

\subsection{Moduli variety of $\cA$}\

For any $\alpha\in \bbN I$, we define an affine space 
$$\rE_\alpha(Q)=\bigoplus_{h\in H}\Hom_{\bbC}(\bbC^{\alpha_{s(h)}},\bbC^{\alpha_{t(h)}}).$$
The algebraic group 
$$\rG_\alpha(Q)=\prod_{i\in I}\textrm{GL}_{\alpha_i}(\bbC)$$
acts on $\rE_\alpha(Q)$ by
$$g.x=(g_{t(h)}x_h g_{s(h)}^{-1})_{h\in H}$$ 
for $g=(g_i)_{i\in I}\in \rG_{\alpha}(Q), x=(x_h)_{h\in H}\in \rE_\alpha(Q)$. 

Recall that $J$ is the admissible ideal of $\bbC Q$ generated by a set of relations $R$. A relation is a linear combination of the form $\sum_{j=1}^r\lambda_j\omega_j$, where $\lambda_j\in \bbC$ and $\omega_j$ are paths of length at least two having the same source and target. Given a point $x=(x_h)_{h\in H}\in \rE_{\alpha}(Q)$ and a path $\omega=h_1...h_l$ in $Q$, we set $x_\omega=x_{h_1}...x_{h_l}$. Then $x$ is said to satisfy the relation $\sum_{j=1}^r\lambda_j\omega_j$, if $\sum_{j=1}^r\lambda_jx_{\omega_j}=0$.

We define a $\rG_\alpha(Q)$-invariant closed subvariety 
$$\rE_\alpha(Q,R)\subset\rE_\alpha(Q)$$
consisting of points satisfying all relations in $R$. Then the quotient stack $$[\rE_\alpha(Q,R)/\rG_\alpha(Q)]$$ parametrizes the isomorphism classes of $(Q,R)$-representations of dimension vector $\alpha$, that is, there exists a bijection between the set of $\rG_\alpha$-orbits in $\rE_\alpha(Q,R)$ and the set of isomorphism classes of $(Q,R)$-representations of dimension vector $\alpha$, denoted by $\cO_x\mapsto[M(x)]$. Moreover, the stabilizer of $x$ is isomorphic to the automorphism group of $M(x)$ in $\cA$ which is special and affine, since it consists of invertible elements in the endomorphism algebra $\End_{\cA}(M(x))$, see Subsection \ref{Relative Grothendieck group of stacks}.

\subsection{Moduli variety of $\cC_2(\cP)$}\label{moduli C2(P)}\

Firstly, we construct the moduli variety for $\cC_2(\cA)$. For any $\ualpha=(\alpha^1,\alpha^0)\in \bbN I\times \bbN I$, we define a closed subvariety 
$$\rC_2(A,\ualpha)\subset\rE_{\alpha^1}(Q,R)\times \rE_{\alpha^0}(Q,R)\times \bigoplus_{i\in I}\Hom_{\bbC}(\bbC^{\alpha^1_i},\bbC^{\alpha^0_i})\times \bigoplus_{i\in I}\Hom_{\bbC}(\bbC^{\alpha^0_i},\bbC^{\alpha^1_i})$$
consisting of $(x^1,x^0,d^1,d^0)=((x^1_h)_{h\in H},(x^0_h)_{h\in H},(d^1_i)_{i\in I},(d^0_i)_{i\in I})$ satisfying 
\begin{align*}
d^j_{t(h)} x^j_h=x^{j+1}_h d^j_{s(h)}, \ d^{j+1}_i d^j_i=0
\end{align*}
for $i\in I,h\in H,j\in \bbZ_2$. Note that the first condition implies that $d^j$ can be regarded as a morphism in $\Hom_{\cA}(M(x^j),M(x^{j+1}))$, and the second condition implies that $d^{j+1}d^j=0$. The algebraic group 
$$\rG_{\ualpha}=\rG_{\alpha^1}(Q)\times \rG_{\alpha^0}(Q)$$
acts on $\rC_2(A,\ualpha)$ by
\begin{align*}
(g^1,g^0).(x^1,x^0,d^1,d^0)=(g^1.x^1,g^0.x^0,(g^0_i d^1_i (g^1_i)^{-1})_{i\in I},(g^1_i d^0_i (g^0_i)^{-1})_{i\in I})
\end{align*}
for $(g^1,g^0)\in \rG_{\ualpha},(x^1,x^0,d^1,d^0)\in \rC_2(A,\alpha)$. Then the quotient stack 
$$[\rC_2(A,\ualpha)/\rG_{\ualpha}]$$
parametrizes the isomorphism classes of two-periodic complexes of dimension vector pair $\ualpha$, that is, there is a bijection between the set of $\rG_{\ualpha}$-orbits in $\rC_2(A,\ualpha)$ and the set of isomorphism classes of two periodic complexes of dimension vector pair $\ualpha$, denoted by $\cO_{(x^1,x^0,d^1,d^0)}\mapsto [(M(x^1),M(x^0),d^1,d^0)]$. Moreover, the stabilizer of $(x^1,x^0,d^1,d^0)$ is isomorphic to the automorphism group of $(M(x^1),M(x^0),d^1,d^0)$ in $\cC_2(\cA)$ which is special and affine, since it consists of invertible elements in the endomorphism algebra $\End_{\cC_2(\cA)}(M(x^1),M(x^0),d^1,d^0)$, see Subsection \ref{Relative Grothendieck group of stacks}.

Next, we construct the moduli variety for $\cC_2(\cP)$. For any $\ue=(e^1,e^0)\in \bbN I\times \bbN I$, we set $\ualpha(\ue)=(\alpha^1(\ue),\alpha^0(\ue))\in \bbN I\times \bbN I$, where $\alpha^j(\ue)$ is the dimension vector of $\bigoplus_{i\in I}e^j_iP_i$, and set $\rG_{\ue}=\rG_{\ualpha(\ue)}$. Then there is a $\rG_\ue$-equivariant natural projection
\begin{align*} 
\pi:\rC_2(A,\ualpha(\ue))&\rightarrow \rE_{\alpha^1(\ue)}(Q,R)\times \rE_{\alpha^0(\ue)}(Q,R)\\
(x^1,x^0,d^1,d^0)&\mapsto (x^1,x^0).
\end{align*}
We define a $\rG_{\ue}$-invariant locally closed subset
$$\rP_2(A,\ue)=\pi^{-1}(\cO_{p_{e^1}},\cO_{p_{e^0}})\subset \rC_2(A,\ualpha(\ue)),$$
 where $\cO_{p_{e^j}}$ is the $\rG_{\alpha^j(\ue)}$-orbit in $\rE_{\alpha^j(\ue)}(Q,R)$ corresponding to the isomorphism class $[\bigoplus_{i\in I}e^j_iP_i]$.
Then the quotient stack 
$$[\rP_2(A,\ue)/\rG_{\ue}]$$ 
parametrizes the isomorphism classes of two-periodic projective complexes having projective dimension vector pair $\ue$.

Finally, we construct the subvariety of radical complexes in $\cC_2(\cP)$ in preparation for next subsection. Recall a basic property of radical of modules.
\begin{proposition}[\cite{Assem-Simson-Skowronski-2006}, Proposition 3.7(a)]\label{radical}
Let $M$ be a right $A$-module, then an element $m\in M$ belongs to $\rad M$ if and only if $f(m)=0$ for any $f\in \Hom_A(M,S)$ and any simple right $A$-module $S$.
\end{proposition} 
Note that for any projective module $P$ of dimension vector $\alpha\in \bbN I$ and any simple representation $S_i$ corresponding to $i\in I$, we have $\Hom_A(P,S_i)=\Hom_{\bbC}(\bbC^{\alpha_i},\bbC)$. For any $\ue\in \bbN I\times \bbN I$, we have set $\ualpha(\ue)=(\alpha^1(\ue),\alpha^0(\ue))\in \bbN I\times \bbN I$ such that $\alpha^j(\ue)$ is the dimension vector of $\bigoplus_{i\in I}e^j_iP_i$,  For any $i\in I, j\in \bbZ_2$, we fix a basis $\{f^j_{i,k}\mid k=1,...,\alpha^j(\ue)_i\}$ of $\Hom_\bbC(\bbC^{\alpha^j(\ue)_i},\bbC)$, then define a closed subset
$$\rP_2^{\rad}(A,\ue)\subset \rP_2(A,\ue)$$ 
consisting of points $(x^1,x^0,d^1,d^0)$ satisfying 
$$f^j_{i,k}d_i^{j+1}=0,$$
for $i\in I,j\in \bbZ_2, k=1,...,\alpha^j(\ue)_i$. By Proposition \ref{radical}, $\rP_2^{\rad}(A,\ue)$ parametrizes radical complexes in $\cC_2(\cP)$ having projective dimension vector pair $\ue$. We claim that $\rP_2^{\rad}(A,\ue)$ is $\rG_{\ue}$-invariant. Indeed, for $(x^1,x^0,d^1,d^0)\in \rP_2^{\rad}(A,\ue)$ and $(g^1,g^0)\in \rG_{\ue}$, since $\{f^j_{i,k}\mid k=1,...,\alpha^j(\ue)_i\}$ is a basis of $\Hom_\bbC(\bbC^{\alpha^j(\ue)_i},\bbC)$, we have $fd^{j+1}_i=0$ for any $f\in \Hom_\bbC(\bbC^{\alpha^j(\ue)_i},\bbC)$. In particular, for $f^j_{i,k}g^j_i\in \Hom_\bbC(\bbC^{\alpha^j(\ue)_i},\bbC)$, we have $f^j_{i,k}g^j_id^{j+1}_i=0$ and so $f^j_{i,k}g^j_id^{j+1}_i(g^{j+1}_i)^{-1}=0$, as desired. Then the quotient stack 
$$[\rP_2^{\rad}(A,\ue)/\rG_{\ue}]$$
parametrizes the isomorphism classes of radical complexes in $\cC_2(\cP)$ having projective dimension vector pair $\ue$.

\subsection{Moduli space of $\cK_2(\cP)$}\label{Moduli K_2(P)}\

Firstly, we construct the moduli variety of $*$-saturated complexes in $\cC_2(\cP)$. There is a partial order on $\bbN I\times \bbN I$ defined by $\ue'\leqslant \ue$ if and only if there exists a contractible complex $K_P\oplus K^*_P$ having projective dimension vector pair $\ue-\ue'$. 

For any $\ue\in \bbN I\times \bbN I$, we define a finite disjoint union 
$$\rP_2^*(A,\ue)=\bigsqcup_{\ue'\leqslant \ue}\rP_2^{\rad}(A,\ue')\oplus \cO_{(p_{\frac{e^1-e'^1}{2}},p_{\frac{e^0-e'^0}{2}},1,0)}\oplus \cO_{(p_{\frac{e^1-e'^1}{2}},p_{\frac{e^0-e'^0}{2}},0,1)},$$
where each component $\rP_2^{\rad}(A,\ue')\oplus \cO_{(p_{\frac{e^1-e'^1}{2}},p_{\frac{e^0-e'^0}{2}},1,0)}\oplus \cO_{(p_{\frac{e^1-e'^1}{2}},p_{\frac{e^0-e'^0}{2}},0,1)}$ is the locally closed subset of $\rP_2(A,\ue)$ consisting of $(x^1,x^0,d^1,d^0)$ satisfying 
$$(M(x^1),M(x^0),d^1,d^0)\cong (M(x'^1),M(x'^0),d'^1,d'^0)\oplus K_P\oplus K_P^*,$$
for some $(x'^1,x'^0,d'^1,d'^0)\in \rP_2^{\rad}(A,\ue')$ and $P=M(p_{\frac{e^1-e'^1}{2}})=M(p_{\frac{e^0-e'^0}{2}})$. We claim that $\rP_2^*(A,\ue)\subset\rP_2(A,\ue)$ is constructible. Indeed, each component is constructible as the image of the following algebraic morphism
\begin{align*}
&\rG_{\ue}\times \rP_2^{\rad}(A,\ue')\times \cO_{(p_{\frac{e^1-e'^1}{2}},p_{\frac{e^0-e'^0}{2}},1,0)}\times \cO_{(p_{\frac{e^1-e'^1}{2}},p_{\frac{e^0-e'^0}{2}},0,1)}\rightarrow \rP_2(A,\ue)\\
&((g^1,g^0),(x'^1,x'^0,d'^1,d'^0),(p^1_{\frac{\ue-\ue'}{2}},p^0_{\frac{\ue-\ue'}{2}},1,0),(p^1_{\frac{\ue-\ue'}{2}},p^0_{\frac{\ue-\ue'}{2}},0,1))\mapsto\\
&(g^1,g^0).(\begin{pmatrix}\begin{smallmatrix}
x'^1 & &\\
       &p_{\frac{e^1-e'^1}{2}} &\\
       &  &p_{\frac{e^1-e'^1}{2}}
\end{smallmatrix}\end{pmatrix},
\begin{pmatrix}\begin{smallmatrix}
x'^0 & &\\
       &p_{\frac{e^0-e'^0}{2}} &\\
       &  &p_{\frac{e^0-e'^0}{2}}
\end{smallmatrix}\end{pmatrix},
\begin{pmatrix}\begin{smallmatrix}
d'^1 & &\\
       &1 &\\
       &  &0
\end{smallmatrix}\end{pmatrix},
\begin{pmatrix}\begin{smallmatrix}
d'^0 & &\\
       &0 &\\
       &  &1
\end{smallmatrix}\end{pmatrix}),
\end{align*}
and so is their finite disjoint union $\rP_2^*(A,\ue)$. By Proposition \ref{saturated isomorphism}, the quotient stack 
$$[\rP_2^*(A,\ue)/\rG_{\ue}]$$
parametrizes the isomorphism classes of $*$-saturated complexes in $\cC_2(\cP)$ having projective dimension vector pair $\ue$, as well as the homotopy equivalence classes in $\cK_2(\cP)$.

Let $K_0$ be the Grothendieck group of the triangulated category $\cK_2(\cP)$, then the canonical surjection $\cK_2(\cP)\rightarrow K_0$ induces a surjection 
$\udim:\textrm{pdvp}^*\rightarrow K_0$, where $\textrm{pdvp}^*\subset \bbN I\times\bbN I$ is the subset consisting of projective dimension vector pairs of $*$-saturated complexes.

For any $\bd\in K_0$, the set $\udim^{-1}(\bd)$ is a directed partial order set, that is, for any $\ue',\ue''\in \udim^{-1}(\bd)$, there exists $\ue\in \udim^{-1}(\bd)$ such that $\ue'\leqslant \ue,\ue''\leqslant \ue$. For any $\ue,\ue'\in\udim^{-1}(\bd)$ satisfying $\ue'\leqslant \ue$, there is a constructible map 
\begin{align*}
&t_{\ue'\ue}:\rP_2^*(A,\ue')\rightarrow \rP_2^*(A,\ue)\\
&(x^1,x^0,d^1,d^0)\mapsto\\
&(\begin{pmatrix}\begin{smallmatrix}
x^1 & &\\
       &p_{\frac{e^1-e'^1}{2}} &\\
       &  &p_{\frac{e^1-e'^1}{2}}
\end{smallmatrix}\end{pmatrix},
\begin{pmatrix}\begin{smallmatrix}
x^0 & &\\
       &p_{\frac{e^0-e'^0}{2}} &\\
       &  &p_{\frac{e^0-e'^0}{2}}
\end{smallmatrix}\end{pmatrix},
\begin{pmatrix}\begin{smallmatrix}
d^1 & &\\
       &1 &\\
       &  &0
\end{smallmatrix}\end{pmatrix},
\begin{pmatrix}\begin{smallmatrix}
d^0 & &\\
       &0 &\\
       &  &1
\end{smallmatrix}\end{pmatrix}),
\end{align*}
corresponding to the morphism $$(M(x^1),M(x^0),d^1,d^0)\mapsto (M(x^1),M(x^0),d^1,d^0)\oplus K_P\oplus K_P^*,$$ 
where $P=M(p_{\frac{e^1-e'^1}{2}})=M(p_{\frac{e^0-e'^0}{2}})$, such that $\{\rP^*_2(A,\ue)\}_{\ue\in \udim^{-1}(\bd)}$ is a direct system. We define the ind-limit
$$\rP_2(A,\bd)=``\varinjlim"_{\ue\in \udim^{-1}(\bd)}\rP^*_2(A,\ue).$$
For any $\ue\in \udim^{-1}(\bd)$, the natural morphism $t_{\ue}:\rP_2^*(A,\ue)\rightarrow \rP_2(A,\bd)$ corresponds to the map which send a complex to its homotopy equivalence class
$$X\mapsto \tilde{X}=\{X_r\oplus K_P\oplus K_P^*|P\in \cP\}.$$

\begin{remark}\label{remark2}
(a) The definition of ind-lim $``\varinjlim"$ is given in \cite[Section 2.6]{Kashiwara-Schapira-2006}. From now on, we write $\varinjlim$ instead of $``\varinjlim"$ for convenience.\\
(b) The limit stick two homotopy equivalent $*$-saturated complexes having different projective dimension vector pairs. More precisely, let $X_r$ be a radical complex, then $X_r\oplus K_P\oplus K^*_P$ and $X_r\oplus K_Q\oplus K^*_Q$ are homotopy equivalent, but not isomorphic. Notice that under some suitable maps $t_{\ue'\ue}$ respectively, they have the same image $X_r\oplus K_P\oplus K^*_P\oplus K_Q\oplus K^*_Q$, which implies that they can be regarded as the same element in the limit.
\end{remark}

The algebraic group action is compatible with the limit. Indeed, for any $\ue,\ue'\in\udim^{-1}(\bd)$ satisfying $\ue'\leqslant \ue$, there is a morphism between algebraic groups 
\begin{align*}
f_{\ue'\ue}:\rG_{\ue'}&\rightarrow \rG_{\ue}\\
(g^1,g^0)&\mapsto (\begin{pmatrix}\begin{smallmatrix}
g^1 & &\\
 &1 &\\
 &  &1
\end{smallmatrix}\end{pmatrix},\begin{pmatrix}\begin{smallmatrix}
g^0 & &\\
 &1 &\\
 & &1
\end{smallmatrix}\end{pmatrix}),
\end{align*}
such that $\{\rG_{\ue}\}_{\ue\in \udim^{-1}(\bd)}$ is a direct system. Its ind-limit 
$$\rG_{\bd}=\varinjlim_{\ue\in \udim^{-1}(\bd)}\rG_{\ue}$$
acts on $\rP_2(A,\bd)$ as follows, for any $(g^1,g^0)\in \rG_{\ue'},(x^1,x^0,d^1,d^0)\in \rP_2^*(A,\ue'')$, let $f_{\ue'}:\rG_{\ue'}\rightarrow \rG_{\bd},t_{\ue''}:\rP_2^*(A,\ue'')\rightarrow \rP_2(A,\bd),t_{\ue}:\rP_2^*(A,\ue)\rightarrow \rP_2(A,\bd)$ be natural morphisms, where $\ue\in \udim^{-1}(\bd)$ satisfying $\ue'\leqslant \ue,\ue''\leqslant \ue$, then 
$$f_{\ue'}(g^1,g^0).t_{\ue''}(x^1,x^0,d^1,d^0)=t_{\ue}(f_{\ue'\ue}(g^1,g^0).t_{\ue''\ue}(x^1,x^0,d^1,d^0)).$$

\begin{remark}\label{ind-constructible}
The ind-limit $[\rP_2(A,\bd)/\rG_{\bd}]$ is an ind-constructible stack in the sense of \cite[Section 3.1, 3.2, 4.2]{Kontsevich-Soibelman-2008}, that is, it is a countable union of non-intersecting constructible sets, and each constructible set is endowed with an affine algebraic group action. Indeed, there is a bijection $\rP_2(A,\bd)\rightarrow \bigsqcup_{\ue\in \udim^{-1}(\bd)}\rP_2^{\rad}(A,\ue)$ corresponding to the morphism $X\mapsto X_r$, and each $\rP_2^{\rad}(A,\ue)$ has an algebraic group $\rG_{\ue}$-action.
\end{remark}

By Proposition \ref{saturated isomorphism}, Remark \ref{remark2}, we have the following corollary.
\begin{corollary}\label{homotopy class and orbit}
The quotient 
$$[\rP_2(A,\bd)/\rG_{\bd}]$$
parametrizes the homotopy equivalence classes of complexes in $\cK_2(\cP)$ whose images in the Grothendieck group are $\bd$, that is, there is a bijection between the set of $\rG_{\bd}$-orbits in $\rP_2(A,\bd)$ and the set of homotopy equivalence classes of complexes in $\cK_2(\cP)$ having images $\bd$ in the Grothendieck group $K_0$.
\end{corollary}

\subsection{Derived invariance}\label{Derived invariance 2}\

By Subsection \ref{derived invariance}, we know that if $\cD^b(B)\simeq \cD^b(A)$, then there is a standard equivalence $-\boxtimes^L_B\Delta:\cK_2(\cP_B)\rightarrow \cK_2(\cP_A)$. It induces an isomorphism between Grothendieck groups $K_0(\cK_2(\cP_B))\xrightarrow{\cong} K_0(\cK_2(\cP_A))$. For any $\bd_B\in K_0(\cK_2(\cP_B))$, letting $\bd_A\in K_0(\cK_2(\cP_A))$ be its image under the isomorphism, we construct a morphism $\rP_2(B,\bd_B)\rightarrow \rP_2(A,\bd_A)$ as follows.

For any $\ue_B\in\udim^{-1}(\bd_B)$, by the definition of the functor $-\boxtimes_B^L\Delta$, there exists some $\ue_A\in \bbN I\times \bbN I$ and a constructible map 
$$\varphi_{\ue_B}:\rP^*_2(B,\ue_B)\rightarrow \rP_2(A,\ue_A)$$
corresponding to the morphism $Y\mapsto Y\boxtimes^L_B\Delta$.
By Lemma \ref{decomposition}, any complex can be decomposed as the direct sum of a radical complex and a contractible complex, and so we may divide $\rP_2(A,\ue_A)$ into a finite disjoint union
\begin{align*}
\rP_2(A,\ue_A)=\bigsqcup_{\ue_1,\ue_2\leqslant \ue_A}\rP_2^*(A,\ue_A-\ue_1-\ue_2)\oplus \cO_{(p_{e^1_1},p_{e^0_1},1,0)}\oplus \cO_{(p_{e^1_2},p_{e^0_2},0,1)},
\end{align*}
where each component $\rP_2(\ue_A,\ue_1,\ue_2)=\rP_2^*(A,\ue_A-\ue_1-\ue_2)\oplus \cO_{(p_{e^1_1},p_{e^0_1},1,0)}\oplus \cO_{(p_{e^1_2},p_{e^0_2},0,1)}$ consists of points corresponding to complexes isomorphic to $X\oplus K_P\oplus K_Q^*$, where $X$ is $*$-saturated and $K_P, K_Q$ are contractible having projective dimension pairs $\ue_1,\ue_2$, respectively. Then for each component, there is a constructible map 
$$f_{\ue_A,\ue_1,\ue_2}:\rP_2(\ue_A,\ue_1,\ue_2)\rightarrow \bigsqcup_{\ue'_A\in \udim^{-1}(\bd_A)}\rP_2^*(A,\ue'_A)$$
corresponding to the morphism $X\oplus K_P\oplus K^*_Q\mapsto X\oplus K_P\oplus K^*_P\oplus K_Q\oplus K^*_Q$, and so we obtain constructible maps
\begin{align*}
f=(f_{\ue_A,\ue_1,\ue_2})&:\rP_2(A,\ue_A)\rightarrow \bigsqcup_{\ue_A'\in \udim^{-1}(\bd_A)}\rP_2^*(A,\ue_A'),\\
\Phi_{\ue_B}=f\varphi_{\ue_B}&:\rP_2^*(B,\ue_B)\rightarrow \bigsqcup_{\ue_A'\in \udim^{-1}(\bd_A)}\rP_2^*(A,\ue_A').
\end{align*}
It is routine to check that $\{\Phi_{\ue_B}\}_{\ue_B\in\udim^{-1}(\bd_B)}$ are compatible with the limit. Hence they induce a morphism $\Phi:\rP_2(B,\bd_B)\rightarrow \rP_2(A,\bd_A)$. Similarly, by using of $-\boxtimes_A^L\Delta'$ the quasi-inverse of $-\boxtimes_B^L\Delta$, we can construct a morphism $\Psi:\rP_2(A,\bd_A)\rightarrow \rP_2(B,\bd_B)$ which induces an isomorphism
$$[\rP_2(A,\bd_A)/\rG_{\bd_A}]\cong [\rP_2(B,\bd_B)/\rG_{\bd_B}].$$

\section{Motivic form of Bridgeland's Hall algebra}

In this section, we define the motivic form of Bridgeland's Hall algebra \cite{Bridgeland-2013}. As far as we know, there is no study involving motivic Bridgeland's Hall algebra before. We remark that motivic Hall algebra for abelian category is defined by Joyce in \cite{Joyce-2007}, and we also refer \cite{Bridgeland-2012,Nagao-2013,Bridgeland-2017,Behrend-Ronagh-2019} which involve motivic Hall algebra for details.

\subsection{Moduli stacks of objects in $\cC_2(\cA)$}\label{Moduli stacks of objects}\

In Subsection \ref{moduli C2(P)}, we define an affine variety $\rC_2(A,\ualpha)$ together with an algebraic group $\rG_{\ualpha}$-action for any $\ualpha\in \bbN I\times \bbN I$. We denote by $\cM_{\ualpha}=[\rC_2(A,\ualpha)/\rG_{\ualpha}]$ the corresponding quotient stack and 
$$\cM=\bigsqcup_{\ualpha \in \bbN I\times \bbN I}\cM_{\ualpha},$$
then $\cM$ is an Artin stack with affine stabilizers which is locally of finite type over $\bbC$.

Let $\cM(\bbC)$ be the set of $\bbC$-valued points (also called geometric points) of $\cM$, that is, it is the set of $2$-isomorphism classes of $1$-morphisms $\Spec\bbC\rightarrow \cM$. We denote by
$$[X]\mapsto \delta_{[X]}=[\Spec\bbC\xrightarrow{i_{[X]}} \cM]$$
the bijection between the set of isomorphism classes of objects in $\cC_2(\cA)$ and $\cM(\bbC)$.

\subsection{Moduli stacks of filtrations in $\cC_2(\cA)$}\label{Moduli stacks of filtrations}\

For any $\ualpha,\ualpha'\in \bbN I\times \bbN I$, we define a closed subvariety 
$$\textbf{Hom}(\ualpha, \ualpha')\subset\rC_2(A,\ualpha)\times \rC_2(A,\ualpha')\times \bigoplus_{i\in I}\Hom_{\bbC}(\bbC^{\alpha^1_i},\bbC^{\alpha'^1_i})\times \bigoplus_{i\in I}\Hom_{\bbC}(\bbC^{\alpha^0_i},\bbC^{\alpha'^0_i})$$
consisting of $((x^1,x^0,d^1,d^0),(x'^1,x'^0,d'^1,d'^0),(f^1_i)_{i\in I},(f^0_i)_{i\in I})$ satisfying 
\begin{align*}
f^j_{t(h)} x^j_h=x'^j_h f^j_{s(h)},\ d'^j_if^j_i=f^{j+1}_id^j_i
\end{align*}
for any $i\in I,h\in H,j\in \bbZ_2$. Note that the first condition implies that $f^j$ can be regarded as a morphism in $\Hom_{\cA}(M(x^j),M(x'^j))$, and then the second condition implies that $(f^1,f^0)$ can be regarded as a morphism form $(M(x^1),M(x^0),d^1,d^0)$ to $(M(x'^1),M(x'^0),d'^1,d'^0)$ in $\cC_2(\cA)$.

For any $\ualpha=\ualpha'+\ualpha''\in \bbN I\times \bbN I$, we define a constructible subset 
$$\rm{W}_{\ualpha'\ualpha''}^{\ualpha}\subset\textbf{Hom}(\ualpha'', \ualpha)\times \textbf{Hom}(\ualpha, \ualpha')$$ 
consisting of 
$$(((x''^1,x''^0,d''^1,d''^0),(x^1,x^0,d^1,d^0),f''^1,f''^0),((x^1,x^0,d^1,d^0),(x'^1,x'^0,d'^1,d'^0),f'^1,f'^0))$$
such that 
$$
0\rightarrow X''\xrightarrow{(f''^1,f''^0)}X\xrightarrow{(f'^1,f'^0)}X'\rightarrow 0$$
is a short exact sequence in $\cC_2(\cA)$, where $X=(M(x^1),M(x^0),d^1,d^0)$ and $X',X''$ are similar. We simply denote elements in $\rm{W}^{\ualpha}_{\ualpha'\ualpha''}$ by 
$$((x''^1,x''^0,d''^1,d''^0),(x^1,x^0,d^1,d^0),(x'^1,x'^0,d'^1,d'^0),(f''^1,f''^0),(f'^1,f'^0)).$$
The algebraic group $\rG_{\ualpha'}\times \rG_{\ualpha''}$ acts on $\rm{W}^{\ualpha}_{\ualpha'\ualpha''}$ by
\begin{align*}
&(g',g'').((x''^1,x''^0,d''^1,d''^0),(x^1,x^0,d^1,d^0),(x'^1,x'^0,d'^1,d'^0),(f''^1,f''^0),(f'^1,f'^0))\\
=&(g''.(x''^1,x''^0,d''^1,d''^0),(x^1,x^0,d^1,d^0),g'.(x'^1,x'^0,d'^1,d'^0),\\
&(f''^1(g''^1)^{-1},f''^0(g''^0)^{-1}),(g'^1f'^1,g'^0f'^0)),
\end{align*}
where $(g',g'')=((g'^1,g'^0),(g''^1,g''^0))\in \rG_{\ualpha'}\times \rG_{\ualpha''}$, that is, there is a commutative diagram about short exact sequences
\begin{diagram}[midshaft,size=2em]
0 &\rTo &X'' &\rTo^{(f''^1,f''^0)} &X &\rTo^{(f'^1,f'^0)} &X'&\rTo  &0\\
 & &\dTo^{g'} & &\vEq & &\dTo_{g''}\\
0 &\rTo &g''.X'' &\rTo^{(f''^1(g''^1)^{-1},f''^0(g''^0)^{-1})} &X &\rTo^{(g'^1f'^1,g'^0f'^0)} &g'.X'&\rTo  &0.
\end{diagram}
Note that the action is free. We denote by ${\cM^{(2)}}^{\ualpha}_{\ualpha'\ualpha''}=[W^{\ualpha}_{\ualpha'\ualpha''}/\rG_{\ualpha'}\times \rG_{\ualpha''}]$ the quotient stack and 
$$\cM^{(2)}=\bigsqcup_{\ualpha=\ualpha'+\ualpha''}{\cM^{(2)}}^{\ualpha}_{\ualpha'\ualpha''}$$
which parametrizes filtrations of obejcts in $\cC_2(\cA)$ via 
\begin{align*}
&((x''^1,x''^0,d''^1,d''^0),(x^1,x^0,d^1,d^0),(x'^1,x'^0,d'^1,d'^0),f''^1,f''^0,f'^1,f'^0)\\
\mapsto&(\im (f''^1,f''^0)\subset X).
\end{align*}

Consider the diagram of morphisms of stacks
\begin{diagram}[midshaft,size=2em]
\cM^{(2)} &\rTo^b &\cM\\
\dTo^{(a_1,a_2)}\\
\cM\times \cM,
\end{diagram}
where the morphism $(a_1,a_2)$ corresponds to the map $(X''\subset X)\mapsto (X/X'',X'')$ and the morphism $b$ corresponds to the map $(X''\subset X)\mapsto X$.

\begin{lemma}\label{convolution diagram}
(a) The morphism $b$ is representable and proper.\\
(b) The morphism $(a_1,a_2)$ is of finite type.
\end{lemma}
\begin{proof}
(a) Consider the restriction $b|_{{\cM^{(2)}}^{\ualpha}_{\ualpha'\ualpha''}}:{\cM^{(2)}}^{\ualpha}_{\ualpha'\ualpha''}\rightarrow \cM_{\ualpha}$,
its fiber at a $\bbC$-valued point $\delta_{[X]}$, where $X=(M(x^1),M(x^0),d^1,d^0)$, is represented by a closed subvariety of a product of Grassmannians $\prod_{i\in I}\textrm{Gr}(\alpha^1_i,\alpha''^1_i)\times \textrm{Gr}(\alpha^0_i,\alpha''^0_i)$ which is projective. More precisely, the closed subvariety consisting of a family of subspaces $V^j_i\subset \bbC^{\alpha^j_i}$ satisfying $x^j_h(V^j_{s(h)})\subset V^j_{t(h)}, d^j_i(V^j_i)\subset V^{j+1}_i$ for $i\in I,h\in H,j\in \bbZ_2$. Thus the restriction $b|_{{\cM^{(2)}}^{\ualpha}_{\ualpha'\ualpha''}}$ is representable and proper, and so is $b$.

(b) Since the stack $\cM_{\ualpha}$ and the morphism $b$ are of finite type, the substack
$$b^{-1}(\cM_{\ualpha})=\bigsqcup_{\ualpha'+\ualpha''=\ualpha}(a_1,a_2)^{-1}(\cM_{\ualpha'}\times \cM_{\ualpha''})\subset \cM^{(2)}$$
is of finite type, and so the morphism $(a_1,a_2)$ is of finite type.
\end{proof}

By above lemma, since pulling back of an atlas for $\cM$ along the morphism $b$ gives an atlas for $\cM^{(2)}$, we have the following corollary.

\begin{corollary}
The stack $\cM^{(2)}$ is an algebraic stack.
\end{corollary}

\subsection{Relative Grothendieck group of stacks}\label{Relative Grothendieck group of stacks}\

Let $\cM$ be an Artin stack with affine stabilizers which is locally of finite type over $\bbC$, then there is a $2$-category of algebraic stacks over $\cM$. Let $\St/\cM$ be the full subcategory consisting of objects of the form $f:\cF\rightarrow \cM$, where $\cF$ is of finite type over $\bbC$ and has affine stabilizers.

A morphism of stacks is said to be a locally trivial fibration in the Zariski topology, if it is representable and its each pullback to a scheme is a locally trivial fibration of schemes in the Zariski topology.

\begin{definition}\label{Grothendieck group of stacks}
Define $K(\St/\cM)$ to be the free abelian group with a basis given by isomorphism classes of objects in $\St/\cM$, subject to following relations:\\
(i) for any object $f:\cF\rightarrow \cM$ and any closed substack $\cV\subset \cF$ with open complement $\cU=\cF\setminus \cV$, we have
$$[\cF\xrightarrow{f} \cM]=[\cV\xrightarrow{f|_{\cV}} \cM]+[\cU\xrightarrow{f|_{\cU}} \cM];$$
(ii) for any object $f:\cF\rightarrow \cM$ and any morphisms $h_1:\cG_1\rightarrow \cF,h_2:\cG_2\rightarrow \cF$ which are locally trivial fibrations in the Zariski topology with the same fibers, we have 
$$[\cG_1\xrightarrow{fh_1} \cM]=[\cG_2\xrightarrow{fh_2} \cM].$$
We denote by $K^0(\St/\cM)$ the subgroup generated by elements of the form $[\cF\xrightarrow{f}\cM]$, where $f$ is representable.
\end{definition}

Viewing $\bbC$ as a stack, the abelian group $K(\St/\bbC)$ has been defined as above. In this case, we denote its element by $[\cF]$ instead of $[\cF\rightarrow \bbC]$. Viewing varieties as stacks, $K(\St/\bbC)$ contains a subgroup $K(\Var/\bbC)$ spanned by elements of the form $[X]$, where $X$ is a variety over $\bbC$. Moreover, $K(\St/\bbC)$ is a commutative ring with the multiplication given by $[\cF_1]\cdot [\cF_2]=[\cF_1\times \cF_2]$, and $K(\Var/\bbC)$ is a subring.

An algebraic group $G$ is said to be special, if any morphism of schemes $f:X\rightarrow Y$ which is a principal $G$-bundle in the \'{e}tale topology is a Zariski fibration. By \cite{Chevalley-1958}, also see \cite[Section 2.1]{Joyce-2007.2}, we have\\
(a) the multiplication group $\bbC^*$ and the general linear groups $\textrm{GL}_n(\bbC)$ are special;\\
(b) if $A$ is a finite-dimensional $\bbC$-algebra, then $A^*$ the group of invertible elements is special;\\
(c) products of special groups are special;\\
(d) special groups are affine and connected.

\begin{lemma}[\cite{Bridgeland-2012}, Lemma 3.8, 3.9]
Localizing the ring $K(\Var/\bbC)$ with respect to any of the following three sets of elements gives the same result:\\
(i) $\{[G]|G\ \textrm{is a special algebraic group}\}$;\\
(ii) $\{[\textrm{GL}_d(\bbC)]| d\geqslant 1\}$;\\
(iii) $\{[\bbC]\}\cup\{[\bbC]^i-1|i\geqslant 1\},$\\
which is isomorphic to $K(\St/\bbC)$ as commutative rings.
\end{lemma}

The abelian group $K(\St/\cM)$ has a $K(\St/\bbC)$-module structure given by 
$$[\cG].[\cF\xrightarrow{f}\cM]=[\cG\times\cF\xrightarrow{f\pi_2}\cM],$$
where $\pi_2:\cG\times \cF\rightarrow\cF$ is the second projection, and the subgroup $K^0(\St/\cM)$ is a submodule.

There is a unique ring homomorphism $\Upsilon:K(\St/\bbC)\rightarrow \bbC(t)$ such that for any smooth projective variety $X$ over $\bbC$, regarded as a representable stack, 
$$\Upsilon([X])=\sum^{2\textrm{dim}\, X}_{i=0}\dim\rH^i_c(X_{\textrm{an}},\bbC)t^i\in \bbC[t]$$
is the Poincar\'{e} polynomial of $X$, where $X_{\textrm{an}}=X$ regarded as a compact complex manifold with analytic topology, and for any projective variety $X$ over $\bbC$ together with a special algebraic group $G$-action,
$$\Upsilon([X/G])=\Upsilon([X])\Upsilon([G])^{-1},$$
see \cite[Example 2.12, Theorem 2.14]{Joyce-2007} and \cite[Section 5.3]{Bridgeland-2017}. The following facts $\Upsilon(\bbC^*)=t^2-1$ and $\Upsilon(\bbC^n)=t^{2n}$ for any $n\in \bbZ_{\geqslant 0}$ are frequently used.

\begin{definition}\label{relation3}
Define $K_\Upsilon(\St/\cM)=\bbC(t)\otimes_{K(\St/\bbC)}K(\St/\cM)$ which is a $\bbC(t)$-vector space with a basis given by isomorphism classes of objects in $\St/\cM$, subject to the relations (i),(ii) in Definition \ref{Grothendieck group of stacks} and the following relations:\\
(iii) for any object $f:\cF\rightarrow \cM$ in $\St/\cM$ and object $\cG$ in $\St/\bbC$, we have 
$$[\cG\times \cF\xrightarrow{f\pi_2}\cM]=\Upsilon([\cG])[\cF\xrightarrow{f}\cM].$$
We deonte by $K^0_\Upsilon(\St/\cM)=\bbC(t)\otimes_{K(\St/\bbC)}K^0(\St/\cM)$ the $\bbC(t)$-subspace spanned by elements of the form $[\cF\xrightarrow{f}\cM]$, where $f$ is representable.
\end{definition}

\begin{remark}
The $\bbC(t)$-vector space $K^0_\Upsilon(\St/\cM)$ is denoted by $\textrm{SF}(\cM,\Upsilon,\bbC(t))$ in \cite[Section 2.5]{Joyce-2007}.
\end{remark}

\begin{proposition}[\cite{Joyce-2007}, Definition 2.8, \cite{Bridgeland-2012}, Section 3.5]\label{functorial}
Let $\cM_1,\cM_2$ be Artin stacks with affine stabilizers which are locally of finite type over $\bbC$ and $\varphi:\cM_1\rightarrow \cM_2$ be a morphism of stacks, then\\ 
(a) the morphism $\varphi$ induces a morphism of $K(\St/\bbC)$-modules
\begin{align*}
\varphi_*:K(\St/\cM_1)&\rightarrow K(\St/\cM_2)\\
[\cF\xrightarrow{f}\cM_1]&\mapsto [\cF\xrightarrow{\varphi f}\cM_2];
\end{align*}
moreover, if $\varphi$ is representable, then $\varphi_*:K^0(\St/\cM_1)\rightarrow K^0(\St/\cM_2)$;\\
(b) if $\varphi$ is of finite type, then it induces a morphism of $K(\St/\bbC)$-modules
\begin{align*}
\varphi^*:K(\St/\cM_2)&\rightarrow K(\St/\cM_1)\\
[\cF\xrightarrow{f}\cM_2]&\mapsto [\cG\xrightarrow{f'}\cM_1],
\end{align*}
where $f':\cG\rightarrow \cM_1$ is given by the cartesian diagram
\begin{diagram}[midshaft,size=2em]
\cG &\rTo^{f'} &\cM_1\\
\dTo &\square &\dTo_{\varphi}\\
\cF &\rTo^{f} &\cM_2;
\end{diagram}
moreover, $\varphi^*:K^0(\St/\cM_2)\rightarrow K^0(\St/\cM_1)$;\\
(c) there is a morphism $K(\St/\bbC)$-module ($K\ddot{u}nneth$ map) given by
\begin{align*}
K:K(\St/\cM_1)\otimes K(\St/\cM_2)&\rightarrow K(\St/\cM_1\times \cM_2)\\
[\cF_1\xrightarrow{f_1}\cM_1]\otimes[\cF_2\xrightarrow{f_2}\cM_2]&\mapsto [\cF_1\times\cF_2\xrightarrow{f_1\times f_2}\cM_1\times \cM_2];
\end{align*}
moreover, $K:K^0(\St/\cM_1)\otimes K^0(\St/\cM_2)\rightarrow K^0(\St/\cM_1\times \cM_2)$.
\end{proposition}

\subsection{Motivic Hall algebra for $\cC_2(\cA)$}\

Let $\cM,\cM^{(2)}$ be the moduli stacks of objects and filtrations in $\cC_2(\cA)$ defined in Subsections \ref{Moduli stacks of objects}, \ref{Moduli stacks of filtrations}.

By Lemma \ref{convolution diagram} and Proposition \ref{functorial}, there are morphisms $b_*$ and $(a_1,a_2)^*$. Then the convolution products on $K(\St/\cM)$ and $K^0(\St/\cM)$ are defined by 
\begin{align*}
-\Diamond-&=b_*(a_1,a_2)^*K:K(\St/\cM)\otimes K(\St/\cM)\rightarrow K(\St/\cM),\\
-\Diamond-&=b_*(a_1,a_2)^*K:K^0(\St/\cM)\otimes K^0(\St/\cM)\rightarrow K^0(\St/\cM).
\end{align*}
Explicitly, the product is given by the rule
$$[\cF_1\xrightarrow{f_1}\cM]\Diamond[\cF_2\xrightarrow{f_2}\cM]=[\cG\xrightarrow{bg}\cM],$$
where $g:\cM\rightarrow \cM^{(2)}$ is given by the cartesian diagram
\begin{diagram}[midshaft,size=2em]
\cG &\rTo^{g} &\cM^{(2)} &\rTo^b &\cM\\
\dTo &\square &\dTo_{(a_1,a_2)}\\
\cF_1\times \cF_2 &\rTo^{f_1\times f_2} &\cM\times \cM.
\end{diagram}

\begin{theorem}[\cite{Joyce-2007}, Theorem 5.2, \cite{Bridgeland-2013}, Theorem 4.3]
The convolution product $-\Diamond-$ defines a multiplication such that $K(\St/\cM)$ is an associative algebra over $K(\St/\bbC)$ with the unit element 
$$1=[\cM_0\hookrightarrow \cM],$$
where $\cM_0\cong \Spec\bbC$ is the substack of the zero complex in $\cC_2(\cA)$, and $K^0(\St/\cM)$ is a $K(\St/\bbC)$-subalgebra.
\end{theorem}

\begin{corollary}\label{associative}
The convolution product gives an associative algebra structure for $K_\Upsilon(\St/\cM)$ over $\bbC(t)$ with the unit element $1$, and $K^0_\Upsilon(\St/\cM)$ is a $\bbC(t)$-subalgebra.
\end{corollary}

The decomposition 
$\cM=\bigsqcup_{\ualpha\in \bbN I\times \bbN I}\cM_{\ualpha}$
induces a $\bbN I\times \bbN I$-grading on the algebra
$$K^0_\Upsilon(\St/\cM)=\bigoplus_{\ualpha\in \bbN I\times \bbN I}K^0_\Upsilon(\St/\cM)_{\ualpha},$$
where $K^0_\Upsilon(\St/\cM)_{\ualpha}=K^0_\Upsilon(\St/\cM_{\ualpha})$.

\subsection{Twisted subalgebra for $\cC_2(\cP)$}\

In Subsection \ref{moduli C2(P)}, we define a locally closed subset $\rP_2(A,\ue)\subset \rC_2(A,\ualpha(\ue))$ which is $\rG_{\ue}$-invariant for any $\ue\in \bbN I\times \bbN I$. We denote by $\cN_{\ue}=[\rP_2(A,\ue)/\rG_{\ue}]$ the quotient stack and 
$$\cN=\bigsqcup_{\ue\in \bbN I\times \bbN I}\cN_{\ue},$$
then $\cN_{\ue}$ and $\cN$ are locally closed substacks of $\cM_{\ualpha(\ue)}$ and $\cM$, respectively.

Let $\cH_t(\cC_2(\cP))$ be the $\bbC(t)$-subspace of $K^0_{\Upsilon}(\St/\cM)$ spanned by elements of the form 
$[\cF\xrightarrow{f}\cN\hookrightarrow \cM]$, then $\cH_t(\cC_2(\cP))$ is a subalgebra of $K^0_{\Upsilon}(\St/\cM)$, since the middle terms of extensions of projective complexes are projective complexes. Moreover, since 
$b((a_1,a_2)^{-1}(\cN_{\ue'}\times \cN_{\ue''}))\subset \cN_{\ue'+\ue''},$
the $\bbC(t)$-algebra $\cH_t(\cC_2(\cP))$ is also $\bbN I\times \bbN I$-graded,
$$\cH_t(\cC_2(\cP))=\bigoplus_{\ue\in \bbN I\times \bbN I}\cH_t(\cC_2(\cP))_{\ue},$$
where $\cH_t(\cC_2(\cP))_{\ue}$ is spanned by elements of the form 
$[\cF\xrightarrow{f}\cN_{\ue}\hookrightarrow \cM].$

Recall that the Euler form and the symmetric Euler form of $\cA$ are bilinear forms on the Grothendieck group $K(\cA)$ defined by
\begin{equation}
\begin{aligned}\label{Euler form and the symmetric Euler form}
K(\cA)\times K(\cA)&\rightarrow \bbZ\\
\langle \hat{M},\hat{N}\rangle&=\sum_{s\in \bbZ}(-1)^i\dim_{\bbC}\Ext^s_{\cA}(M,N)\\
(\hat{M},\hat{N})&=\langle\hat{M},\hat{N}\rangle+\langle\hat{N},\hat{M}\rangle,
\end{aligned}
\end{equation}
where $\hat{M},\hat{N}\in K(\cA)$ are the images of $M,N\in \cA$ in $K(\cA)$, respectively. Note that (\ref{Euler form and the symmetric Euler form}) is a finite sum, since $\cA$ is of finite global dimension.

We define the twist form $\cH^{\tw}_t(\cC_2(\cP))$ to be a $\bbC(t)$-algebra, where $\cH^{\tw}_t(\cC_2(\cP))=\cH_t(\cC_2(\cP))$ as $\bbC(t)$-vector spaces and the multiplication is replaced by  
\begin{align*}
&[\cF\xrightarrow{f'}\cN_{\ue'}\hookrightarrow \cM]*[\cF\xrightarrow{f''}\cN_{\ue''}\hookrightarrow \cM]\\
=&(-t)^{\langle \hat{P}'^1,\hat{P}''^1\rangle+\langle \hat{P}'^0,\hat{P}''^0\rangle}[\cF\xrightarrow{f'}\cN_{\ue'}\hookrightarrow \cM]\Diamond[\cF\xrightarrow{f''}\cN_{\ue''}\hookrightarrow \cM],
\end{align*}
where $P'^j=\bigoplus_{i\in I}e'^j_iP_i,P''^j=\bigoplus_{i\in I}e''^j_iP_i$. By Corollary \ref{associative} and the bilinearity of the Euler form, $\cH^{\tw}_t(\cC_2(\cP))$ is a $\bbN I\times \bbN I$-graded associative $\bbC(t)$-algebra with the unit element $1$.

\subsection{Localization}\

The following proposition is given by \cite[Proposition 5.15, Corollary 5.16]{Joyce-2007}, also see \cite[Proposition 6.2]{Bridgeland-2012}.

\begin{proposition}\label{Motivic Riedtmann-Peng}
For any $[\cF\xrightarrow{f}\cM],[\cG\xrightarrow{g}\cM]\in K^0(\St/\cM)$, assume that
$$\cF\times\cG\xrightarrow{f\times g}\cV\hookrightarrow\cM\times \cM,$$
where $\cV$ is a contractible substack of $\cM\times \cM$. Then there exists a finite decomposition of $\cV$ into quotient stacks
$$\cV=\bigsqcup_{m\in M}[V_m/G_m],$$
where $V_m$ is a quasi-projective variety over $\bbC$, $G_m$ is a special algebraic group acting on $V_m$ and $E^0_m,E^1_m$ are two finite-dimensional representations of $G_m$ such that\\ 
(a) if $v\in V_m$ projects to $([X],[Y])\in \cV(\bbC)\subset \cM(\bbC)\times \cM(\bbC)$ corresponding to two isomorphism classes of objects in $\cC_2(\cA)$, we have 
\begin{equation}
\begin{aligned}\label{identification}
&\Stab_{G_m}(v)\cong \Aut_{\cC_2(\cA)} (X)\times \Aut_{\cC_2(\cA)} (Y),\\
E^0_m&\cong \Hom_{\cC_2(\cA)}(X,Y),\  E^1_m\cong \Ext^1_{\cC_2(\cA)}(X,Y),
\end{aligned}
\end{equation}
and the action of $\Stab_{G_m}(v)$ on $E^0_m,E^1_m$ coincides with the action of $\Aut (X)\times \Aut (Y)$ on $\Hom_{\cC_2(\cA)}(X,Y),\Ext_{\cC_2(\cA)}(X,Y)$ respectively;\\
(b) there is a cartesian diagram
\begin{diagram}[midshaft,size=2em]
[V_m\times E^1_m/G_m\ltimes E^0_m] &\rTo^{\pi_m} &\cM^{(2)}\\
\dTo &\square &\dTo_{(a_1,a_2)}\\
[V_m/G_m] &\rInto &\cM\times \cM,
\end{diagram}
where the multiplication on $G_m\ltimes E^0_m$ is given by $(\gamma,\theta)(\gamma',\theta')=(\gamma\gamma',\theta+\gamma.\theta')$ for $\gamma,\gamma'\in G_m,\theta,\theta'\in E^0_m$, and $E^0_m$ acts trivially on $V_m\times E^1_m$, and if $v\in V_m$ projects to $([X],[Y])\in \cV(\bbC)$, then under the identifications $(\ref{identification})$, the morphism $\pi_m$ corresponds to the map $(v,\xi)\mapsto [\im a_\xi\subset Z_\xi]$, where 
$$0\rightarrow Y\xrightarrow{a_\xi}Z_\xi\xrightarrow{b_\xi} X\rightarrow 0$$
is the short exact sequence determined by $\xi\in E_m^1\cong\Ext^1_{\cC_2(\cA)}(X,Y)$.\\
(c) the element $[\cF\times\cG\xrightarrow{f\times g}\cV\hookrightarrow\cM\times \cM]$ can be written as 
$$\sum_{m\in M,n\in N_m}c_{mn}[[W_{mn}/G_m]\xrightarrow{\overline{\tau}_{mn}}[V_m/G_m]\hookrightarrow \cV\hookrightarrow \cM\times \cM],$$
where $N_m$ is a finite set, $W_{mn}$ is a quasi-projective variety over $\bbC$ together with a $G_m$-action, and $\overline{\tau}_{mn}:[W_{mn}/G_m]\rightarrow [V_m/G_m]$ is induced by a $G_m$-equivariant morphism of varieties $\tau_{mn}:W_{mn}\rightarrow V_m$, such that there is a cartesian diagram
\begin{diagram}[midshaft,size=2em]
[W_{mn}\times E^1_m/G_m\ltimes E^0_m] &\rTo^{\tilde{\tau}_{mn}} &[V_m\times E^1_m/G_m\ltimes E^0_m] &\rTo^{\pi_m} &\cM^{(2)}\\
\dTo & &\square & &\dTo_{(a_1,a_2)}\\
[W_{mn}/G_m] &\rTo^{\overline{\tau}_{mn}} &[V_m/G_m] &\rInto &\cM\times \cM,
\end{diagram}
where $\tilde{\tau}_{mn}:[W_{mn}\times E^1_m/G_m\ltimes E^0_m]\rightarrow [V_m\times E^1_m/G_m\ltimes E^0_m]$ is induced by $\tau_{mn}$, and so 
\begin{align*}
&[\cF\xrightarrow{f}\cM]\Diamond[\cG\xrightarrow{g}\cM]\\
=&\sum_{m\in M,n\in N_m}c_{mn}[[W_{mn}\times E^1_m/G_m\ltimes E^0_m]\xrightarrow{\tilde{\tau}_{mn}}[V_m\times E^1_m/G_m\ltimes E^0_m]\xrightarrow{b\pi_m}\cM].
\end{align*}
\end{proposition}

\begin{remark}\label{special case}
In particular, for $\delta_{[X]}=[\Spec\bbC\xrightarrow{i_{[X]}} \cM], \delta_{[Y]}=[\Spec\bbC\xrightarrow{i_{[Y]}} \cM]$ corresponding to isomorphism classes $[X],[Y]$ in $\cC_2(\cA)$, there is a cartesian diagram
\begin{diagram}[midshaft,size=2em]
[\Ext^1_{\cC_2(\cA)}(X,Y)/(\Aut_{\cC_2(\cA)} (X)\!\times\! \Aut_{\cC_2(\cA)} (Y))\!\ltimes\! \Hom_{\cC_2(\cA)}(X,Y)] &\rTo^{\pi} &\cM^{(2)}\\
\dTo &\square &\dTo^{(a_1,a_2)}\\
\Spec\bbC\!\times\! \Spec\bbC &\rTo^{i_{[X]}\times i_{[Y]}} &\cM\!\times\! \cM
\end{diagram} 
such that 
\begin{align*}
\delta_{[X]}\Diamond\delta_{[Y]}
\!=\![[\Ext^1_{\cC_2(\cA)}(X,Y)/(\Aut_{\cC_2(\cA)} (X)\!\times\! \Aut_{\cC_2(\cA)} (Y))\!\ltimes\! \Hom_{\cC_2(\cA)}(X,Y)]\xrightarrow{b\pi}\cM].
\end{align*}
It is analogue to the Riedtmann-Peng formula for Ringel-Hall algebra over a finite field $\bbF_q$. More precisely, consider $\cC_2(\cA_{\bbF_q})$, the category of two-periodic complexes over $\bbF_q$, for any $X, Y, Z\in \cC_2(\cA_{\bbF_q})$, the structure constant for Ringel-Hall algebra is given by the filtration number $g^Z_{XY}$ which is defined to be the number of subobjects $Z'\cong Z$ satisfying $Z'\cong Y, Z/Z'\cong X$, then the Riedtmann-Peng formula is
$$g^Z_{XY}=\frac{|\Ext^1_{\cC_2(\cA_{\bbF_q})}(X,Y)_Z||\Aut_{\cC_2(\cA_{\bbF_q})}(Z)|}{|\Hom_{\cC_2(\cA_{\bbF_q})}(X,Y)||\Aut_{\cC_2(\cA_{\bbF_q})}(X)||\Aut_{\cC_2(\cA_{\bbF_q})}(Y)|},$$
where $\Ext^1_{\cC_2(\cA_{\bbF_q})}(X,Y)_Z\subset \Ext^1_{\cC_2(\cA_{\bbF_q})}(X,Y)$ is the subset consisting of equivalence classes of short exact sequences whose middle terms are isomorphic to $Z$, see \cite{Ringel-1996,Schiffmann-2012-1} for details.
\end{remark}

\begin{proposition}
For any $[\cF\xrightarrow{f}\cN_{\ue}\hookrightarrow \cM]\in \cH^{\tw}_t(\cC_2(\cP))$ and $P\in \cP$, we have
\begin{align*}
&\delta_{[K_P]}*[\cF\xrightarrow{f}\cN_{\ue}\hookrightarrow \cM]=(-t)^{(\hat{P},\hat{P}^0-\hat{P}^1)}[\cF\xrightarrow{f}\cN_{\ue}\hookrightarrow \cM]*\delta_{[K_P]},\\
&\delta_{[K_P^*]}*[\cF\xrightarrow{f}\cN_{\ue}\hookrightarrow \cM]=(-t)^{-(\hat{P},\hat{P}^0-\hat{P}^1)}[\cF\xrightarrow{f}\cN_{\ue}\hookrightarrow \cM]*\delta_{[K_P^*]},
\end{align*}
where $P^j=\bigoplus_{i\in I}e_i^jP_i$.
\end{proposition}
\begin{proof}
The proof is essentially the same as \cite[Lemma 3.4, 3.5]{Bridgeland-2013}. By Proposition \ref{Motivic Riedtmann-Peng}, we can calculate the product $\delta_{[K_P]}*[\cF\xrightarrow{f}\cN_{\ue}\hookrightarrow \cM]$ as follows,
\begin{align*}
&\delta_{[K_P]}\Diamond[\cF\xrightarrow{f}\cN_{\ue}\hookrightarrow \cM]\\
=&\sum_{m\in M,n\in N_m}c_{mn}[[W_{mn}\times E^1_m/G_m\ltimes E^0_m]\xrightarrow{\tilde{\tau}_{mn}}[V_m\times E^1_m/G_m\ltimes E^0_m]\xrightarrow{b\pi_m}\cM]
\end{align*}
where if $v\in V_m$ projects to $([K_P],[X])$, then
$$E^1_m\cong \Ext_{\cC_2(\cA)}(K_P,X)\!=\!0,\ 
E^0_m\cong \Hom_{\cC_2(\cA)}(K_P,X)\!=\!\Hom_{\cA}(P,X^1)\!=\!\Hom_{\cA}(P,P^1)$$
see \cite[Lemma 3.3]{Bridgeland-2013}, and $b\pi_m$ corresponds to the map $(v,0)\mapsto [K_P\oplus X]$. Then by $\Upsilon(\bbC^n)=t^{2n}$, we obtain
\begin{equation}
\begin{aligned}\label{K_P*-}
&\delta_{[K_P]}*[\cF\xrightarrow{f}\cN_{\ue}\hookrightarrow \cM]=(-t)^{\langle \hat{P},\hat{P}^1\rangle+\langle \hat{P},\hat{P}^0\rangle}\delta_{[K_P]}\Diamond[\cF\xrightarrow{f}\cN_{\ue}\hookrightarrow \cM]\\
=&(-t)^{\langle \hat{P},\hat{P}^1\rangle+\langle \hat{P},\hat{P}^0\rangle}\sum_{m\in M,n\in N_m}\Upsilon(G_m\ltimes E^0_m)^{-1}c_{mn}[W_{mn}\times E^1_m\xrightarrow{b\pi_m\tilde{\tau}_{mn}\pi_{mn}}\cM]\\
=&(-t)^{\langle \hat{P},\hat{P}^0\rangle-\langle \hat{P},\hat{P}^1\rangle}\sum_{m\in M,n\in N_m}\Upsilon(G_m)^{-1}c_{mn}[W_{mn}\times E^1_m\xrightarrow{b\pi_m\tilde{\tau}_{mn}\pi_{mn}}\cM],
\end{aligned}
\end{equation}
where $\pi_{mn}:W_{mn}\times E^1_m\rightarrow [W_{mn}\times E^1_m/G_m\ltimes E^0_m]$ is the natural morphism. The calculation of $[\cF\xrightarrow{f}\cN_{\ue}\hookrightarrow \cM]*\delta_{[K_P]}$ is similar. The only difference is 
\begin{align*}
\Ext^1_{\cC_2(\cA)}(X,K_P)=0,\ \Hom_{\cC_2(\cA)}(X,K_P)=\Hom_{\cA}(X^0,P).
\end{align*} 
Note that 
$[\Spec\bbC\times \cF\xrightarrow{i_{[K_P]}\times f}\cM]=[\cF\times \Spec\bbC\xrightarrow{f\times i_{[K_P]}}\cM]$
via the natural switching isomorphism $\Spec\bbC\times \cF\cong \cF\times \Spec\bbC$. Thus we only need to change $E_m^0$ and the twist in (\ref{K_P*-}) and then obtain
\begin{align*}
&[\cF\xrightarrow{f}\cN_{\ue}\hookrightarrow \cM]*\delta_{[K_P]}\\
=&(-t)^{\langle \hat{P}^1,\hat{P}\rangle-\langle \hat{P}^0,\hat{P}\rangle}\sum_{m\in M,n\in N_m}\Upsilon(G_m)^{-1}c_{mn}[W_{mn}\times E^1_m\xrightarrow{b\pi_m\tilde{\tau}_{mn}\pi_{mn}}\cM].
\end{align*}
Therefore,
$$\delta_{[K_P]}*[\cF\xrightarrow{f}\cN_{\ue}\hookrightarrow \cM]=(-t)^{(\hat{P},\hat{P}^0-\hat{P}^1)}[\cF\xrightarrow{f}\cN_{\ue}\hookrightarrow \cM]*\delta_{[K_P]},$$
as desired. The second equation can be proved similarly.
\end{proof}

For any $P\in \cP$, the automorphism group $\Aut_{\cA}(P)$ consists of invertible elements in the endomorphism algebra $\End_{\cA}(P)$, so the automorphism groups 
$$\Aut_{\cC_2(\cA)} (K_P)=\Aut_{\cC_2(\cA)}(K_P^*)=\Aut_{\cA}(P)$$
are special, see Subsection \ref{Relative Grothendieck group of stacks}.

\begin{definition}
For each $P\in \cP$, define two elements
$$b_P=\Upsilon(\Aut_{\cC_2(\cA)} (K_P))\delta_{[K_P]},\ b_P^*=\Upsilon(\Aut_{\cC_2(\cA)} (K_P^*))\delta_{[K_P^*]}\in\cH^{\tw}_t(\cC_2(\cP)).$$
\end{definition}

\begin{corollary}\label{Ore conditions}
For any $[\cF\xrightarrow{f}\cN_{\ue}\hookrightarrow \cM]\in \cH^{\tw}_t(\cC_2(\cP))$ and $P\in \cP$, we have
\begin{align*}
&b_P*[\cF\xrightarrow{f}\cN_{\ue}\hookrightarrow \cM]=(-t)^{(\hat{P},\hat{P}^0-\hat{P}^1)}[\cF\xrightarrow{f}\cN_{\ue}\hookrightarrow \cM]*b_P,\\
&b_P^**[\cF\xrightarrow{f}\cN_{\ue}\hookrightarrow \cM]=(-t)^{-(\hat{P},\hat{P}^0-\hat{P}^1)}[\cF\xrightarrow{f}\cN_{\ue}\hookrightarrow \cM]*b_P^*,
\end{align*}
where $P^j=\bigoplus_{i\in I}e_i^jP_i$. Hence $\{b_P,b_P^*|P\in \cP\}$ satisfies Ore conditions. 
\end{corollary}

\begin{lemma}\label{regular}
For any $P,Q\in \cP$, we have
\begin{align*}
&b_P*b_Q=b_{P\oplus Q},\\
&b_P*b_Q^*=\Upsilon(\Aut_{\cC_2(\cA)}(K_P\oplus K_Q^*))\delta_{[K_P\oplus K_Q^*]},\\
&[b_P,b_Q]=[b_P, b_Q^*]=[b_P^*,b_Q^*]=0,
\end{align*}
where the Lie bracket is the commutator $[a,b]=a*b-b*a$.
\end{lemma}
\begin{proof}
By \cite[Lemma 3.3]{Bridgeland-2013} and direct checking, we have
$$\Ext^1_{\cC_2(\cA)}(K_P,K_Q)=0,\ \Hom_{\cC_2(\cA)}(K_P,K_Q)=\Hom_\cA(P,Q),$$
and so $\Upsilon(\Hom_{\cC_2(\cA)}(K_P,K_Q))=t^{2\langle \hat{P},\hat{Q}\rangle}$. By Remark \ref{special case}, that is, a special case of Proposition \ref{Motivic Riedtmann-Peng}, we have 
\begin{align*}
&\delta_{[K_P]}\Diamond \delta_{[K_Q]}\\
=&[[\Ext^1_{\cC_2(\cA)}(K_P,K_Q)/(\Aut_{\cC_2(\cA)}(K_P)\times \Aut_{\cC_2(\cA)}(K_Q))\ltimes \Hom_{\cC_2(\cA)}(K_P,K_Q)]\xrightarrow{b\pi_m}\cM]\\
=&\Upsilon(\Aut_{\cC_2(\cA)}(K_P))^{-1}\Upsilon(\Aut_{\cC_2(\cA)}(K_Q))^{-1}t^{-2\langle \hat{P},\hat{Q}\rangle}[\Ext^1_{\cC_2(\cA)}(K_P,K_Q)\xrightarrow{b\pi_m\pi}\cM],
\end{align*}
where $b\pi_m\pi$ corresponds to the map $0\mapsto [K_P\oplus K_Q]$, thus 
\begin{align*}
b_P*b_Q=&(-t)^{2\langle \hat{P},\hat{Q}\rangle}\Upsilon(\Aut_{\cC_2(\cA)}(K_P))\Upsilon(\Aut_{\cC_2(\cA)}(K_Q))\delta_{[K_P]}\Diamond \delta_{[K_Q]}\\
=&[\Ext^1_{\cC_2(\cA)}(K_P,K_Q)\xrightarrow{b\pi_m\pi}\cM]\\
=&\Upsilon(\Aut_{\cC_2(\cA)}(K_P\oplus K_Q))[\Spec\bbC\xrightarrow{i_{[K_P\oplus K_Q]}}\cM]=b_{P\oplus Q}.
\end{align*}
The proof of the second equation is similar, and then the third equation follows directly, by applying the $*$-involution. 
\end{proof}

As in \cite[Section 3.6]{Bridgeland-2013}, the localization of $\cH_t^{\tw}(\cC_2(\cP))$ with respect to $\{b_P,b_P^*|P\in \cP\}$ is well-defined by Corollary \ref{Ore conditions} and Lemma \ref{regular}.

\begin{definition}
Define the $\bbC(t)$-algebra
$$\cD\cH_t(\cA)=\cH_t^{\tw}(\cC_2(\cP))[b_P^{-1},{b_P^*}^{-1}|P\in \cP]$$
to be the localization of $\cH_t^{\tw}(\cC_2(\cP))$, and define its reduced quotient 
$$\cD\cH_t^{\red}(\cA)=\cD\cH_t(\cA)/\langle b_P*b_P^*-1|P\in \cP\rangle.$$
For any element in $\cD\cH_t(\cA)$, we denote by the same symbol its image in $\cD\cH_t^{\red}(\cA)$.
\end{definition}

Any element in $K(\cA)$ can be written in the form $\hat{P}-\hat{Q}$ for some $P,Q\in \cP$. Indeed, for $M\in \cA$, by taking its projective resolution, $\hat{M}=\sum(-1)^i\hat{P}_i$ for some $P_i\in \cP$, since $\cA$ is of finite global dimension. Thus $\hat{M}=\hat{P}-\hat{Q}$ for some $P,Q\in \cP$ and then so is any element in $K(\cA)$. For any $\alpha\in K(\cA)$, we assume that  $\alpha=\hat{P}-\hat{Q}$ for some $P,Q\in \cP$, by Lemme \ref{regular}, there is a well-defined element
$$b_{\alpha}=b_P*b_Q^*=b_P*b_Q^{-1}\in \cD\cH_t^{\red}(\cA).$$
By Corollary \ref{Ore conditions} and Lemma \ref{regular}, we have the following corollary. 

\begin{corollary}\label{property of b_alpha}
For any $[\cF\xrightarrow{f}\cN_{\ue}\hookrightarrow\cM]$ and $\alpha,\beta\in K(\cA)$, we have 
\begin{align*}
&b_{\alpha}*b_{\beta}=b_{\alpha+\beta},\ [b_{\alpha},b_{\beta}]=0,\\
&b_{\alpha}*[\cF\xrightarrow{f}\cN_{\ue}\hookrightarrow\cM]=(-t)^{(\alpha,\hat{P}^0-\hat{P}^1)}[\cF\xrightarrow{f}\cN_{\ue}\hookrightarrow\cM]*b_{\alpha},
\end{align*}
where $P^j=\bigoplus_{i\in I}e_i^jP_i$.
\end{corollary}

\section{Euler characteristic and constructible function}

In this section, we review Euler characteristic for algebraic variety $X$, as well as the naive  Euler characteristic for algebraic stack, especially for quotient stack $[X/G]$, where $G$ is an algebraic group acting on $X$. We refer \cite{Joyce-2006,Joyce-2007} for details.

\subsection{Euler characteristic and constructible function}\label{pushforward functor}\

Let $X$ be an algebraic variety over $\bbC$, the Euler characteristic of $X$ is defined by 
$$\chi(X)=\sum_{i=0}^{2\textrm{dim}\, X}(-1)^i\dim \rH^i_c(X_{\textrm{an}},\bbC),$$
where $X_{\textrm{an}}=X$ regarded as a compact complex manifold with analytic topology. Note that $\Upsilon(X)|_{t=-1}=\chi(X)$, where $\Upsilon(X)\in \bbC[t]$ is the Poincar\'{e} polynomial.

\begin{proposition}[\cite{Joyce-2006}, Proposition 3.6]\label{Euler characteristic}
Let $X,Y$ be algebraic varieties over $\bbC$, \\
(a) if $Z$ is a closed subvariety of $X$, then $\chi(X)=\chi(Z)+\chi(X\setminus Z)$;\\
(b) if $X$ is the disjoint union of finitely many subvarieties $X_1,...,X_n$, then 
$$\chi(X)=\sum^n_{i=1}\chi(X_i);$$
(c) if $\varphi:X\rightarrow Y$ is a morphism of varieties over $\bbC$ whose fibers have the same Euler characteristic $\chi$, then $\chi(X)=\chi.\chi(Y)$;\\
(d) $\chi(\bbC^n)=1,\chi(\bbC\bbP^n)=n+1$ for any $n\in \bbZ_{\geqslant 0}$.
\end{proposition}

A subset $\cO\subset X$ is called constructible, if $\cO=\bigsqcup_{i=1}^n U_i$, where $U_i\subset X$ are locally closed subsets. For such a constructible subset, we define $\chi(\cO)=\sum^n_{i=1}\chi(U_i)$. It is well-defined by using of property (b) in above proposition, see \cite[Definition 3.7]{Joyce-2006} and \cite[Section 2.4]{Riedtmann-1994}. Moreover,\\
\emph{(b') if $\cO$ is the disjoint union of finitely many constructible subsets $\cO_1,...,\cO_n$, then}
$$\chi(\cO)=\sum^n_{i=1}\chi(\cO_i).$$

A function $f:X\rightarrow \bbC$ is called constructible, if the image $f(X)\subset \bbC$ is a finite subset, and $f^{-1}(c)\subset X$ is constructible subset for any $c\in f(X)$. We denote by $M(X)$ the $\bbC$-vector spaces of constructible functions on $X$.

\subsection{Geometric quotient and naive Euler characteristic}\label{naive Euler characteristic}\
                              
Let $X$ be a smooth algebraic variety over $\bbC$ together with an algebraic group $G$-action given by $\sigma:G\times X\rightarrow X$.

\begin{definition}
A geometric quotient of $X$ by $G$ is an algebraic variety $Y$ over $\bbC$ together with a morphism $\pi:X\rightarrow Y$ of varieties over $\bbC$ such that\\
(i) $\pi$ is constant on $G$-orbits;\\
(ii) $\pi$ is surjective and the image of $(\pi,\textrm{pr}_2):G\times X\rightarrow X\times X$ is identified with $X\times_YX$, where $\textrm{pr}_2:G\times X\rightarrow X$ is the second projection;\\
(iii) a subset $U\subset Y$ is open if and only if $\pi^{-1}(U)\subset X$ is open;\\
(iv) for any open subset $U\subset Y$, the ring of regular functions $\cO_Y(U)$ is identified with the ring of $G$-invariant functions $\cO_X(\pi^{-1}(U))^G$ via the morphism induced by $\pi$.
\end{definition}

\begin{theorem}[\cite{Rosenlicht-1963}]\label{Rosenlicht}
Let $X$ be an algebraic variety over $\bbC$, and $G$ be an algebraic group acting on $X$. Then there exists an open dense $G$-invariant subset of $X$ which has a geometric quotient by $G$.
\end{theorem}

The definition of the naive Euler characteristic for algebraic $\bbC$-stacks with affine geometric stabilizers is given by \cite[Definition 4.8]{Joyce-2006}. Roughly speaking, for such a stack $\cF$, each constructible subset $C\subset \cF(\bbC)$ is pseudoisomorphic to $Y(\bbC)$ for some algebraic variety $Y$ over $\bbC$, by \cite[Proposition 4.7]{Joyce-2006}. The naive Euler characteristic of $C$ is defined by $\chi^{\textrm{na}}(C)=\chi(Y)$, which is independent of the choices of $Y$. The proof of \cite[Proposition 4.7]{Joyce-2006} is essentially based on Theorem \ref{Rosenlicht}. In order to avoid introducing the definition of pseudoisomorphism, we restate the definition of naive Euler characteristic for the quotient stack $[X/G]$ as follows.

Firstly, we construct a finite stratification of $X$. By Theorem \ref{Rosenlicht}, there exists an open dense $G$-invariant subset $U_1\subset X$ with a geometric quotient $\pi_1:U_1\rightarrow Y_1$. For $X\setminus U_1$ whose dimension is strictly smaller than $X$, by Theorem \ref{Rosenlicht}, there exists an open dense $G$-invariant subset $U_2\subset X\setminus U_1$ with a geometric quotient $\pi_2:U_2\rightarrow Y_2$. Since $\dim X$ is finite, inductively, we obtain a finite stratification $X=\bigsqcup_{i=1}^n U_i$, where $U_i$ are locally closed $G$-invariant subsets with geometric quotients $\pi_i:U_i\rightarrow Y_i$. Let $Y$ be the abstract disjoint union of all $Y_i$, then we define
$$\chi^{\textrm{na}}([X/G](\bbC))=\chi(Y)=\sum^n_{i=1}\chi(Y_i).$$
In general, for each finite type algebraic stack $\cF$ with affine geometric stabilizers, by \cite[Theorem 2.7]{Joyce-2006}, it can be stratified by global quotient stack, that is, its associated reduced stack can be stratified as $\cF^{\textrm{red}}=\bigsqcup_{i=1}^m\cF_i$ and $\cF(\bbC)=\cF^{\textrm{red}}(\bbC)=\bigsqcup^m_{i=1}\cF_i(\bbC)$, where $\cF_i$ is $1$-isomorphic to some quotient stack $[X_i/G_i]$, then we define
$$\chi^{\textrm{na}}(\cF(\bbC))=\sum^m_{i=1}\chi^{\textrm{na}}([X_i/G_i](\bbC)).$$
A constructible subset of $[X/G](\bbC)$ is a finite disjoint union $C=\bigsqcup_{i=1}^m\cF_i(\bbC)$, where each $\cF_i$ is a finite type algebraic substacks of $[X/G]$, then we define
$$\chi^{\textrm{na}}(C)=\sum^m_{i=1}\chi^{\textrm{na}}(\cF_i(\bbC)).$$
Above definitions are well-defined by a routine argument as \cite[Definition 3.7]{Joyce-2006}. 

For convenience, we simply denote by $\chi([X/G])=\chi^{\textrm{na}}([X/G](\bbC))$.

If $X'$ is another algebraic variety over $\bbC$ together with another algebraic group $G'$-action, then each $1$-morphism between quotient stacks $\varphi:[X/G]\rightarrow [X'/G']$ induces a morphism $\varphi(\bbC):[X/G](\bbC)\rightarrow [X'/G'](\bbC)$ between the sets of $\bbC$-valued points. By properties (b') and (c) in Proposition \ref{Euler characteristic}, the naive Euler characteristic has analogue properties.

\begin{corollary}\label{fibre naive Euler characteristic}
Let $X ,X'$ be algebraic varieties over $\bbC$ together with algebraic group $G$-action, $G'$-action, respectively, and $\varphi:[X/G]\rightarrow [X'/G']$ be a $1$-morphism between quotient stacks. Then\\
(a) if $[X/G](\bbC)$ is the disjoint union of finitely many constructible subsets $C_1,..,C_n$, then $\chi([X/G])=\sum^n_{i=1}\chi^{\textrm{na}}(C_i)$;\\
(b) if $\varphi(\bbC):[X/G](\bbC)\rightarrow [X'/G'](\bbC)$ is surjective and all fibers have the same naive Euler characteristic $\chi^{\textrm{na}}$, then $\chi([X/G])=\chi^{\textrm{na}}\cdot\chi([X'/G'])$.
\end{corollary}

\begin{definition}\label{integration}
Let $f$ be a $G$-invariant constructible function on $X$, we define the integration by
$$\int_{[X/G]}f(x)=\sum_{c\in \bbC}c\chi([f^{-1}(c)/G]).$$
\end{definition}

Recall that in Subsection \ref{Relative Grothendieck group of stacks}, we use the ring homomorphism $\Upsilon:K(\St/\bbC)\rightarrow \bbC(t)$ to define $K^0_\Upsilon(\St/\cM)$. Let 
$$\bbC_{-1}=\bbC[t]_{(t+1)}=\{\frac{f(t)}{g(t)}\in \bbC(t)|f(t),g(t)\in \bbC[t],g(-1)\not=0\}$$
be the localization of $\bbC[t]$ with respect to the prime ideal $(t+1)$. It is a local ring with the unique maximal ideal $(t+1)\bbC_{-1}$ and residue field $\bbC$. The evaluation morphism 
\begin{align*}
\pi:\bbC_{-1}&\rightarrow\bbC\\
\frac{f(t)}{g(t)}&\mapsto \frac{f(-1)}{g(-1)},
\end{align*}
induces an isomorphism $\bbC_{-1}/(t+1)\bbC_{-1}\cong \bbC$. Moreover, for any Artin stack $\cF$ with affine stabilizers which is of finite type over $\bbC$, we have
$$\chi^{\textrm{na}}(\cF)=\pi\Upsilon(\cF),$$
see \cite[Example 2.12]{Joyce-2007} for details.

\section{Lie algebra arising from $\cC_2(\cP)$}

\subsection{Classical limit}\label{Classical limit}\

In Subsection \ref{moduli C2(P)}, we define a closed subset $\rP_2^{\rad}(A,\ue)\subset \rP_2(A,\ue)$ which is $\rG_{\ue}$-invariant for any $\ue\in \bbN I\times \bbN I$. We denote by $\cN^{\rad}_{\ue}=[\rP_2^{\rad}(A,\ue)/\rG_{\ue}]$ the quotient stack and denote by
$$\cN^{\rad}=\bigsqcup_{\ue\in \bbN I\times \bbN I}\cN^{\rad}_{\ue},$$
then $\cN^{\rad}_{\ue}$ and $\cN^{\rad}$ are closed substacks of $\cN_{\ue}$ and $\cN$ respectively.

\begin{definition}\label{C_-1}
(a) Define the $\bbC_{-1}$-form $\cD\cH_t(\cA)_{\bbC_{-1}}$ to be the $\bbC_{-1}$-subalgebra of $\cD\cH^{\red}_t(\cA)$ generated by elements of the forms 
$$[[\cO/\rG_{\ue}]\hookrightarrow\cN^{\rad}\hookrightarrow \cM],\ b_{\alpha},\ \frac{b_{\alpha}-1}{-t-1},$$
for any $\rG_{\ue}$-invariant constructible subset $\cO\subset \rP_2^{\rad}(A,\ue),\ue\in \bbN I\times \bbN I$ and $\alpha\in K(\cA)$. \\
(b) Define the classical limit of $\cD\cH_t(\cA)_{\bbC_{-1}}$ to be the $\bbC$-algebra
$$\cD\cH_{-1}(\cA)=\bbC_{-1}/(t+1)\bbC_{-1}\otimes_{\bbC_{-1}}\cD\cH_t(\cA)_{\bbC_{-1}}.$$
We still denote by $1$ the unit $1\otimes 1\in \cD\cH_{-1}(\cA)$.
\end{definition}

Note that elements in $\cD\cH_t(\cA)_{\bbC_{-1}}$ are finite sums $\sum\frac{f(t)}{g(t)}a_1*...*a_n,$
where $\frac{f(t)}{g(t)}\in \bbC_{-1}$ and $a_i$ is one of the generators in Definition \ref{C_-1}. We denote by the $\bbC$-linear map
\begin{align*}
\ ^-:\cD\cH_t(\cA)_{\bbC_{-1}}&\rightarrow \cD\cH_{-1}(\cA)\\
\sum\frac{f(t)}{g(t)}a_1*...*a_s&\mapsto 1\otimes \sum\frac{f(t)}{g(t)}a_1*...*a_s=\sum\frac{f(-1)}{g(-1)}\otimes a_1*...*a_s
\end{align*}
which is a ring homomorphism. 

\begin{definition}
(a) For any $\alpha\in K(\cA)$, define the element $$h_\alpha=(\frac{b_\alpha-1}{-t-1})^-=1\otimes \frac{b_\alpha-1}{-t-1}\in \cD\cH_{-1}(\cA).$$
(b) Define $\fh$ to be the $\bbC$-subspace of $\cD\cH_{-1}(\cA)$ spanned by $\{h_\alpha|\alpha\in K(\cA)\}$.
\end{definition}

\begin{lemma}\label{property of h}
(a) For any $\alpha,\beta\in K(\cA)$ and $x=[\cF\xrightarrow{f} \cN^{\rad}_{\ue}\hookrightarrow\cM]$, we have
\begin{align*}
\overline{b_\alpha}=1,\ [h_\alpha,h_\beta]=0,\ h_{\alpha+\beta}=h_{\alpha}+h_{\beta}.
\end{align*}
Moreover, the map $\alpha\mapsto h_{\alpha}$ extends to an isomorphism $\bbC\otimes_{\bbZ}K(\cA)\cong \fh$.\\
(b) For any $\alpha\in K(\cA)$ and $x=[\cF\xrightarrow{f} \cN^{\rad}_{\ue}\hookrightarrow\cM]$, we have
\begin{align*}
[h_\alpha,\overline{x}]=(\alpha,\hat{P}^0-\hat{P}^1)\overline{x},
\end{align*}
where $P^j=\bigoplus_{i\in I}e^j_iP_i$.
\end{lemma}
\begin{proof}
(a) Applying map $^-$ to the following equations
\begin{align*}
&b_{\alpha}-1=-(t+1)\frac{b_{\alpha}-1}{-t-1},\\
&\frac{b_{\alpha}-1}{-t-1}*\frac{b_{\beta}-1}{-t-1}=\frac{b_{\beta}-1}{-t-1}*\frac{b_{\alpha}-1}{-t-1},\\
&\frac{b_{\alpha+\beta}-1}{-t-1}=\frac{b_{\alpha}-1}{-t-1}*b_{\beta}+\frac{b_{\beta}-1}{-t-1},
\end{align*}
where the second equation is given by Corollary \ref{property of b_alpha}, we obtain
$$\overline{b_\alpha}-1=0,\ h_\alpha*h_\beta=h_\beta*h_\alpha,\ h_{\alpha+\beta}=h_{\alpha}*\overline{b_{\beta}}+h_{\beta}=h_{\alpha}+h_{\beta}.$$ 
Thus the map $\alpha\mapsto h_{\alpha}$ defines a linear surjection from $K(\cA)$ to the $\bbZ$-lattice spanned by $\{h_\alpha|\alpha\in K(\cA)\}$. Suppose $\alpha=\hat{P}-\hat{Q}\in \cK(\cA)$, where $P,Q\in \cP$ such that $h_{\alpha}=0$, then we have $b_{\alpha}=1$ in $\cD\cH_t^{\red}(\cA)$, that is, $b_P*b_Q^{-1}=1=b_Q*b_Q^{-1}$, and so $b_P=b_Q, P=Q,\alpha=0$. Therefore, the map $\alpha\mapsto h_{\alpha}$ is also injective, and extends to an isomorphism $\bbC\otimes_{\bbZ}K(\cA)\cong \fh$.

(b) By Corollary \ref{property of b_alpha}, we have
\begin{align*}
&\frac{b_{\alpha}-1}{-t-1}*x-x*\frac{b_{\alpha}-1}{-t-1}
=\frac{(-t)^{(\alpha,\hat{P}^0-\hat{P}^1)}-1}{-t-1}x*b_{\alpha}\\
=&((-t)^{(\alpha,\hat{P}^0-\hat{P}^1)-1}+...+(-t)+1)x*b_{\alpha}.
\end{align*}
Applying map $^-$, we obtain $h_\alpha*\overline{x}-\overline{x}*h_\alpha=(\alpha,\hat{P}^0-\hat{P}^1)\overline{x}$, as desired.
\end{proof}

\subsection{Two Lemmas about extensions of two-periodic projective complexes.}\

For any $X,Y\in\cC_2(\cP)$, we have $\Hom_{\cC_2(\cA)}(X,Y)=\Hom_{\cC_2(\cP)}(X,Y)$ and $$\Ext^1_{\cC_2(\cA)}(X,Y)\cong\Ext^1_{\cC_2(\cP)}(X,Y),$$
where $\Ext^1_{\cC_2(\cP)}(-,-)$ is the extension space of $\cC_2(\cP)$ as an exact category. Since we only concern about projective complexes, from now on, we replace $\Hom_{\cC_2(\cA)}(-,-)$ and $\Ext^1_{\cC_2(\cA)}(-,-)$ by $\Hom_{\cC_2(\cP)}(-,-)$ and $\Ext^1_{\cC_2(\cP)}(-,-)$, respectively.

For any objects $X,Y,Z\in \cC_2(\cP)$, let $\Ext^1_{\cC_2(\cP)}(X,Y)_Z\subset \Ext^1_{\cC_2(\cP)}(X,Y)$ be the subset consisting of equivalence classes of short exact sequences whose middle terms are isomorphic to $Z$, and let $\Hom_{\cK_2(\cP)}(X,Y^*)_{Z^*}\subset \Hom_{\cK_2(\cP)}(X,Y^*)$ be the subset consisting of morphisms whose mapping cones are homotopy equivalent to $Z^*$.

By Lemma \cite[Lemma 3.3]{Bridgeland-2013}, there is a bijection $\Ext^1_{\cC_2(\cP)}(X,Y)\cong \Hom_{\cK_2(\cP)}(X,Y^*)$. The following lemma is a refinement.

\begin{lemma}[{\cite[Lemma 2.2]{Fang-Lan-Xiao-2023}}]\label{bijection between Ext and Hom}
For any $X,Y,Z\in\cC_2(\cP)$, if $\Ext^1_{\cC_2(\cP)}(X,Y)_Z \neq\varnothing$, then 
$$\Ext^1_{\cC_2(\cP)}(X,Y)_Z\cong \Hom_{\cK_2(\cP)}(X,Y^*)_{Z^*}.$$
\end{lemma}
\begin{proof}
For convenience, we sketch the construction of  
$$\Ext^1_{\cC_2(\cP)}(X,Y)\cong \Hom_{\cK_2(\cP)}(X,Y^*)$$
in \cite[Lemma 3.3]{Bridgeland-2013}. For any short exact sequence
\begin{align}\label{short exact sequence}
0\rightarrow Y\rightarrow Z\rightarrow X\rightarrow 0
\end{align}
in $\cC_2(\cP)$, since $X^j$ is projective, the short exact sequence $0\rightarrow Y^j\rightarrow Z^j\rightarrow X^j\rightarrow 0$ in $\cA$ splits. Selecting a splitting, we may assume that the short exact sequence (\ref{short exact sequence}) is of the form 
\begin{diagram}[midshaft,size=2em]
0 &\rTo &Y^1 &\rTo &X^1\oplus Y^1 &\rTo &X^1 &\rTo &0\\
 & &\dTo^{d^1_Y}\uTo_{d^0_Y}& &\dTo^{d^1_Z}\uTo_{d^0_Z}&&\dTo^{d^1_X}\uTo_{d^0_X}\\
0 &\rTo &Y^0 &\rTo &X^0\oplus Y^0 &\rTo &X^0 &\rTo &0,\\
\end{diagram}
where
$$d^1_Z=\begin{pmatrix}\begin{smallmatrix}
d^1_X &0\\
-f^1 &d^1_Y
\end{smallmatrix}\end{pmatrix}, \ d^0_Z=\begin{pmatrix}\begin{smallmatrix}
d^0_X &0\\
-f^0 &d^0_Y
\end{smallmatrix}\end{pmatrix},$$
for some $f^j\in \Hom_{\cA}(X^j,Y^{j+1})$, then the condition $d^{j+1}_Zd^j_Z=0$ is equivalent to condition $(f^1,f^0)\in \Hom_{\cC_2(\cA)}(X,Y^*)$. This gives a bijection between the set of short exact sequences of the form  (\ref{short exact sequence}) and the set $\Hom_{\cC_2(\cA)}(X,Y^*)$. Moreover, assume that two morphisms $f=(f^1,f^0),g=(g^1,g^0)\in \Hom_{\cC_2(\cP)}(X,Y^*)$ determine two equivalent short exact sequences, that is, there is a commutative diagram
\begin{diagram}[midshaft,size=2em]
(Y^1&\pile{\rTo^{d^1_Y}\\ \lTo_{d^0_Y}} &Y^0) &\rTo 
&(X^1\oplus Y^1&\pile{\rTo^{\begin{pmatrix}\begin{smallmatrix}
d^1_X &0\\
-f^1 &d^1_Y
\end{smallmatrix}\end{pmatrix}}\\ \lTo_{\begin{pmatrix}\begin{smallmatrix}
d^0_X &-0\\
-f^0 &d^0_Y
\end{smallmatrix}\end{pmatrix}}} &X^0\oplus Y^0) &\rTo &(X^1&\pile{\rTo^{d^1_X}\\ \lTo_{d^0_X}} &X^0)\\
 &\vEq & & &\dTo^{s^1} & &\dTo_{s^0} & & &\vEq\\
 (Y^1&\pile{\rTo^{d^1_Y}\\ \lTo_{d^0_Y}} &Y^0) &\rTo 
&(X^1\oplus Y^1&\pile{\rTo^{\begin{pmatrix}\begin{smallmatrix}
d^1_X &0\\
-g^1 &d^1_Y
\end{smallmatrix}\end{pmatrix}}\\ \lTo_{\begin{pmatrix}\begin{smallmatrix}
d^0_X &0\\
-g^0 &d^0_Y
\end{smallmatrix}\end{pmatrix}}} &X^0\oplus Y^0) &\rTo &(X^1&\pile{\rTo^{d^1_X}\\ \lTo_{d^0_X}} &X^0),
\end{diagram}
where $(s^1,s^0)$ is an isomorphism in $\cC_2(\cP)$, then the condition that the left and right squares commute is equivalent to the condition that $s^1,s^0$ are of the form 
$$s^1=\begin{pmatrix}
1 &h^1\\
0 &1
\end{pmatrix},\ s^0=\begin{pmatrix}
1 &h^0\\
0 &1
\end{pmatrix},$$
for some $h^j\in \Hom_{\cA}(X^j,Y^j)$, and the condition that the middle square commutes is equivalent to the condition that $h^1,h^0$ form a homotopy between $f$ and $g$,
\begin{diagram}[midshaft,size=2em]
X^1 &\pile{\rTo^{d^1_X}\\ \lTo_{d^0_X}} &X^0\\
\dTo^{f^1-g^1}_{\ \ \ h^1} &\rdDashto \ldDashto&\dTo^{h^0\ \ \ }_{f^0-g^0}\\
Y^0 &\pile{\rTo^{-d^0_Y}\\ \lTo_{-d^1_Y}} &Y^1.
\end{diagram}
Hence we obtain a bijection $\Ext^1_{\cC_2(\cP)}(X,Y)\cong \Hom_{\cK_2(\cP)}(X,Y^*)$.

For $f\in \Hom_{\cC_2(\cP)}(X,Y^*)$, by above construction, the middle term of the short exact sequence determined by the homotopy class of $f$ is isomorphic to the complex
\begin{diagram}[midshaft,size=2em]
X^1\oplus Y^1&\pile{\rTo^{\begin{pmatrix}\begin{smallmatrix}
d^1_X &0\\
-f^1 &d^1_Y
\end{smallmatrix}\end{pmatrix}}\\ \lTo_{\begin{pmatrix}\begin{smallmatrix}
d^0_X &0\\
-f^0 &d^0_Y
\end{smallmatrix}\end{pmatrix}}} &X^0\oplus Y^0,
\end{diagram}
which is the shift of the $\textrm{Cone}(f)$, as desired.
\end{proof}

\begin{lemma}\label{middle term contractible}
For any radical complexes $X,Y$ and contractible complex $Z$ in $\cC_2(\cP)$, if $\Ext^1_{\cC_2(\cP)}(X,Y)_Z\not=\varnothing$, there are isomorphisms $Y\cong X^*,Z\cong K_{X^1}\oplus K_{X^0}^*$ and 
$$\Ext^1_{\cC_2(\cP)}(X,X^*)_{K_{X^1}\oplus K_{X^0}^*}\cong \Aut_{\cK_2(\cP)}(X).$$
\end{lemma}
\begin{proof}
Since $\Ext^1_{\cC_2(\cP)}(X,Y)_Z\not=\varnothing$, by Lemma \ref{bijection between Ext and Hom}, we have 
$$\Ext^1_{\cC_2(\cP)}(X,Y)_Z\cong \Hom_{\cK_2(\cP)}(X,Y^*)_{Z^*}=\Hom_{\cK_2(\cP)}(X,Y^*)_0,$$
where $0$ is the zero object in $\cK_2(\cP)$. Note that for any $f\in \Hom_{\cC_2(\cP)}(X,Y^*)$, its mapping cone is homotopy equivalent to $0$ if and only if $f$ is a homotopy equivalence, equivalently, $f$ is an isomorphism, by Lemma \ref{isomorphism and homotopy equivalence}. Hence $X\cong Y^*, Y\cong X^*$. 

For convenience, we assume $Y=X^*$. For any $\xi\in \Ext^1_{\cC_2(\cP)}(X,Y)_Z$, there exists $f\in \Aut_{\cC_2(\cP)}(X)$ such that the homotopy class of $f$, as an element in $\Aut_{\cK_2(\cP)}(X)$, corresponds to $\xi$. By the proof of Lemma \ref{bijection between Ext and Hom}, we have 
\begin{align}\label{kernel}
Z\cong
(\begin{diagram}[midshaft,size=2em]
X^1\oplus X^0&\pile{\rTo^{\begin{pmatrix}\begin{smallmatrix}
d^1_X &0\\
-f^1 &-d^0_X
\end{smallmatrix}\end{pmatrix}}\\ \lTo_{\begin{pmatrix}\begin{smallmatrix}
d^0_X &0\\
-f^0 &-d^1_X
\end{smallmatrix}\end{pmatrix}}} &X^0\oplus X^1).
\end{diagram}
\end{align}
By Lemma \ref{contractible complex}, the contractible complex $Z\cong K_P\oplus K_Q^*$ for some $P,Q\in \cP$. Notice that $P\cong \Ker d_Z^0$ and $Q\cong \Ker d_Z^1$. By (\ref{kernel}), we have 
$$\ker d^0_Z\cong \{\begin{pmatrix}\begin{smallmatrix} x^0\\x^1\end{smallmatrix}\end{pmatrix}\in X^0\oplus X^1|d^0_X(x^0)=0,-d^1_X(x^1)=f^0(x^0)\}$$ 
which is isomorphic to $X^1$ via
\begin{align*}
X^1&\xrightarrow{\cong}\{\begin{pmatrix}\begin{smallmatrix} x^0\\x^1\end{smallmatrix}\end{pmatrix}\in X^0\oplus X^1|d^0_X(x^0)=0,-d^1_X(x^1)=f^0(x^0)\}\\
x^1&\mapsto \begin{pmatrix}\begin{smallmatrix} -(f^0)^{-1}d^1_X(x^1)\\x^1\end{smallmatrix}\end{pmatrix}.
\end{align*}
Similarly, we have $\ker d^1_Z\cong X^0$. Therefore, we have $Z\cong K_{X^1}\oplus K_{X^0}^*$, and 
\begin{align*}
\Ext^1_{\cC_2(\cP)}(X,X^*)_{K_{X^1}\oplus K_{X^0}^*}\cong &\Hom_{\cK_2(\cP)}(X,X)_{K_{X^1}\oplus K_{X^0}^*}
=\Aut_{\cK_2(\cP)}(X),
\end{align*}
as desired.
\end{proof}

\subsection{Lie algebra $\fg$}\label{Lie algebra spanned by contractible complexes and indecomposable radical complexes}\

In Subsection \ref{Classical limit}, we define a $\bbC_{-1}$-algebra $\cD\cH_t(\cA)_{\bbC_{-1}}$ and its classical limit $\cD\cH_{-1}(\cA)$. As a $\bbC$-associative algebra, $\cD\cH_{-1}(\cA)$ has a natural Lie bracket given by commutator.

\begin{definition}
(a) Define $\cD\cH^{\ind}_t(\cA)_{\bbC_{-1}}\subset \cD\cH_t(\cA)_{\bbC_{-1}}$ to be the $\bbC_{-1}$-submodule spanned by elements of the forms 
$$[[\cO/\rG_{\ue}]\hookrightarrow\cN^{\rad}\hookrightarrow\cM],\ b_{\alpha},\ 
\frac{b_{\alpha}-1}{-t-1},$$
where $\cO\subset \rP_2^{\rad}(A,\ue)$ is a $\rG_{\ue}$-invariant constructible subset consisting of points corresponding to indecomposable radical complexes, $\ue\in \bbN I\times \bbN I$ and $\alpha\in K(\cA)$.\\
(b) Define $\fg$ to be the $\bbC$-subspace of $\bbC_{-1}/(t+1)\bbC_{-1}\otimes_{\bbC_{-1}}\cD\cH^{\ind}_t(\cA)_{\bbC_{-1}}\subset \cD\cH_{-1}(\cA)$ spanned by elements of the form 
$$1\otimes [[\cO/\rG_{\ue}]\hookrightarrow\cN^{\rad}\hookrightarrow\cM],\ h_{\alpha}=1\otimes \frac{b_{\alpha}-1}{-t-1},$$
where $\cO\subset \rP_2^{\rad}(A,\ue)$ is a $\rG_{\ue}$-invariant constructible subset consisting of points corresponding to indecomposable radical complexes, $\ue\in \bbN I\times \bbN I$ and $\alpha\in K(\cA)$.\\
(c) Define $\fn$ to be the $\bbC$-subspace of $\fg$ spanned by elements of the form 
$$1\otimes [[\cO/\rG_{\ue}]\hookrightarrow\cN^{\rad}\hookrightarrow\cM],$$
where $\cO\subset \rP_2^{\rad}(A,\ue)$ is a $\rG_{\ue}$-invariant constructible subset consisting of points corresponding to indecomposable radical complexes and $\ue\in \bbN I\times \bbN I$.
\end{definition}

By definition, we have the decomposition
$$\fg=\fn\oplus \fh,$$
where $\fh$ is the $\bbC$-subspace spanned by $\{h_\alpha=1\otimes \frac{b_{\alpha}-1}{-t-1}|\alpha\in K(\cA)\}$ and $\fh\cong \bbC\otimes_{\bbZ}K(\cA)$, see Subsection \ref{Classical limit}.

\begin{theorem}\label{Lie algebra g}
The $\bbC$-subspace $\fg\subset \cD\cH_{-1}(\cA)$ is closed under the Lie bracket, and so it is a Lie subalgebra.
\end{theorem}

We deal with the following simple case firstly.
\begin{lemma}\label{simple}
Let $X,Y\in \cC_2(\cP)$ be indecomposable radical complexes, then $$[\overline{\delta_{[X]}},\overline{\delta_{[Y]}}]\in \fg.$$
\end{lemma}

\begin{lemma}\label{Upsilon(Aut)}
Let $X\in \cC_2(\cP)$ be an indecomposable radical complex, then 
$$\Upsilon(\Aut_{\cC_2(\cP)}(X))=(t^2-1)t^{2\dim \mbox{\rad}\End_{\cC_2(\cP)}(X)}.$$
Moreover, if $X_1,...,X_m\in \cC_2(\cP)$ are indecomposable radical complexes having $n$ many distinct isomorphism classes with sizes $m_1,...,m_n$ such that $m_1+...+m_n=m$, then 
$$\Upsilon(\Aut_{\cC_2(\cP)}(X_1\oplus...\oplus X_m))=t^{2l}\prod^m_{s=1}\Upsilon({\Aut_{\cC_2(\cP)}(X_s)})\cdot\prod^n_{i=1}\frac{t^{m_i(m_i-1)}\prod^{m_i}_{k=1}(t^{2k}-1)}{(t^2-1)^{m_i}}$$
for some $l\geqslant 0$. In particular, $(t^2-1)^2|\Upsilon(\Aut_{\cC_2(\cP)}(X_1\oplus...\oplus X_m))$ whenever $m\geqslant 2$.
\end{lemma}
\begin{proof}
Since $X$ is indecomposable, its endomorphism algebra $\End_{\cC_2(\cP)}(X)$ is local such that the radical $\rad \End_{\cC_2(\cP)}(X)$ consisting of nilpotent morphisms is an ideal of codimension $1$ and so 
\begin{equation}\label{AutX=bbC^*}
\begin{aligned}
\Aut_{\cC_2(\cP)}(X)&=\bbC^*\ltimes(1+\rad \End_{\cC_2(\cP)}(X)), \\
\Upsilon(\Aut_{\cC_2(\cP)}(X))&=(t^2-1)t^{2\dim \mbox{\rad}\End_{\cC_2(\cP)}(X)}.
\end{aligned}
\end{equation}
Moreover, by the proof of \cite[Proposition 4.14]{Joyce-2007} and \cite[Lemma 2.6]{Bridgeland-2012}, we have 
\begin{align*}
\Aut_{\cC_2(\cP)}(X_1\oplus...\oplus &X_m)/\Aut_{\cC_2(\cP)}(X_1)\!\times\!\!...\!\!\times\! \Aut_{\cC_2(\cP)}(X_m)\cong \bbC^l\!\times\! \prod^n_{i=1}(\textrm{GL}_{m_i}(\bbC)/{\bbC^*}^{m_i}),\\
&\Upsilon(\textrm{GL}_{m_i}(\bbC))=t^{m_i(m_i-1)}\prod^{m_i}_{k=1}(t^{2k}-1).
\end{align*}
The remain conclusions follow directly.
\end{proof}

\begin{proof}[Proof of Lemma \ref{simple}]
By Remark \ref{special case}, we have 
\begin{align*}
\delta_{[X]}\diamond \delta_{[Y]}&\!=\![[\Ext_{\cC_2(\cP)}^1(X,Y)/(\Aut_{\cC_2(\cP)}(X)\!\times\! \Aut_{\cC_2(\cP)}(Y))\!\!\ltimes\!\!\Hom_{\cC_2(\cP)}(X,Y)]\!\xrightarrow{\varphi}\!\cM],\\
\overline{\delta_{[X]}}*\overline{\delta_{[Y]}}&\!=\!\!1\!\otimes \![[\Ext_{\cC_2(\cP)}^1(X,Y)/(\Aut_{\cC_2(\cP)}(X)\!\times\! \Aut_{\cC_2(\cP)}(Y))\!\!\ltimes\!\!\Hom_{\cC_2(\cP)}(X,Y)]\!\xrightarrow{\varphi}\!\!\cM],
\end{align*}
where $\varphi$ corresponds to the map $\xi\mapsto [Z_\xi]$ such that $[Z_\xi]$ is the isomorphism class of the middle term of $\xi\in \Ext_{\cC_2(\cA)}^1(X,Y)$. By Lemma \ref{decomposition}, each $Z_\xi$ can be decomposed as $Z_\xi\cong (Z_\xi)_r\oplus (Z_\xi)_c$, where $(Z_\xi)_r$ is radical and $(Z_\xi)_c$ is contractible. We divide 
$$\Ext_{\cC_2(\cP)}^1(X,Y)=\bigsqcup^3_{i=0}E^1_i$$
into the disjoint union of constructible subsets according to $(Z_\xi)_r$, where $E^1_0$ consists of $\xi$ such that $(Z_\xi)_r=0$, $E^1_1$ consists of $\xi$ such that $(Z_\xi)_r$ is indecomposable, $E^1_2$ consists of $\xi$ such that $(Z_\xi)_r\cong X\oplus Y$, and $E^1_3$ is their complement, which are all invariant under the $(\Aut_{\cC_2(\cP)}(X)\times \Aut_{\cC_2(\cP)}(Y))\ltimes\Hom_{\cC_2(\cA)}(X,Y)$-action, then
\begin{align*}
\delta_{[X]}\diamond \delta_{[Y]}&=\sum_{i=0}^3[[E^1_i/(\Aut_{\cC_2(\cP)}(X)\times \Aut_{\cC_2(\cP)}(Y))\ltimes\Hom_{\cC_2(\cP)}(X,Y)]\xrightarrow{\varphi_i}\cM],\\
\overline{\delta_{[X]}}*\overline{\delta_{[Y]}}&=\sum_{i=0}^31\otimes[[E^1_i/(\Aut_{\cC_2(\cP)}(X)\times \Aut_{\cC_2(\cP)}(Y))\ltimes\Hom_{\cC_2(\cP)}(X,Y)]\xrightarrow{\varphi_i}\cM],
\end{align*}
where $\varphi_i$ is the restriction of $\varphi$. We discuss case by case as follows. 

\textbf{Case $(i=0)$.} The radical part $(Z_\xi)_r=0$ means $Z_\xi$ is contractible. By Lemma \ref{middle term contractible}, we have $Y\cong X^*,Z_\xi\cong K_{X^1}\oplus K_{X^0}^*,\Ext^1_{\cC_2(\cP)}(X,Y)_{Z_\xi}\cong \Aut_{\cK_2(\cP)} (X)$. So $E^1_0\not=\varnothing$ if and only if $X\cong Y^*$, in this case, $E_0^1\cong \Aut_{\cK_2(\cP)}(X)$. By Lemma \ref{automorphism groups coincide} and \ref{Upsilon(Aut)},
\begin{align*}
&[[E^1_0/(\Aut_{\cC_2(\cP)}(X)\times \Aut_{\cC_2(\cP)}(Y))\ltimes\Hom_{\cC_2(\cP)}(X,Y)]\xrightarrow{\varphi_0}\cM]\\
=&\delta_{[X][Y^*]}\Upsilon(\Aut_{\cK_2(\cP)}(X))\Upsilon((\Aut_{\cC_2(\cP)}(X)\times \Aut_{\cC_2(\cP)}(Y))\ltimes\Hom_{\cC_2(\cP)}(X,Y))^{-1}\\
&\Upsilon(\Aut_{\cC_2(\cA)}(Z_\xi))\delta_{[Z_\xi]}\\
=&\delta_{[X][Y^*]}\frac{t^{-2\dim \Hom_{\cC_2(\cP)}(X,Y)}}{(t^2-1)t^{2\dim \mbox{\rad} \End_{\cC_2(\cP)}(X)+2\dim \Htp(X,X)}}b_{\hat{X}^1-\hat{X}^0},\\
\end{align*}
\begin{align*}
&1\otimes [[E^1_0/(\Aut_{\cC_2(\cP)}(X)\times \Aut_{\cC_2(\cP)}(Y))\ltimes\Hom_{\cC_2(\cP)}(X,Y)]\xrightarrow{\varphi_0}\cM]\\
=&\delta_{[X][Y^*]}\otimes\frac{b_{\hat{X}^1-\hat{X}^0}}{t^2-1},
\end{align*}
where $\delta_{[X][Y^*]}$ is the Kronecker symbol.

\textbf{Case $(i=1)$.} For any $\xi\in E^1_2$, suppose the contractible part $(Z_\xi)_c\cong K_P\oplus K_Q^*$ for some $P,Q\in \cP$, by Lemma \ref{regular} and Remark \ref{special case}, we have
\begin{align*}
&b_{\hat{P}-\hat{Q}}\diamond \Upsilon(\Aut_{\cC_2(\cP)}((Z_\xi)_r))\delta_{[(Z_\xi)_r]}\\=&b_P*b_Q^*\diamond \Upsilon(\Aut_{\cC_2(\cP)}((Z_\xi)_r))\delta_{[(Z_\xi)_r]}\\
=&\Upsilon(\Aut_{\cC_2(\cP)}((Z_\xi)_c))\delta_{[(Z_\xi)_c]}\diamond \Upsilon(\Aut_{\cC_2(\cP)}((Z_\xi)_r))\delta_{[(Z_\xi)_r]}\\
=&t^{-2\dim \Hom_{\cC_2(\cP)}((Z_\xi)_r,(Z_\xi)_c)}\Upsilon(\Aut_{\cC_2(\cP)}(Z_\xi))\delta_{[Z_\xi]}.
\end{align*}
We subdivide 
$$E^1_1=\bigsqcup_{P,Q\in \cP} E^1_{1,P,Q}$$
into disjoint union of constructible subsets, where $E^1_{1,P,Q}$ consists of $\xi\in E^1_1$ such that $(Z_\xi)_c\cong K_P\oplus K_Q^*$, which are invariant under  $(\Aut_{\cC_2(\cP)}(X)\times \Aut_{\cC_2(\cP)}(Y))\ltimes\Hom_{\cC_2(\cP)}(X,Y)$-action. Note that $(Z_\xi)_r^j\in \cP$ satisfies $(Z_\xi)_r^j\oplus P\oplus Q\cong X^j\oplus Y^j$, and so the dimension of $\Hom_{\cC_2(\cP)}((Z_\xi)_r,(Z_\xi)_c)\cong\Hom_{\cA}(P,(Z_\xi)_r^1)\oplus \Hom_{\cA}(Q,(Z_\xi)_r^0)$ is constant for $\xi\in E^1_{1,P,Q}$, denoted by $n(P,Q)$. By Lemma \ref{property of h}, we have $\overline{b_{\hat{P}-\hat{Q}}}=1$ and 
\begin{align*}
&[[E^1_1/(\Aut_{\cC_2(\cP)}(X)\times \Aut_{\cC_2(\cP)}(Y))\ltimes\Hom_{\cC_2(\cP)}(X,Y)]\xrightarrow{\varphi_1}\cM]\\
=&\!\!\sum_{P,Q\in \cP}\!\! t^{2n(P,Q)}b_{\hat{P}-\hat{Q}}\!\diamond\![[E^1_{1,P,Q}/(\Aut_{\cC_2(\cP)}(X)\!\!\times\!\! \Aut_{\cC_2(\cP)}(Y))\!\!\ltimes\!\!\Hom_{\cC_2(\cP)}(X,Y)]\!\xrightarrow{r\varphi_{1,P,Q}}\!\!\cM],\\
&1\otimes [[E^1_1/(\Aut_{\cC_2(\cP)}(X)\times \Aut_{\cC_2(\cP)}(Y))\ltimes\Hom_{\cC_2(\cP)}(X,Y)]\xrightarrow{\varphi_1}\cM]\\
=&\sum_{P,Q\in \cP}1\otimes [[E^1_{1,P,Q}/(\Aut_{\cC_2(\cP)}(X)\!\times \!\Aut_{\cC_2(\cP)}(Y))\ltimes\Hom_{\cC_2(\cP)}(X,Y)]\xrightarrow{r\varphi_{1,P,Q}}\cM]\in \fn,
\end{align*}
where $\varphi_{1,P,Q}$ is the restriction of $\varphi_1$, and $r$ corresponds to the map $[Z_\xi]\mapsto [(Z_\xi)_r]$.

\textbf{Case $(i=2)$.} In this case, we have $Z_\xi=(Z_\xi)_r\cong X\oplus Y$ and $E^1_2=\{0\}$. By Lemma \ref{Upsilon(Aut)}, we have
\begin{align*}
&[[E^1_2/(\Aut_{\cC_2(\cP)}(X)\times \Aut_{\cC_2(\cP)}(Y))\ltimes\Hom_{\cC_2(\cP)}(X,Y)]\xrightarrow{\varphi_2}\cM]\\
=&\!\Upsilon(\Aut_{\cC_2(\cP)}(X))^{-1}\Upsilon(\Aut_{\cC_2(\cP)}(Y))^{-1}t^{-2\dim \Hom_{\cC_2(\cP)}(X,Y)}\Upsilon(\Aut_{\cC_2(\cP)} (X\oplus Y))\delta_{[X\oplus Y]}\\
=&\begin{cases}
t^{2l-2\dim \Hom_{\cC_2(\cP)}(X,Y)}\delta_{[X\oplus Y]},\ X\not\cong Y,\\
t^{2l-2\dim \Hom_{\cC_2(\cP)}(X,Y)+2}(t^2+1)\delta_{[X\oplus Y]},\ X\cong Y,
\end{cases}\\
&1\otimes [[E^1_2/(\Aut_{\cC_2(\cP)}(X)\times \Aut_{\cC_2(\cP)}(Y))\ltimes\Hom_{\cC_2(\cP)}(X,Y)]\xrightarrow{\varphi_2}\cM]\\
=&(1+\delta_{[X][Y]})\otimes \delta_{[X\oplus Y]}.
\end{align*}

\textbf{Case $(i=3)$.} By similar argument as in \textbf{Case $(i=1)$}, we have
\begin{align*}
&[[E^1_3/(\Aut_{\cC_2(\cP)}(X)\times \Aut_{\cC_2(\cP)}(Y))\ltimes\Hom_{\cC_2(\cP)}(X,Y)]\xrightarrow{\varphi_3}\cM]\\
=&\sum_{P,Q\in \cP}\!\! t^{2n(P,Q)}b_{\hat{P}-\hat{Q}}\!\diamond\![[E^1_{3,P,Q}/(\Aut_{\cC_2(\cP)}(X)\!\!\times\!\! \Aut_{\cC_2(\cP)}(Y))\!\!\ltimes\!\!\Hom_{\cC_2(\cP)}(X,Y)]\!\!\xrightarrow{r\varphi_{3,P,Q}}\!\!\cM]\\
=&\!\!\!\!\sum_{P,Q\in \cP}\!\!\!t^{2n(P,Q)\!-\!2\dim\! \Hom_{\cC_2(\cP)}(X,Y)}b_{\hat{P}-\hat{Q}}\!\diamond\! \!\Upsilon(\!\Aut_{\cC_2(\cP)}(X)\!\!\times\!\!\Aut_{\cC_2(\cP)}(Y)\!)^{-1}[\!E^1_{3,P,Q}\!\!\xrightarrow{r\varphi_{3,P,Q}q}\!\!\cM]
\end{align*}
where $q:E^1_{3,P,Q}\rightarrow [E^1_{3,P,Q}/(\Aut_{\cC_2(\cP)}(X)\times \Aut_{\cC_2(\cP)}(Y))\ltimes\Hom_{\cC_2(\cP)}(X,Y)]$ is the natural morphism. We claim that
\begin{equation}
(t^2-1)|\Upsilon(\Aut_{\cC_2(\cP)}(X)\times\Aut_{\cC_2(\cP)}(Y))^{-1}[E^1_{3,P,Q} \xrightarrow{r\varphi_{3,P,Q}q}\cM] \tag{$\star$}
\end{equation}
Indeed, by Lemma \ref{Upsilon(Aut)}, we have 
$$\Upsilon(\Aut_{\cC_2(\cP)}(X)\times \Aut_{\cC_2(\cP)}(Y))\!=\!(t^2-1)^2t^{2(\dim \mbox{\rad} \End_{\cC_2(\cP)}(X)+\dim \mbox{\rad}\End_{\cC_2(\cP)}(Y))},$$
so $(t^2-1)^{-2}$ appears in the right hand side of ($\star$). For any $\xi\in E^1_{3,P,Q}$, by Lemma \ref{Upsilon(Aut)}, we know that $(t^2-1)^2|\Upsilon(\Aut_{\cC_2(\cP)}(Z_\xi)_r)$. 
Furthermore, there is a free $\bbC^*$-action on $\Ext^1_{\cC_2(\cP)}(X,Y)_{Z_\xi}$, see the proof of \cite[Theorem 4]{Hubery-2006} or \cite[Corollary 5.8]{Chen-Deng-2015}, which implies that $(t^2-1)|\Upsilon(\Ext^1_{\cC_2(\cP)}(X,Y)_{Z_\xi})$. This finishes the proof of ($\star$), and so
$$1\otimes [[E^1_3/(\Aut_{\cC_2(\cP)}(X)\times \Aut_{\cC_2(\cP)}(Y))\ltimes\Hom_{\cC_2(\cP)}(X,Y)]\xrightarrow{\varphi_3}\cM]=0.$$

Hence we have
$$\overline{\delta_{[X]}}*\overline{\delta_{[Y]}}=\delta_{[X][Y^*]}\otimes\frac{b_{\hat{X}^1-\hat{X}^0}}{t^2-1}+\theta+(1+\delta_{[X][Y]})\otimes \delta_{[X\oplus Y]},$$
where $\theta\in \fn$. Similarly, we can calculate 
$$\overline{\delta_{[Y]}}*\overline{\delta_{[X]}}=\delta_{[Y][X^*]}1\otimes\frac{b_{\hat{Y^1}-\hat{Y^0}}}{t^2-1}+\theta'+(1+\delta_{[Y][X]})1\otimes \delta_{[Y\oplus X]},$$
where $\theta'\in \fn$. Therefore, 
\begin{align*}
&\overline{\delta_{[X]}}*\overline{\delta_{[Y]}}-\overline{\delta_{[Y]}}*\overline{\delta_{[X]}}\\
=&\theta-\theta'+\delta_{[X][Y^*]}1\otimes \frac{b_{\hat{X}^1-\hat{X}^0}-b_{\hat{X}^0-\hat{X}^1}}{t^2-1}\\
=&\theta-\theta'+\delta_{[X][Y^*]}1\otimes (\frac{-b_{\hat{X}^1-\hat{X}^0}-1}{-t+1}*\frac{b_{\hat{X}^0-\hat{X}^1}-1}{-t-1})\\
=&\theta-\theta'+\delta_{[X][Y^*]}\frac{-1-1}{1+1}h_{\hat{X}^0-\hat{X}^1}\\
=&\theta-\theta'-\delta_{[X][Y^*]}h_{\hat{X}^0-\hat{X}^1}\in \fg,
\end{align*}
where we use $b_\alpha*b_{-\alpha}=1,\overline{b_\alpha}=1\otimes b_\alpha=1$ for any $\alpha\in K(\cA)$, see Lemma \ref{property of h}.
\end{proof}

\begin{proof}[Proof of Theorem \ref{Lie algebra g}]
For any $\alpha,\beta\in K(\cA)$ and any 
\begin{align*}
a=[[\cO/\rG_{\ue}]\hookrightarrow \cN^{\rad}\hookrightarrow \cM],\ b=[[\cO'/\rG_{\ue'}]\hookrightarrow \cN^{\rad}\hookrightarrow \cM],
\end{align*}
where $\cO\subset \rP_2^{\rad}(A,\ue),\cO'\subset \rP_2^{\rad}(A,\ue')$ are $\rG_{\ue}$-invariant, $\rG_{\ue'}$-invariant, respectively, constructible subsets consisting of points corresponding to indecomposable radical complexes, by Lemma \ref{property of h}, we have
$$[h_\alpha,h_\beta]=0\in \fg,\ [h_\alpha,\overline{a}]=(\alpha,\hat{P}^0-\hat{P}^1)\overline{a}\in \fg,$$
where $P^j=\bigoplus_{i\in I}e^j_iP_i$. It remains to prove $[\overline{a},\overline{b}]\in \fg$. Our proof is essentially the same as the proof of Lemma \ref{simple}. By Proposition \ref{Motivic Riedtmann-Peng}, we have
\begin{align*}
a\diamond b&=\sum_{m\in M}[[V_m\times E^1_m/G_m\ltimes E^0_m]\xrightarrow{\varphi_m}\cM],\\
\overline{a}*\overline{b}&=\sum_{m\in M}1\otimes [[V_m\times E^1_m/G_m\ltimes E^0_m]\xrightarrow{\varphi_m}\cM],
\end{align*}
where $V_m\subset \cO\times \cO'$ is a constructible subset such that if $v\in V_m$ corresponds to two indecomposable radical complexes $X,Y$, then $\Stab_{G_m}(v)\!\cong\! \Aut_{\cC_2(\cP)}(X)\!\times\! \Aut_{\cC_2(\cP)}(Y)$ and the dimensions of $E^1_m\cong \Ext^1_{\cC_2(\cP)}(X,Y), E^0_m\cong \Hom_{\cC_2(\cP)}(X,Y)$ are constant for various $v\in V_m$. Under these identifications, the morphism $\varphi_m$ corresponds to the map $(v,\xi)\mapsto [Z_\xi]$, where $[Z_\xi]$ is the isomorphism class of the middle term of $\xi\in E^1_m\cong\Ext^1_{\cC_2(\cP)}(X,Y)$. We divide 
$$V_m\times E^1_m=\bigsqcup^3_{i=0}(V_m\times E^1_m)_i$$
into disjoint union of constructible subsets, where $(V_m\times E^1_m)_0$ consists of $(v,\xi)$ such that $(Z_\xi)_r=0$, $(V_m\times E^1_m)_1$ consists of $(v,\xi)$ such that $(Z_\xi)_r$ is indecomposable, $(V_m\times E^1_m)_2$ consists of $(v,\xi)$ such that $(Z_\xi)_r\cong X\oplus Y$, and $(V_m\times E^1_m)_3$ is their complement, which are all invariant under the $G_m\ltimes E^0_m$-action, and then
\begin{align*}
a\diamond b&=\sum_{m\in M}\sum_{i=0}^3[[(V_m\times E^1_m)_i/G_m\ltimes E^0_m]\xrightarrow{(\varphi_m)_i}\cM],\\
\overline{a}*\overline{b}&=\sum_{m\in M}\sum_{i=0}^31\otimes [[(V_m\times E^1_m)_i/G_m\ltimes E^0_m]\xrightarrow{(\varphi_m)_i}\cM],
\end{align*}
where $(\varphi_m)_i$ is the restriction of $\varphi_m$. We discuss case by case as follows.

\textbf{Case $(i=0)$'.} The radical part $(Z_\xi)_r=0$ means $Z_\xi$ is contractible. By Lemma \ref{middle term contractible}, we have $Y\cong X^*$ and $\Ext^1_{\cC_2(\cA)}(X,Y)_{Z_\xi}\cong \Aut_{\cK_2(\cP)} (X)$. Moreover, the middle term $Z_\xi\cong K_{X^1}\oplus K_{X^0}^*= K_{P^1}\oplus K_{P^0}^*$ is fixed, where $P^j=\bigoplus_{i\in I}e^j_iP_i$. Thus, 
$$[(V_m\times E^1_m)_0/G_m\ltimes E^0_m]\cong [(\cO\cap\cO'^*)_m/\rG_{\ue}]\times [(E^1_m)_0/G_m\ltimes E^0_m],$$
where $\cO'^*=\{(y^0,y^1,-d^0,-d^1)|(y^1,y^0,d^1,d^0)\in \cO'\}$, $(\cO\cap\cO'^*)_m\subset \cO\cap\cO'^*$ is a constructible subset consisting of points corresponding to $X$ if there exists $v\in V_m$ corresponds to $(X,X^*)$, and $(E^1_m)_0\cong \Ext^1_{\cC_2(\cA)}(X,X^*)_{Z_\xi}$ is a constructible subset of $E^1_m\cong \Ext^1_{\cC_2(\cA)}(X,X^*)$. Moreover, by relation (iii) in Definition \ref{relation3}, we may reduce to the \textbf{Case $(i=0)$} in the proof of Lemma \ref{simple}, and then
\begin{align*}
&\sum_{m\in M}[[(V_m\times E^1_m)_0/G_m\ltimes E^0_m]\xrightarrow{(\varphi_m)_0}\cM]\\
=&\sum_{m\in M}[[(\cO\cap\cO'^*)_m/\rG_{\ue}]\times [(E^1_m)_0/G_m\ltimes E^0_m]\xrightarrow{\varphi_0\pi_2}\cM]\\
=&\sum_{m\in M}\Upsilon([(\cO\cap\cO'^*)_m/\rG_{\ue}]) [[(E^1_m)_0/G_m\ltimes E^0_m]\xrightarrow{\varphi_0}\cM],\\
&1\otimes \sum_{m\in M}[[(V_m\times E^1_m)_0/G_m\ltimes E^0_m]\xrightarrow{(\varphi_m)_0}\cM]\\
=&\sum_{m\in M}1\otimes \Upsilon([(\cO\cap\cO'^*)_m/\rG_{\ue}]) \frac{b_{\hat{P}^1-\hat{P}^0}}{t^2-1}\\
=&1\otimes \Upsilon([(\cO\cap\cO'^*)/\rG_{\ue}]) \frac{b_{\hat{P}^1-\hat{P}^0}}{t^2-1},
\end{align*}
where $\pi_2$ is the natural second projection.

\textbf{Case $(i=1)$'.} By the same argument as in \textbf{Case $(i=1)$} in the proof of Lemma \ref{simple}, we subdivide 
$$(V_m\times E^1_m)_1=\bigsqcup_{P,Q\in \cP}(V_m\times E^1_m)_{1,P,Q}$$ 
into constructible subsets, where $(V_m\times E^1_m)_{1,P,Q}$ consists of $(v,\xi)$ such that $(Z_\xi)_c\cong K_P\oplus K_Q^*$, which are invariant under $G_m\ltimes E^0_m$-action, then we have 
\begin{align*}
&\sum_{m\in M}[[(V_m\times E^1_m)_1/G_m\ltimes E^0_m]\xrightarrow{(\varphi_m)_1}\cM]\\
=&\sum_{m\in M}\sum_{P,Q\in \cP}t^{2n(P,Q)}b_{\hat{P}-\hat{Q}}\diamond [[(V_m\times E^1_m)_{1,P,Q}/G_m\ltimes E^0_m]\xrightarrow{r(\varphi_m)_{1,P,Q}}\cM],\\
&1\otimes \sum_{m\in M}[[(V_m\times E^1_m)_1/G_m\ltimes E^0_m]\xrightarrow{(\varphi_m)_1}\cM]\\
=&\sum_{m\in M}\sum_{P,Q\in \cP}1\otimes [[(V_m\times E^1_m)_{1,P,Q}/G_m\ltimes E^0_m]\xrightarrow{r(\varphi_m)_{1,P,Q}}\cM]\in \fn,
\end{align*}
where $(\varphi_m)_{1,P,Q}$ is the restriction of $(\varphi_m)_1$.

\textbf{Case $(i=2)$'.} In this case, $(V_m\times E^1_m)_2$ consists of $(v,\xi)$ such that $Z_\xi= (Z_\xi)_r\cong X\oplus Y$, and so $\xi=0$. Similarly, in the calculations of $b\diamond a$ and $\overline{b}*\overline{a}$, we also have $[(V_m'\times E'^1_m)_2/G'_m\ltimes E'^0_m\xrightarrow{(\varphi'_m)_2}\cM]$ and $1\otimes [(V_m'\times E'^1_m)_2/G'_m\ltimes E'^0_m\xrightarrow{(\varphi'_m)_2}\cM]$ in the summations. It is clear that 
\begin{align*}
[(V_m'\times E'^1_m)_2/G'_m\ltimes E'^0_m\xrightarrow{(\varphi'_m)_2}\cM]&=[(V_m\times E^1_m)_2/G_m\ltimes E^0_m\xrightarrow{(\varphi_m)_2}\cM],\\
1\otimes [(V_m'\times E'^1_m)_2/G'_m\ltimes E'^0_m\xrightarrow{(\varphi'_m)_2}\cM]&=1\otimes [(V_m\times E^1_m)_2/G_m\ltimes E^0_m\xrightarrow{(\varphi_m)_2}\cM].
\end{align*}
Hence the contributions of these terms cancel out in the commutator $[\overline{a},\overline{b}]$.

\textbf{Case $(i=3)$'.} By similar argument as in \textbf{Case $(i=1)$'}, we have
\begin{align*}
&\sum_{m\in M}[[(V_m\times E^1_m)_3/G_m\ltimes E^0_m]\xrightarrow{(\varphi_m)_3}\cM]\\
=&\sum_{m\in M}\sum_{P,Q\in \cP}t^{2n(P,Q)}b_{\hat{P}-\hat{Q}}\diamond [[[(V_m\times E^1_m)_{3,P,Q}/G_m\ltimes E^0_m]\xrightarrow{r(\varphi_m)_3}\cM]]
\end{align*}
For any $(v,\xi)\in (V_m\times E^1_m)_{3,P,Q}$, by similar argument as in \textbf{Case $(i=3)$} in the proof of Lemma \ref{simple}, we know that 
$$(t^2-1)|[[(E^1_m)_{3,P,Q}/(\Aut_{\cC_2(\cP)}(X)\times \Aut_{\cC_2(\cP)}(Y))\ltimes \Hom_{\cC_2(\cA)}(X,Y)]\xrightarrow{r(\varphi_m)_3}\cM],$$
and so 
$$1\otimes \sum_{m\in M}[[(V_m\times E^1_m)_3/G_m\ltimes E^0_m]\xrightarrow{(\varphi_m)_3}\cM]=0.$$

Similarly, we may calculate $b\diamond a$ and $\overline{b}*\overline{a}$, and then 
$$\overline{a}*\overline{b}-\overline{b}*\overline{a}=\Theta-\Theta'+\Upsilon([(\cO\cap\cO'^*)/\rG_{\ue}])1\otimes \frac{b_{\hat{P}^1-\hat{P}^0}-b_{\hat{P}^0-\hat{P}^1}}{t^2-1},$$
where
\begin{align*}
\Theta&=\sum_{m\in M}\sum_{P,Q\in \cP}1\otimes [[(V_m\times E^1_m)_{1,P,Q}/G_m\ltimes E^0_m]\xrightarrow{r(\varphi_m)_1}\cM]\in \fn,\\
\Theta'&=\sum_{m\in M}\sum_{P,Q\in \cP}1\otimes [[(V'_m\times E'^1_m)_{1,P,Q}/G'_m\ltimes E'^0_m]\xrightarrow{r(\varphi'_m)_1}\cM]\in \fn.
\end{align*}
By $\Upsilon([(\cO\cap\cO'^*)/\rG_{\ue}])|_{t=-1}=\chi([(\cO\cap\cO'^*)/\rG_{\ue}])$ and Lemma \ref{property of h}, we finally obtain
\begin{align*}
[\overline{a},\overline{b}]=&\Theta-\Theta'+\chi([(\cO\cap\cO'^*)/\rG_{\ue}])1\otimes (\frac{-b_{\hat{P}^1-\hat{P}^0}-1}{-t+1}*\frac{b_{\hat{P}^0-\hat{P}^1}}{-t-1})\\
=&\Theta-\Theta'-\chi([(\cO\cap\cO'^*)/\rG_{\ue}])h_{\hat{P}^0-\hat{P}^1}\in \fg,
\end{align*}
as desired.
\end{proof}

For any $\ue\in\bbN I\times \bbN I$ and $\rG_{\ue}$-invariant constructible subset  $\cO\subset \rP_2^{\rad}(A,\ue)$ which consists of points corresponding to indecomposable radical complexes, we denote by 
$$1_{[\cO/\rG_{\ue}]}=[[\cO/\rG_{\ue}]\hookrightarrow \cN\hookrightarrow \cM]\in \cD\cH^{\ind}_t(\cA)_{\bbC_{-1}},\ \overline{1}_{[\cO/\rG_{\ue}]}=1\otimes 1_{[\cO/\rG_{\ue}]}\in \fn$$
the characteristic stack function of the quotient stack $[\cO/\rG_{\ue}]$ and 
its image in the classical limit. In the end of this subsection, we summarize the formulas about the Lie bracket on $\fg$ from Lemma \ref{property of h} and the proof of Theorem \ref{Lie algebra g} as follows,
\begin{equation}
\begin{aligned}\label{Lie algebra g formula}
&[\overline{1}_{[\cO/\rG_{\ue}]},\overline{1}_{[\cO'/\rG_{\ue'}]}]=[\overline{1}_{[\cO/\rG_{\ue}]},\overline{1}_{[\cO'/\rG_{\ue'}]}]_{\fn}-\chi([(\cO\cap\cO'^*)/\rG_{\ue}])h_\alpha,\\
&[\overline{1}_{[\cO/\rG_{\ue}]},\overline{1}_{[\cO'/\rG_{\ue'}]}]_{\fn}\in \fn,\\
&[h_\alpha,\overline{1}_{[\cO'/\rG_{\ue'}]}]=(\alpha, \hat{P}'^0-\hat{P}'^1)\overline{1}_{[\cO'/\rG_{\ue'}]},\\
&[\overline{1}_{[\cO'/\rG_{\ue'}]},h_\alpha]=-[h_\alpha,\overline{1}_{[\cO'/\rG_{\ue'}]}],\\
&[h_\alpha,h_{\alpha'}]=0,
\end{aligned}
\end{equation}
where $P^j=\bigoplus_{i\in I}e^j_iP_i,P'^j=\bigoplus_{i\in I}e'^j_iP_i$ and $\alpha=\hat{P}^0-\hat{P}^1,\alpha'\in K(\cA)$.

\section{Lie algebra arising from $\cK_2(\cP)$}

This section is essentially a revise of \cite{Xiao-Xu-Zhang-2006}. We add the $*$-saturated condition to refine their definition for $\rP_2(A,\bd)$ in Subsection \ref{Moduli K_2(P)} and rewrite a self-contained proof for Theorem \ref{Lie algebra g_2} in this section. Our proof is a reformulation and a refinement of Xiao-Xu-Zhang \cite{Xiao-Xu-Zhang-2006}, and the basic idea is the same as Peng-Xiao \cite{Peng-Xiao-2000}.

\subsection{Support-bounded constructible function}\label{Support-bounded constructible function}\

In Subsection \ref{Moduli K_2(P)}, we define an ind-limit $\rP_2(A,\bd)=\varinjlim_{\ue\in \udim^{-1}(\bd)}\rP^*_2(A,\ue)$ with an ind-limit $\rG_{\bd}$-action for any $\bd\in K_0$.

For any $\bd\in K_0$ and $\ue\in \udim^{-1}(\bd)$, we denote by $M_{\rG_{\ue}}(\rP_2^*(A,\ue))$ the $\bbC$-vector space of $\rG_{\ue}$-invariant constructible functions on $\rP_2^*(A,\ue)$, and denote by $M_{\rG_{\bd}}(\rP_2(A,\bd))$ the $\bbC$-vector space of $\rG_{\bd}$-invariant constructible functions on $\rP_2(A,\bd)$. Then for any $\ue,\ue'\in \udim^{-1}(\bd)$ satisfying $\ue\leqslant \ue'$, the morphism $t_{\ue\ue'}:\rP_2^*(A,\ue)\rightarrow \rP_2^*(A,\ue')$, see Subsection \ref{Moduli K_2(P)}, induces a $\bbC$-linear map
\begin{align*}
t_{\ue\ue'}^*:M_{\rG_{\ue'}}(\rP_2^*(A,\ue'))\rightarrow M_{\rG_{\ue}}(\rP_2^*(A,\ue))\\
f\mapsto (x\mapsto f(t_{\ue\ue'}(x))),
\end{align*}
such that $\{M_{\rG_{\ue}}(\rP_2^*(A,\ue))\}_{\ue\in \udim^{-1}(\bd)}$ is an inverse system, and 
$$M_{\rG_{\bd}}(\rP_2(A,\bd))=\varprojlim_{\ue\in \udim^{-1}(\bd)}M_{\rG_{\ue}}(\rP_2^*(A,\ue)).$$

A subset $\hat\cO\subset \rP_2(A,\bd)$ is said to be $\rG_{\bd}$-invariant support-bounded constructible, if there exists a $\rG_{\ue}$-invariant constructible subset $\cO_{\ue}\subset\rP_2^*(A,\ue)$ for some $\ue\in \udim^{-1}(\bd)$ such that $\hat\cO=\rG_{\bd}.t_{\ue}(\cO_{\ue})$, where $t_{\ue}:\rP_2^*(A,\ue)\rightarrow \rP_2(A,\bd)$ is the natural morphism, see Subsection \ref{Moduli K_2(P)}. In this case, we denote by $\cO_{\ue'}=\rG_{\ue'}.t_{\ue\ue'}(\cO_{\ue})$ for any $\ue'\in \udim^{-1}(\bd)$ satisfying $\ue\leqslant \ue'$, then $\hat\cO=\rG_{\bd}.\varinjlim_{\ue'\in \udim^{-1}(\bd), \ue\leqslant \ue'}\cO_{\ue'}$.

For a $\rG_{\bd}$-invariant support-bounded constructible subset $\hat\cO\subset \rP_2(A,\bd)$, we denote by $1_{\hat\cO}$ the characteristic function of $\hat\cO$. Then each element in $M_{\rG_{\bd}}(\rP_2(A,\bd))$ is of the form $\sum^n_{i=1}c_i1_{\hat\cO_i}$, where $c_i\in \bbC$ and $\hat\cO_i\subset \rP_2(A,\bd)$ is a $\rG_{\bd}$-invariant support-bounded constructible subset.

Recall that in Subsection \ref{Derived invariance 2}, for another $\bbC$-algebra $B$ such that $\cD^b(A)\simeq\cD^b(B)$, we construct a morphism $\Psi:\rP_2(A,\bd_A)\rightarrow \rP_2(B,\bd_B)$ which induces an isomorphism $$[\rP_2(A,\bd_A)/\rG_{\bd_A}]\cong[\rP_2(B,\bd_B)/\rG_{\bd_B}].$$
By definition of $\Psi$, for any $\rG_{\bd_A}$-invariant support-bounded constructible subset $\hat\cO\subset \rP_2(A,\bd_A)$, its image $\Psi(\hat{\cO})\subset \rP_2(B,\bd_B)$ is a finite disjoint union of $\rG_{\bd_B}$-invariant support-bounded constructible subsets. Moreover, $\Psi$ induces an isomorphism
$$M_{\rG_{\bd_A}}(\rP_2(A,\bd_A))\cong M_{\rG_{\bd_B}}(\rP_2(B,\bd_B)).$$

\subsection{Convolution}\label{convolution}\

In this subsection, we define a bilinear map
$$-*-:M_{\rG_{\bd'}}(\rP_2(A,\bd'))\times M_{\rG_{\bd''}}(\rP_2(A,\bd''))\rightarrow M_{\rG_{\bd}}(\rP_2(A,\bd))$$
for any $\bd=\bd'+\bd''\in K_0$.

For convenience, from now on, we identify any point $(x^1,x^0,d^1,d^0)\in \rP_2^*(A,\ue)$ with the corresponding complex $X=(M(x^1),M(x^0),d^1,d^0)\in\cC_2(\cP)$. Moreover, for any $\ue\in \udim^{-1}(\bd)$, we identify the natural morphism $t_{\ue}:\rP_2^*(A,\ue)\rightarrow \rP_2(A,\bd)$ with the map $X\mapsto \tilde{X}=\{X\oplus K_P\oplus K_P^*|P\in \cP\}$.

For any $\rG_{\ue'}$-invariant constructible subset $\cO'\subset \rP_2^*(A,\ue')$ and any $\rG_{\ue''}$-invariant constructible subset $\cO''\subset \rP_2^*(A,\ue'')$, where $\ue'\in \udim^{-1}(\bd'),\ue''\in \udim^{-1}(\bd'')$, and for any $\tilde{Z}\in \rP_2(A,\bd)$, we take a representative $Z\in \tilde{Z}$ in the homotopy equivalence class, and  define a constructible set $W(\cO',\cO'';Z)$ consisting of triples of morphisms $(f,g,h)$ in $\cK_2(\cP)$ such that
\begin{equation}
Y\xrightarrow{f}Z\xrightarrow{g}X\xrightarrow{h}Y^* \tag{$\triangle$}
\end{equation}
is a distinguished triangle, where $X\in \cO',Y\in \cO''$. Indeed, it is a bundle over $\cO'\times \cO''$, whose fibers are constructible sets. More precisely, its fiber at $(X,Y)\in \cO'\times \cO''$ is a subset 
\begin{equation}\label{W(X,Y;Z)}
W(X,Y;Z)\subset\Hom_{\cK_2(\cP)}(Y,Z)\times \Hom_{\cK_2(\cP)}(Z,X)\times \Hom_{\cK_2(\cP)}(X,Y^*)
\end{equation}
consisting of $(f,g,h)$ such that $(\triangle)$ is a distinguished triangle. Since distinguished triangles of the form $(\triangle)$ are isomorphic to the rotation of the  standard triangle given by mapping cone of $h$, the constructibility of $W(X,Y;Z)$ is equivalent to the constructibility of the subset $\Hom_{\cK_2(\cP)}(X,Y^*)_{Z^*}\subset \Hom_{\cK_2(\cP)}(X,Y^*)$. The later is given by \cite[Proposition 9]{Caldero-Keller-2008}. 

The group $\rG_{\ue'}\times \rG_{\ue''}$ acts on $W(\cO',\cO'';Z)$ as follows. For any $(\alpha,\beta)\in \rG_{\ue'}\times \rG_{\ue''}$ and $(f,g,h)\in W(\cO',\cO'';Z)$, suppose $(f,g,h)\in W(X,Y;Z)$, then 
$$(\alpha,\beta).(f,g,h)=(f\beta^{-1},\alpha g,\beta^*h\alpha^{-1})\in W(\alpha.X,\beta.Y;Z),$$
where $\alpha,\beta$ induce isomorphisms $X\rightarrow\alpha.X,Y\rightarrow\beta.Y$ in $\cC_2(\cP)$ which are homotopy equivalences, that is, their homotopy classes are isomorphisms in $\cK_2(\cP)$, still denoted by $\alpha,\beta$, such that $(f\beta^{-1},\alpha g,\beta^*h\alpha^{-1})$ form a distinguished triangle. Indeed, there is a commutative diagram
\begin{diagram}[midshaft,size=2em]
Y &\rTo^{f} &Z &\rTo^{g} &X &\rTo^{h} &Y^*\\
\dTo^{\beta} & &\vEq &&\dTo^{\alpha} &&\dTo^{\beta^*}\\
\beta.Y &\rTo^{f\beta^{-1}} &Z &\rTo^{\alpha g} &\alpha.X &\rTo^{\beta^*h\alpha^{-1}} &(\beta.Y)^*.
\end{diagram}

By Subsection \ref{naive Euler characteristic}, the naive Euler characteristic of
$[W(\cO',\cO'';Z)/\rG_{\ue'}\times \rG_{\ue''}](\bbC)$
has been defined. We define $V(\cO',\cO'';Z)=[W(\cO',\cO'';Z)/\rG_{\ue'}\times \rG_{\ue''}]$ and 
\begin{equation}\label{constant F}
F^Z_{\cO'\cO''}=\chi(V(\cO',\cO'';Z))=\chi^{\textrm{na}}([W(\cO',\cO'';Z)/\rG_{\ue'}\times \rG_{\ue''}](\bbC)).
\end{equation}

Note that $W(\cO',\cO'';Z)$ consists of triples of morphisms in $\cK_2(\cP)$ and contractible complexes $K_P\oplus K_P^*$ are zero objects in $\cK_2(\cP)$ for any $P\in\cP$, thus 
\begin{align*}
W(\cO',\cO'';Z)&=W(\cO',\cO'';Z_r\oplus K_P\oplus K_P^*),\\
V(\cO',\cO'';Z)&=V(\cO',\cO'';Z_r\oplus K_P\oplus K_P),\\
F^Z_{\cO'\cO''}&=F^{Z_r\oplus K_P\oplus K_P^*}_{\cO'\cO''}
\end{align*}
are independent of the choices of $Z\in \tilde{Z}$. Thus there is a well-defined function 
\begin{align*}
1_{\cO'}*1_{\cO''}:\rP_2(A,\bd)&\rightarrow \bbC\\
\tilde{Z}&\mapsto F^Z_{\cO'\cO''}.
\end{align*}
Similarly, for $\ue'_1\in \udim^{-1}(\bd'),\ue''_1\in \udim^{-1}(\bd'')$ satisfying $\ue'\leqslant \ue'_1,\ue''\leqslant \ue''_1$, we have 
\begin{equation}
\begin{aligned}\label{well-defined}
W(\cO',\cO'';Z)&=W(t_{\ue'\ue'_1}(\cO'),t_{\ue''\ue''_1}(\cO'');Z),\\
V(\cO',\cO'';Z)&=V(t_{\ue'\ue'_1}(\cO'),t_{\ue''\ue''_1}(\cO'');Z),\\
F^Z_{\cO'\cO''}&=F^Z_{t_{\ue'\ue'_1}(\cO')t_{\ue''\ue''_1}(\cO'')}.
\end{aligned}
\end{equation}

Therefore, for $\rG_{\bd'}$-invariant support-bounded constructible subset $\hat\cO'\subset \rP_2(A,\bd')$ and $\rG_{\bd''}$-invariant support-bounded constructible subset $\hat\cO''\subset \rP_2(A,\bd'')$, suppose $\hat\cO'=\rG_{\bd'}.t_{\ue'}(\cO_{\ue'})$ and $\hat\cO''=\rG_{\bd''}.t_{\ue''}(\cO_{\ue''})$, where $\cO'\subset \rP^*_2(A,\ue')$ is a $\rG_{\ue}$-invariant constructible subset, $\cO''\subset \rP^*_2(A,\ue'')$ is a $\rG_{\ue'}$-invariant constructible subset, and $\ue'\in \udim^{-1}(\bd'),\ue''\in \udim^{-1}(\bd'')$, the following 
\begin{align*}
W(\hat{\cO}',\hat{\cO}'',\tilde{Z})&=W(\cO',\cO'';Z)\\
V(\hat{\cO}',\hat{\cO}'',\tilde{Z})&=V(\cO',\cO'';Z)
\end{align*}
are independent of the choices $\ue'\in \udim^{-1}(\bd'),\ue''\in \udim^{-1}(\bd'')$, and 
there is a well-defined function
\begin{align*}
1_{\hat\cO'}*1_{\hat\cO''}:\rP_2(A,\bd)&\rightarrow \bbC\\
\tilde{Z}&\mapsto F^Z_{\cO'\cO''}.
\end{align*}
Obviously, the definition completely depends on the triangulated category structure of $\cK_2(\cP)$ which is determined by $\cD^b(A)$, see Subsection \ref{derived invariance}.

\begin{remark}\label{reduce to radical}
We make more remark on $(\ref{well-defined})$. For any fixed $(f,g,h)\in W(\cO',\cO'';Z)$, suppose $(f,g,h)\in W(X,Y;Z)$ of the form $(\triangle)$, there is a distinguished triangle $Y_r\xrightarrow{f}Z\xrightarrow{g}X_r\xrightarrow{h}(Y_r)^*$ such that $(f,g,h)\in W(\cO'_r,\cO''_r;Z)$ for some constructible subsets $\cO'_r\subset \cP_2^*(A,\ue'_r),\cO''_r\subset\rP_2^*(A,\ue''_r)$, where $\ue'_r,\ue''_r$ are projective dimension vector pairs of radical complexes $X_r,Y_r$ respectively. The contribution of $\rG_{\ue'}\times \rG_{\ue''}.(f,g,h)$ is the same as the contribution of $\rG_{\ue'_r}\times \rG_{\ue''_r}.(f,g,h)$ in $1_{\hat\cO'}*1_{\hat\cO''}(\tilde{Z})$, since these two orbits are the same subsets of $W(\cO',\cO'';Z)$. In other word, when we calculate the orbit $\rG_{\ue'}\times \rG_{\ue''}.(f,g,h)$ or its stabilizer, we may assume $X,Y$ are radical complexes.
\end{remark}

\begin{lemma}\label{convolution constructible}
The function $1_{\hat\cO'}*1_{\hat\cO''}$ belongs to $M_{\rG_{\bd}}(\rP_2(A,\bd))$.
\end{lemma}
\begin{proof}
Consider the constructible set $$W(\hat\cO',\hat\cO'')=\{(\tilde{Z},(f,g,h))|\tilde{Z}\in \rP_2(A,\bd),(f,g,h)\in W(\hat\cO',\hat\cO'';\tilde{Z})\}$$
which has a $\rG_{\ue'}\times \rG_{\ue''}$-action given by $(\alpha,\beta).(Z,(f,g,h))=(Z,(\alpha,\beta).(f,g,h))$, where $\ue'\in \udim^{-1}(\bd'),\ue''\in \udim^{-1}(\bd'')$, then the natural projection $W(\hat\cO',\hat\cO'')\rightarrow \rP_2(A,\bd)$ induces a morphism
\begin{align*}
\rho:[W(\cO',\cO'')/\rG_{\ue'}\times \rG_{\ue''}]\rightarrow \rP_2(A,\bd).
\end{align*}
By definition and \cite[Theorem 4.9]{Joyce-2006}, $1_{\hat\cO'}*1_{\hat\cO''}=\textrm{CF}(\rho)(1_{[W(\cO',\cO'')/\rG_{\ue'}\times \rG_{\ue''}]})$ is a constructible function on $\rP_2(A,\bd)$. For any $Z\in \rP_2^*(A,\ue)$, where $\ue\in \udim^{-1}(\bd)$, and any $\gamma\in \rG_{\ue}$, notice that $Y\xrightarrow{f}Z\xrightarrow{g}X\xrightarrow{h}Y^*$ is a distinguished triangle if and only if $Y\xrightarrow{\gamma f}\gamma.Z\xrightarrow{g\gamma^{-1}}X\xrightarrow{h}Y^*$ is a distinguished triangle, and so $W(\cO',\cO'';Z)\cong W(\cO',\cO'';\gamma.Z)$. Moreover, there is an isomorphism $V(\cO',\cO'';Z)\cong V(\cO',\cO'';\gamma.Z)$ such that $F^Z_{\cO'\cO''}=F^{\gamma.Z}_{\cO',\cO''}$. Therefore, $1_{\hat\cO'}*1_{\hat\cO''}$ is $\rG_{\bd}$-invariant.
\end{proof}

Extending the map $(1_{\hat\cO'},1_{\hat\cO''})\mapsto 1_{\hat\cO'}*1_{\hat\cO''}$ linearly, we obtain a bilinear map
$$-*-:M_{\rG_{\bd'}}(\rP_2(A,\bd'))\times M_{\rG_{\bd''}}(\rP_2(A,\bd''))\rightarrow M_{\rG_{\bd}}(\rP_2(A,\bd)).$$
By using of the integration notation, see Definition \ref{integration}, we can express it explicitly. For any $\hat{f}'\in M_{\rG_{\bd'}}(\rP_2(A,\bd')),\hat{f}''\in M_{\rG_{\bd''}}(\rP_2(A,\bd''))$ and $\tilde{Z}\in M_{\rG_{\bd}}(\rP_2(A,\bd))$ , there is a $\rG_{\bd'}\times \rG_{\bd''}$-invariant constructible function on $W(\supp \hat{f}',\supp \hat{f}'';\tilde{Z})$ defined by
$$(\hat{f}'\triangledown\hat{f}'')(\triangle)=(\hat{f}'\triangledown\hat{f}'')(Y\xrightarrow{f}Z\xrightarrow{g}X\xrightarrow{h}Y^*)=\hat{f}'(\tilde{X})\hat{f}''(\tilde{Y}),$$
then
\begin{equation}\label{integration form}
(\hat{f}'*\hat{f}'')(\tilde{Z})\!=\!\int_{V(\supp \hat{f}',\supp \hat{f}'';\tilde{Z})}\!
\!\!\!(\hat{f}'\triangledown\hat{f}'')(\triangle)\!=\!\!\sum_{c',c''\in \bbC}c'c''\chi(V(\hat{f}'^{-1}(c'),\hat{f}'^{-1}(c');\tilde{Z})).
\end{equation}

\subsection{Lie algebra $\tilde{\fg}$}\label{Lie algebra spanned by supported-indecomposable function and the Grothendieck group}\

\begin{definition}Define $\tilde{\fh}=\bbC\otimes_{\bbZ}K_0$. For any $\bd\in K_0$, we denote by $\tilde{h}_{\bd}=1\otimes \bd$. Then $\tilde{\fh}$ is a $\bbC$-vector space spanned by $\{\tilde{h}_{\bd}|\bd\in K_0\}$ subject to relations 
$$\tilde{h}_{\bd}=\tilde{h}_{\bd'}+\tilde{h}_{\bd''}$$
for any $\bd=\bd'+\bd''\in K_0$.
\end{definition}

There is a well-defined symmetric bilinear form $(-|-):\tilde{\fh}\times \tilde{\fh}\rightarrow \bbZ$ given by
\begin{equation}
\begin{aligned}\label{symmetric bilinear form}
(\tilde{h}_{\bd}|\tilde{h}_{\bd'})=\ \ &\dim_{\bbC}\Hom_{\cK_2(\cP)}(X,Y)-\dim_{\bbC}\Hom_{\cK_2(\cP)}(X,Y^*)\\
+&\dim_{\bbC}\Hom_{\cK_2(\cP)}(Y,X)-\dim_{\bbC}\Hom_{\cK_2(\cP)}(Y,X^*),
\end{aligned}
\end{equation}
where $X,Y\in \cK_2(\cP)$ whose images in $K_0$ are $\bd',\bd''$ respectively, see \cite[Section 3.3]{Peng-Xiao-2000}.

\begin{definition}
(a) A function $\hat{f}\in M_{\rG_{\bd}}(\rP_2(A,\bd))$ is called support-indecomposable, if for any $\tilde{X}\in \rP_2(A,\bd)$, the value $\hat{f}(\tilde{X})\not=0$ implies any $X\in \tilde{X}$ is indecomposable as an object in $\cK_2(\cP)$, equivalently, $X_r$ is an indecomposable radical complex.\\
(b) Define $\tilde{\fn}_{\bd}\subset M_{\rG_{\bd}}(\rP_2(A,\bd))$ to be the $\bbC$-subspace consisting of functions which are support-indecomposable, and define 
$$\tilde{\fn}=\bigoplus_{\bd\in K_0}\tilde{\fn}_{\bd}.$$
\end{definition}

\begin{lemma}\label{Lie bracket supported indecomposable}
Let $\hat\cO\subset \rP_2(A,\bd)$ be a $\rG_{\bd}$-invariant support-bounded constructible subset and $\hat\cO'\subset \rP_2(A,\bd')$ be a $\rG_{\bd'}$-invariant support-bounded constructible subset consisting of homotopy equivalence classes of indecomposable objects in $\cK_2(\cP)$, then 
$$1_{\hat\cO}*1_{\hat\cO'}-1_{\hat\cO'}*1_{\hat\cO}\in \tilde{\fn}_{\bd+\bd'}.$$
As a consequence, for any $\hat{f}\in \tilde{\fn}_{\bd}, \hat{g}\in \tilde{\fn}_{\bd'}$, we have 
$$\hat{f}*\hat{g}-\hat{g}*\hat{f}\in \tilde{\fn}_{\bd+\bd'}.$$
\end{lemma}
The proof will be given in Subsection \ref{Preliminary result about the structure constants}.

\begin{definition}
Define the $\bbC$-vector space 
$$\tilde{\fg}=\tilde{\fn}\oplus \tilde{\fh},$$
and a bilinear map $[-,-]:\tilde{\fg}\times \tilde{\fg}\rightarrow \tilde{\fg}$ by
\begin{equation}
\begin{aligned}\label{Lie algebra g_2 formula}
&[1_{\hat\cO},1_{\hat\cO'}]=[1_{\hat\cO},1_{\hat\cO'}]_{\tilde{\fn}}-\chi([(\hat\cO\cap\hat\cO'^*)/\rG_{\bd}])\tilde{h}_{\bd},\\
&[1_{\hat\cO},1_{\hat\cO'}]_{\tilde{\fn}}=1_{\hat\cO}*1_{\hat\cO'}-1_{\hat\cO'}*1_{\hat\cO}\in \tilde{\fn},\\
&[\tilde{h}_{\bd},1_{\hat\cO'}]=(\tilde{h}_{\bd}|\tilde{h}_{\bd'})1_{\hat\cO'},\\
&[1_{\hat\cO'},\tilde{h}_{\bd}]=-[\tilde{h}_{\bd},1_{\hat\cO'}],\\
&[\tilde{h}_{\bd},\tilde{h}_{\bd'}]=0,
\end{aligned}
\end{equation}
where $\hat\cO\subset \rP_2(A,\bd)$ is a $\rG_{\bd}$-invariant support-bounded constructible subset and $\hat\cO'\subset \rP_2(A,\bd')$ is a $\rG_{\bd'}$-invariant support-bounded constructible subset consisting of homotopy equivalence classes of indecomposable objects in $\cK_2(\cP)$, and $\bd,\bd'\in K_0$. The notation $\hat{\cO}'^*\subset \rP_2(A,-\bd')$ is a $\rG_{-\bd'}$-invariant support-bounded constructible subset consisting of $\tilde{X^*}$ for any $\tilde{X}\in \hat\cO'$.
\end{definition}

\begin{remark}
(a) One may compare formulas (\ref{Lie algebra g formula}) with (\ref{Lie algebra g_2 formula}) roughly. Their specific relation will be given in Theorem \ref{two Lie algebra isomorphism}.\\
(b) In general, the bracket $[-,-]$ in (\ref{Lie algebra g_2 formula}) can be written as follows,
\begin{align*}
&[\hat{f},\hat{g}]=[\hat{f},\hat{g}]_{\tilde{\fn}}-\int_{[(\supp \hat{f}\cap (\supp \hat{g})^*)/\rG_{\bd}]}\hat{f}(\tilde{X})\hat{g}(\tilde{X}^*)\cdot\tilde{h}_{\bd},\\
&[\hat{f},\hat{g}]_{\tilde{\fn}}(\tilde{Z})=(\hat{f}*\hat{g})(\tilde{Z})-(\hat{g}*\hat{f})(\tilde{Z}),\\
&[\tilde{h}_{\bd},\hat{g}]=(\tilde{h}_{\bd}|\tilde{h}_{\bd'})\hat{g},\\
&[\hat{g},\tilde{h}_{\bd}]=-[\tilde{h}_{\bd},\hat{g}],\\
&[\tilde{h}_{\bd},\tilde{h}_{\bd'}]=0,
\end{align*}
where $\hat{f}\in \tilde{\fn}_{\bd},\hat{g}\in \tilde{\fn}_{\bd'},\bd,\bd'\in K_0$ and $\tilde{Z}\in \rP_2(A,\bd+\bd')$. In the first equation, we view $\tilde{X}\mapsto \hat{f}(\tilde{X})\hat{g}(\tilde{X}^*)$ as a $\rG_{\bd}$-invariant support bounded constructible function on $\supp \hat{f}\cap (\supp \hat{g})^*$. In the second equation, $(\hat{f}*\hat{g})(\tilde{Z}),(\hat{g}*\hat{f})(\tilde{Z})$ can be written as integrations, see (\ref{integration form}).
\end{remark}

\begin{theorem}\label{Lie algebra g_2}
The $\bbC$-vector space $\tilde{\fg}$ is a Lie algebra with the bracket $[-,-]$.
\end{theorem}

For $i=1,2,3$, let $\bd_i\in K_0$ and $\hat{\cO}_i\subset \rP_2(A,\bd_i)$ be a $\rG_{\bd_i}$-invariant support-bounded constructible subset consisting of homotopy equivalence classes of indecomposable objects in $\cK_2(\cP)$, we need to prove the following Jacobi identities,
\begin{align}
[[1_{\hat{\cO}_1},1_{\hat{\cO}_2}],1_{\hat{\cO}_3}]-[[1_{\hat{\cO}_1},1_{\hat{\cO}_3}],1_{\hat{\cO}_2}]-[[1_{\hat{\cO}_3},1_{\hat{\cO}_2}],1_{\hat{\cO}_1}]&=0, \tag{I} \\
[[\tilde{h}_{\bd_1},1_{\hat{\cO}_2}],1_{\hat{\cO}_3}]-[[\tilde{h}_{\bd_1},1_{\hat{\cO}_3}],1_{\hat{\cO}_2}]-[[1_{\hat{\cO}_3},1_{\hat{\cO}_2}],\tilde{h}_{\bd_1}]&=0, \tag{II}\\
[[\tilde{h}_{\bd_1},\tilde{h}_{\bd_2}],1_{\hat{\cO}_3}]-[[\tilde{h}_{\bd_1},1_{\hat{\cO}_3}],\tilde{h}_{\bd_2}]-[[1_{\hat{\cO}_3},\tilde{h}_{\bd_2}],\tilde{h}_{\bd_1}]&=0, \tag{III}\\
[[\tilde{h}_{\bd_1},\tilde{h}_{\bd_2}],\tilde{h}_{\bd_3}]-[[\tilde{h}_{\bd_1},\tilde{h}_{\bd_3}],\tilde{h}_{\bd_2}]-[[\tilde{h}_{\bd_3},\tilde{h}_{\bd_2}],\tilde{h}_{\bd_1}]&=0. \tag{IV}
\end{align}
The proof will be given in Subsection \ref{Jacobi identity}.

\subsection{An application of the octahedral axiom}\label{An application of octahedral axiom}\

For $i=1,2,3$, let $\bd_i\in K_0$ and $\hat{\cO}_i\subset \rP_2(A,\bd_i)$ be a $\rG_{\bd_i}$-invariant support-bounded constructible subset consisting of homotopy equivalence classes of indecomposable objects in $\cK_2(\cP)$. Suppose $\hat{\cO}_i=\rG_{\bd_i}.t_{e_i}(\cO_{\ue_i})$ for some $\rG_{\ue_i}$-invariant constructible subset $\cO_{\ue_i}\subset \rP_2^*(A,\ue_i)$ and $\ue_i\in \udim^{-1}(\bd_i)$. For any $\ue'_i\in \udim^{-1}(\bd_i)$ satisfying $\ue_i\leqslant \ue'_i$, set $\cO_{\ue'_i}=t_{\ue_i\ue'_i}(\cO_{\ue_i})$, then we also have $\hat{\cO}_i=\rG_{\bd_i}.t_{e'_i}(\cO_{\ue'_i})$.

We define a subset 
$$\hat{\cL}_{\hat{\cO}_1\hat{\cO}_2}\subset \rP_2(A,\bd_1+\bd_2)$$
consisting of $\tilde{L}$ such that there exists a distinguished triangle $Y\rightarrow L\rightarrow X\rightarrow Y^*$ for some $\tilde{X}\in \hat{\cO}_1,\tilde{Y}\in \hat{\cO}_2$. It is a finite disjoint union of $\rG_{\bd_1+\bd_2}$-invariant support-bounded constructible subsets
\begin{align*}
\hat{\cL}_{\hat{\cO}_1\hat{\cO}_2}=\bigsqcup_{\ue\leqslant \ue_1+\ue_2}\rG_{\bd_1+\bd_2}.t_{\ue}(\cL^{\ue}_{\ue_1\ue_2}),
\end{align*}
where
\begin{equation}\label{index L}
\cL^{\ue}_{\ue_1\ue_2}=\rG_{\ue}.\{(\textrm{Cone}(h)^*)_r\in \rP_2^*(A,\ue)|h\in \Hom_{\cK_2(\cP)}(X,Y^*),X\in \cO_{\ue_1},Y\in \cO_{\ue_2}\}
\end{equation}
is a $\rG_{\ue}$-invariant constructible subset of $\rP_2^*(A,\ue)$, since any distinguished triangle $Y\rightarrow L\rightarrow X\xrightarrow{h} Y^*$ is isomorphic to the triangle $Y\rightarrow (\textrm{Cone}(h)^*)_r\rightarrow X\xrightarrow{h} Y^*$ determined by $h\in \Hom_{\cK_2(\cP)}(X,Y^*)$. Notice that if the subset $\cL^{\ue}_{\ue_1\ue_2}\not=\varnothing$, we have $\ue\in \udim^{-1}(\bd_1+\bd_2)$. For any such $\ue$, the subset $\cL^{\ue}_{\ue_1\ue_2}$ is independent of the choices of $\ue_1\in\udim^{-1}(\bd_1),\ue_2\in\udim^{-1}(\bd_2)$, that is, $\cL^{\ue}_{\ue_1\ue_2}=\cL^{\ue}_{\ue'_1\ue'_2}$.

For any $\tilde{M}\in \rP_2(A, \bd_1+\bd_2+\bd_3)$, we take a representative $M\in \tilde{M}$ and define 
$$W^{\tilde{M}}_{(\hat{\cO}_1\hat{\cO}_2)\hat{\cO}_3}=\bigsqcup_{\tilde{L}\in \hat{\cL}_{\hat{\cO}_1\hat{\cO}_2}}W(\hat{\cO}_1,\hat{\cO}_2;\tilde{L})\times W(\tilde{L},\hat{\cO}_3;\tilde{M}),$$
where $W(\hat{\cO}_1,\hat{\cO}_2;\tilde{L})\times W(\tilde{L},\hat{\cO}_3;\tilde{M})$ consists of morphisms $((f,g,h),(i,j,k))$ in $\cK_2(\cP)$ such that 
\begin{equation}
Y\xrightarrow{f}L\xrightarrow{g}X\xrightarrow{h}Y^*,\ Z\xrightarrow{i}M\xrightarrow{j}L\xrightarrow{k}Z^* \tag{$\triangle\triangle$}
\end{equation}
are distinguished triangles for some $\tilde{X}\in \hat{\cO}_1,\tilde{Y}\in \hat{\cO}_2,\tilde{Z}\in \hat{\cO}_3$. From now on, we use the notation ($\triangle\triangle$) to denote above two distinguished triangles. Since contractible complexes $K_P\oplus K_P^*$ are zero objects in $\cK_2(\cP)$ for any $P\in \cP$, and the morphisms involved them are zero, the set $W^{\tilde{M}}_{(\hat{\cO}_1\hat{\cO}_2)\hat{\cO}_3}$ can be written as 
\begin{equation}
\begin{aligned}\label{union 1}
W^{\tilde{M}}_{(\hat{\cO}_1\hat{\cO}_2)\hat{\cO}_3}=&\bigsqcup_{\ue\leqslant \ue_1+\ue_2}\bigsqcup_{L\in \cL^{\ue}_{\ue_1\ue_2}}W(\cO_{\ue_1},\cO_{\ue_2};L)\times W(L,\cO_{\ue_3};M)
\end{aligned}
\end{equation}
 which is independent of the choices of $\ue_i\in \udim^{-1}(\bd_i)$. Similarly, it is independent of the choices of $M\in \tilde{M}$. It is constructible, by the same argument of constructibility of $W(\cO',\cO'';Z)$ in Subsection \ref{convolution}.

We define a $\rG_{\bd_1}\times \rG_{\bd_2}\times \rG_{\bd_3}\times \rG_{\bd_1+\bd_2}$-action on $W^{\tilde{M}}_{(\hat{\cO}_1\hat{\cO}_2)\hat{\cO}_3}$ as follows. Firstly, we define a $\rG_{\ue_1}\times \rG_{\ue_2}\times \rG_{\ue_3}\times \rG_{\ue}$-action on $\bigsqcup_{L\in \cL^{\ue}_{\ue_1\ue_2}}W(\cO_{\ue_1},\cO_{\ue_2};L)\times W(L,\cO_{\ue_3};M)$. For any $(\alpha,\beta,\gamma,\lambda)\in \rG_{\ue_1}\times \rG_{\ue_2}\times \rG_{\ue_3}\times \rG_{\ue}$ and ($\triangle\triangle)\in W(\cO_{\ue_1},\cO_{\ue_2};L)\times W(L,\cO_{\ue_3};M)$, \begin{align*}
(\alpha,\beta,\gamma,\lambda).((f,g,h),(i,j,k))&=((\lambda f\beta^{-1},\alpha g\lambda^{-1},\beta^*h\alpha^{-1}),(i\gamma^{-1},\lambda j,\gamma^*k\lambda^{-1}))\\
&\in W(\cO_{\ue_1},\cO_{\ue_2};\lambda.L)\times W(\lambda.L,\cO_{\ue_3};M).
\end{align*}
Indeed, there are commutative diagrams
\begin{diagram}[midshaft,size=2em]
Y &\rTo^{f} &L &\rTo^{g} &X &\rTo^{h} &Y^* &Z &\rTo^{i} &M &\rTo^{j} &L &\rTo^{k} &Z^*\\
\dTo^{\beta} & &\dTo^{\lambda} & &\dTo^{\alpha} & &\dTo^{\beta^*} &\dTo^{\gamma} & &\vEq & &\dTo^{\lambda} & &\dTo^{\gamma^*}\\
\beta.Y &\rTo^{\lambda f\beta^{-1}} &\lambda.L &\rTo^{\alpha g\lambda^{-1}} &\alpha.X &\rTo^{\beta^*h\alpha^{-1}} &(\beta.Y)^* &\gamma.Z &\rTo^{i\gamma^{-1}} &M &\rTo^{\lambda j} &\lambda.L &\rTo^{\gamma^*k\lambda^{-1}} &(\gamma.Z)^*
\end{diagram}
where $(\alpha,\beta,\gamma,\lambda)$ induces isomorphisms in $\cC_2(\cP)$, and then their homotopy classes are isomorphisms in $\cK_2(\cP)$ such that 
$$\beta.Y\xrightarrow{\lambda f \beta^{-1}}\lambda.L\xrightarrow{\alpha g\lambda^{-1}}\alpha.X\xrightarrow{\beta^* h \alpha^{-1}}(\beta.Y)^*,\ \gamma.Z\xrightarrow{i \gamma^{-1}}M\xrightarrow{\lambda j}\lambda.L\xrightarrow{\gamma^* k \lambda^{-1}}(\gamma.Z)^*$$
are distinguished triangles. Similarly, we define a $\rG_{\ue'_1}\times \rG_{\ue'_2}\times \rG_{\ue'_3}\times \rG_{\ue'}$-action on $\bigsqcup_{L\in \cL^{\ue'}_{\ue'_1\ue'_2}}W(\cO_{\ue'_1},\cO_{\ue'_2};L)\times W(L,\cO_{\ue'_3};M)$ for any $\ue'_i\in \udim^{-1}(\bd_i), \ue'\in \udim_{-1}(\bd_1+\bd_2)$ satisfying $\ue_i\leqslant \ue'_i,\ue\leqslant \ue'$, and then
\begin{align*}
&[\bigsqcup_{L\in \cL^{\ue}_{\ue_1\ue_2}}W(\cO_{\ue_1},\cO_{\ue_2};L)\times W(L,\cO_{\ue_3};M)/\rG_{\ue_1}\times \rG_{\ue_2}\times \rG_{\ue_3}\times \rG_{\ue}]\\
=&[\bigsqcup_{L\in \cL^{\ue'}_{\ue'_1\ue'_2}}W(\cO_{\ue'_1},\cO_{\ue'_2};L)\times W(L,\cO_{\ue'_3};M)/\rG_{\ue'_1}\times \rG_{\ue'_2}\times \rG_{\ue'_3}\times \rG_{\ue'}]
\end{align*}
is independent of the choices of $\ue_i\in \udim^{-1}(\bd_i), \ue\in \udim^{-1}(\bd_1+\bd_2)$. Moreover, the action is compatible with ind-limits, and so we obtain a $\rG_{\bd_1}\!\times\! \rG_{\bd_2}\!\times\! \rG_{\bd_3}\!\times\! \rG_{\bd_1+\bd_2}$-action on $W^{\tilde{M}}_{(\hat{\cO}_1\hat{\cO}_2)\hat{\cO}_3}$, and denote the quotient by
\begin{equation}
\begin{aligned}\label{quotient union 1}
&V^{\tilde{M}}_{(\hat{\cO}_1\hat{\cO}_2)\hat{\cO}_3}=[W^{\tilde{M}}_{(\hat{\cO}_1\hat{\cO}_2)\hat{\cO}_3}/\rG_{\bd_1}\times \rG_{\bd_2}\times \rG_{\bd_3}\times \rG_{\bd_1+\bd_2}]\\
=&\bigsqcup_{\ue\leqslant \ue_1+\ue_2}[(\bigsqcup_{L\in \cL^{\ue}_{\ue_1\ue_2}}W(\cO_{\ue_1},\cO_{\ue_2};L)\times W(L,\cO_{\ue_3};M))/\rG_{\ue_1}\times\rG_{\ue_2}\times\rG_{\ue_3}\times \rG_{\ue}].
\end{aligned}
\end{equation}

Dually, we define a finite disjoint union of support-bounded constructible subsets
\begin{align*}
\hat{\cL}_{\hat{\cO}_2\hat{\cO}_3}&=\bigsqcup_{\ue'\leqslant \ue_2+\ue_3}\rG_{\bd_2+\bd_3}.t_{\ue'}(\cL^{\ue'}_{\ue_2\ue_3})\subset \rP_2(A,\bd_2+\bd_3),\\
\cL^{\ue'}_{\ue_2\ue_3}=\rG_{\ue'}.\{(\textrm{Cone}&(k')^*)_r\in \rP_2^*(A,\ue)|k'\in \Hom_{\cK_2(\cP)}(Y,Z^*),Y\in \cO_{\ue_2},Z\in \cO_{\ue_3}\},
\end{align*}
and define
\begin{equation}
\begin{aligned}\label{union 2}
W^{\tilde{M}}_{\hat{\cO}_1(\hat{\cO}_2\hat{\cO}_3)}=&\bigsqcup_{\tilde{L}'\in \hat{\cL}_{\hat{\cO}_2\hat{\cO}_3}}W(\hat{\cO}_1,\tilde{L}';\tilde{M})\times W(\hat{\cO}_2,\hat{\cO}_3;\tilde{L}')\\
=&\bigsqcup_{\ue'\leqslant \ue_2+\ue_3}\bigsqcup_{L'\in \cL^{\ue'}_{\ue_2\ue_3}}W(\cO_{\ue_1},L';M)\times W(\cO_{\ue_2},\cO_{\ue_3};L')
\end{aligned}
\end{equation}
for any $\tilde{M}\in \rP_2(A,\bd_1+\bd_2+\bd_3)$, where $W(\hat{\cO}_1,\tilde{L}';\tilde{M})\times W(\hat{\cO}_2,\hat{\cO}_3;\tilde{L}')$ consists of morphisms $((f',g',h'),(i',j',k'))$ in $\cK_2(\cP)$ such that 
\begin{equation}
L'\xrightarrow{f'}M \xrightarrow{g'}X\xrightarrow{h'}L'^*,\ Z\xrightarrow{i'}L'\xrightarrow{j'} Y\xrightarrow{k'}Z^* \tag{$\triangle\triangle'$}
\end{equation}
are distinguished triangles for some $\tilde{X}\in \hat{\cO}_1,\tilde{Y}\in \hat{\cP}_2,\tilde{Z}\in \hat{\cO}_3$. From now on, we use the notation ($\triangle\triangle'$) to denote by above two distinguished triangles. Then we define a $\rG_{\ue_1}\times \rG_{\ue_2}\times \rG_{\ue_3}\times \rG_{\ue'}$-action on $\bigsqcup_{L'\in \cL^{\ue'}_{\ue_2\ue_3}}W(\cO_{\ue_1},L';M)\times W(\cO_{\ue_2},\cO_{\ue_3};L')$. For any $(\alpha,\beta,\gamma,\lambda')\in \rG_{\ue_1}\times \rG_{\ue_2}\times \rG_{\ue_3}\times \rG_{\ue'}$ and $(\triangle\triangle')\in W(\cO_{\ue_1},L';M)\times W(\cO_{\ue_2},\cO_{\ue_3};L')$,
\begin{align*}
(\alpha,\beta,\gamma,\lambda').((f',g',h'),(i',j',k'))\!&=\!((f'\lambda'^{-1},\alpha g',\lambda'^*h'\alpha^{-1}),(\lambda'i'\gamma^{-1},\beta j'\lambda'^{-1},\gamma^*k'\beta^{-1}))\\
&\in W(\cO_{\ue_1},\lambda'.L';M)\times W(\cO_{\ue_2},\cO_{\ue_3};\lambda'.L'),
\end{align*}
that is, there are commutative diagrams
\begin{diagram}[midshaft,size=2em]
L'\! \!&\rTo^{f'} &\!M\!\! &\rTo^{g'} &\!X\!\! &\rTo^{h'} &\!\!L'^* &Z\!\! &\rTo^{i'} &\!L'\!\! &\rTo^{j'} &\!Y\!\!\! &\rTo^{k'} &\!\!Z^*\\
\dTo^{\lambda'} & &\vEq & &\dTo^{\alpha} & &\dTo^{\lambda'^*} &\dTo^\gamma & &\dTo^{\lambda'} & &\dTo^\beta & &\dTo^{\gamma^*}\\
\lambda'.L'\!\! &\rTo^{f'\lambda'^{-1}} &\!M\!\! &\rTo^{\alpha g'} &\!\alpha.X\!\!&\rTo^{\lambda'^*h'\alpha^{-1}} &\!(\lambda'.L')^* &\gamma.Z\!\! &\rTo^{\lambda' i' \gamma^{-1}} &\!\!\lambda'.L'\!\! &\rTo^{\beta j'\lambda'^{-1}} &\!\beta.Y\!\!\!&\rTo^{\gamma^*k'\beta^{-1}} &\!\!(\gamma.Z)^*
\end{diagram}
Then we obtain a $\rG_{\bd_1}\times \rG_{\bd_2}\times \rG_{\bd_3}\times \rG_{\bd_1+\bd_2}$-action on $W^{\tilde{M}}_{\hat{\cO}_1(\hat{\cO}_2\hat{\cO}_3)}$, and denote the quotient of ind-limits by
\begin{equation}
\begin{aligned}\label{quotient union 2}
&V^{\tilde{M}}_{\hat{\cO}_1(\hat{\cO}_2\hat{\cO}_3)}=[W^{\tilde{M}}_{\hat{\cO}_1(\hat{\cO}_2\hat{\cO}_3)}/\rG_{\bd_1}\times \rG_{\bd_2}\times \rG_{\bd_3}\times \rG_{\bd_1+\bd_2}]\\
=&\bigsqcup_{\ue'\leqslant \ue_2+\ue_3}[(\bigsqcup_{L'\in \cL^{\ue'}_{\ue_2\ue_3}}W(\cO_{\ue_1},L';M)\times W(\cO_{\ue_2},\cO_{\ue_3};L'))/\rG_{\ue_1}\times\rG_{\ue_2}\times \rG_{\ue_3}\times \!\rG_{\ue'}]
\end{aligned}
\end{equation}

\begin{proposition}\label{octahedral Euler characteristic}
The stacks $V^{\tilde{M}}_{(\hat{\cO}_1\hat{\cO}_2)\hat{\cO}_3}$ and $V^{\tilde{M}}_{\hat{\cO}_1(\hat{\cO}_2\hat{\cO}_3)}$ are isomorphic, and so
$$\chi(V^{\tilde{M}}_{(\hat{\cO}_1\hat{\cO}_2)\hat{\cO}_3})=\chi(V^{\tilde{M}}_{\hat{\cO}_1(\hat{\cO}_2\hat{\cO}_3)}).$$
\end{proposition}
\begin{proof}
We define a morphism $\varphi:V^{\tilde{M}}_{(\hat{\cO}_1\hat{\cO}_2)\hat{\cO}_3}\rightarrow V^{\tilde{M}}_{\hat{\cO}_1(\hat{\cO}_2\hat{\cO}_3)}$ firstly, by (\ref{quotient union 1}), (\ref{quotient union 2}), that is, 
\begin{align*}
\varphi&:\bigsqcup_{\ue\leqslant \ue_1+\ue_2}[(\bigsqcup_{L\in \cL^{\ue}_{\ue_1\ue_2}}W(\cO_{\ue_1},\cO_{\ue_2};L)\times W(L,\cO_{\ue_3};M))/\rG_{\ue_1}\times\rG_{\ue_2}\times\rG_{\ue_3}\times \rG_{\ue}]\\
&\rightarrow \bigsqcup_{\ue'\leqslant \ue_2+\ue_3}[(\bigsqcup_{L'\in \cL^{\ue'}_{\ue_2\ue_3}}W(\cO_{\ue_1},L';M)\times W(\cO_{\ue_2},\cO_{\ue_3};L'))/\rG_{\ue_1}\times\rG_{\ue_2}\times \rG_{\ue_3}\times \!\rG_{\ue'}].
\end{align*}
For any $L\in \cL^{\ue}_{\ue_1\ue_2}$ and $((f,g,h),(i,j,k))\in W(\cO_{\ue_1},\cO_{\ue_2};L)\times W(L,\cO_{\ue_3};M)$ such that
$$Y\xrightarrow{f}L\xrightarrow{g}X\xrightarrow{h}Y^*,\ Z\xrightarrow{i}M\xrightarrow{j}L\xrightarrow{k}Z^*$$
are distinguished triangles, we set 
\begin{align}
&k'=kf\in \Hom_{\cK_2(\cP)}(Y,Z^*), \label{algebraic 1}\\
&L'=\textrm{Cone}(k')_r^*\in \cL_{\ue'}(\cO_{\ue_2},\cO_{\ue_3}),\\
&(Z\xrightarrow{i'}L'\xrightarrow{j'}Y\xrightarrow{k'}Z^*)=(Z\xrightarrow{-\pi\begin{pmatrix}\begin{smallmatrix}0\\1\end{smallmatrix}\end{pmatrix}^*}\textrm{Cone}(k')_r^*\xrightarrow{-\begin{pmatrix}\begin{smallmatrix}1 &0\end{smallmatrix}\end{pmatrix}^*\iota}Y\xrightarrow{k'}Z^*),
\end{align}
where $\ue'\leqslant \ue_2+\ue_3$ and $\pi:\textrm{Cone}(k')^*\rightarrow L',\iota:L'\rightarrow \textrm{Cone}(k')^*$ are natural projection and inclusion respectively, such that $(i',j',k')\in W(\cO_{\ue_2},\cO_{\ue_3};L')$.
By the octahedral axiom (TR4) of $\cK_2(\cP)$, there exists $(f',g',h')\in W(\cO_{\ue_1},L';M)$ such that the following diagram about distinguished triangles commutes
\begin{diagram}[midshaft,size=2em]
Z &\hEq &Z\\
\dTo^{i'} & &\dTo_{i}\\
L' &\rDashto^{f'} &M &\rDashto^{g'} &X &\rDashto^{h'} &L'^*\\
\dTo^{j'} &\square &\dTo_j & &\vEq & &\dTo_{j'^*}\\
Y &\rTo^{f} &L &\rTo^{g} &X &\rTo^{h} &Y^*\\
\dTo^{k'} & &\dTo_k\\
Z^* &\hEq &Z^*,
\end{diagram}
and the middle left square is homotopy cartesian with differential $h'g$, see \cite[Lemma 1.4.3]{Neeman-2001} or \cite[Section 3]{Hubery-2006}, that is, there is a distinguished triangle
$$L'\xrightarrow{\begin{pmatrix}\begin{smallmatrix}f'\\-j'\end{smallmatrix}\end{pmatrix}}M\oplus Y\xrightarrow{\begin{pmatrix}\begin{smallmatrix}j &f\end{smallmatrix}\end{pmatrix}}L \xrightarrow{h'g}L'^*.$$
More precisely, by the proof of (TR4) of $\cK_2(\cP)$, we have
\begin{align}
&f'=s\begin{pmatrix}\begin{smallmatrix}f &\\ &1\end{smallmatrix}\end{pmatrix}\iota:L'\rightarrow\textrm{Cone}(k')^*\rightarrow\textrm{Cone}(k)^*\rightarrow M,\\
&g'=gj:M\rightarrow L\rightarrow X,\\
&h'=\pi^*\begin{pmatrix}\begin{smallmatrix}1 &\\ &k\end{smallmatrix}\end{pmatrix}t^{-1}:X\rightarrow \textrm{Cone}(f)\rightarrow\textrm{Cone}(k')\rightarrow L'^*, \label{algebraic 2}
\end{align}
where $s,t$ are fixed homotopy equivalences such that 
\begin{align*}
(1,s,1)&:(Z\xrightarrow{-\begin{pmatrix}\begin{smallmatrix}0\\1\end{smallmatrix}\end{pmatrix}^*}\textrm{Cone}(k)^*\xrightarrow{-\begin{pmatrix}\begin{smallmatrix}1 &0\end{smallmatrix}\end{pmatrix}^*}L\xrightarrow{k}Z^*)\rightarrow(Z\xrightarrow{i}M\xrightarrow{j}L\xrightarrow{k}Z^*),\\
(1,1,t)&:(Y\xrightarrow{f}L\xrightarrow{\begin{pmatrix}\begin{smallmatrix}0 &1\end{smallmatrix}\end{pmatrix}}\textrm{Cone}(f)\xrightarrow{\begin{pmatrix}\begin{smallmatrix}1 &0\end{smallmatrix}\end{pmatrix}}Y^*)\rightarrow(Y\xrightarrow{f}L\xrightarrow{g}X\xrightarrow{h}Y^*)
\end{align*}
are isomorphisms between distinguished triangles. 

For any $(\alpha,\beta,\gamma,\lambda)\in \rG_{\ue_1}\times \rG_{\ue_2}\times \rG_{\ue_3}\times \rG_{\ue}$, recall that 
\begin{align*}
(\alpha,\beta,\gamma,\lambda).((f,g,h),(i,j,k))&=((\lambda f\beta^{-1},\alpha g\lambda^{-1},\beta^*h\alpha^{-1}),(i\gamma^{-1},\lambda j,\gamma^*k\lambda^{-1}))\\
&\in W(\cO_{\ue_1},\cO_{\ue_2};\lambda.L)\times W(\lambda.L,\cO_{\ue_3};M).
\end{align*}
Similarly, we set 
\begin{align*}
&k''=\gamma^*kf\beta^{-1}\in \Hom_{\cK_2(\cP)}(\beta.Y,(\gamma.Z)^*),\\
&L''=\textrm{Cone}(k'')^*_r\in \cL_{\ue''}(\cO_{\ue_2},\cO_{\ue_3}),\\
&(\gamma.Z\xrightarrow{i''}L''\xrightarrow{j''}\beta.Y\xrightarrow{k''}(\gamma.Z)^*)\!=\!(\gamma.Z\xrightarrow{-\pi'\begin{pmatrix}\begin{smallmatrix}0\\1\end{smallmatrix}\end{pmatrix}^*}\!\textrm{Cone}(k'')_r^*\!\xrightarrow{-\begin{pmatrix}\begin{smallmatrix}1 &0\end{smallmatrix}\end{pmatrix}^*\iota'}\!\beta.Y\xrightarrow{k''}\!(\gamma.Z)^*),
\end{align*}
where $\ue''\leqslant \ue_2+\ue_3$ and $\pi':\textrm{Cone}(k'')^*\rightarrow L'', \iota':L''\rightarrow \textrm{Cone}(k'')^*$ are natural projection and inclusion respectively, such that $(i'',j'',k'')\in W(\cO_{\ue_2},\cO_{\ue_3};L'')$. By the octahedral axiom (TR4) of $\cK_2(\cP)$, there exists $(f'',g'',h'')\in W(\cO_{\ue_1},L'';M)$ such that the following diagram about distinguished triangles commutes
\begin{diagram}[midshaft,size=2em]
\gamma.Z &\hEq &\gamma.Z\\
\dTo^{i''} & &\dTo_{i\gamma^{-1}}\\
L'' &\rDashto^{f''} &M &\rDashto^{g''} &\alpha.X &\rDashto^{h''} &L''^*\\
\dTo^{j''} &\square &\dTo_{\lambda j} & &\vEq & &\dTo_{j''^*}\\
\beta.Y &\rTo^{\lambda f\beta^{-1}} &\lambda.L &\rTo^{\alpha g\lambda^{-1}} &\alpha.X &\rTo^{\beta^*h\alpha^{-1}} &(\beta.Y)^*\\
\dTo^{k''} & &\dTo_{\gamma^*k\lambda^{-1}}\\
(\gamma.Z)^* &\hEq &(\gamma.Z)^*,
\end{diagram}
and the middle left square is homotopy cartesian with differential $h''\alpha g\lambda^{-1}$, that is, there is a distinguished triangle
$$L''\xrightarrow{\begin{pmatrix}\begin{smallmatrix}f''\\-j''\end{smallmatrix}\end{pmatrix}}M\oplus\beta.Y\xrightarrow{\begin{pmatrix}\begin{smallmatrix}\lambda j &\lambda f\beta^{-1}\end{smallmatrix}\end{pmatrix}}\lambda.L \xrightarrow{h''\alpha g\lambda^{-1}}L''^*.$$
By the axiom (TR3) of $\cK_2(\cP)$, there exists $\lambda'\in \Hom_{\cC_2(\cP)}(L',L'')$ such that the following diagram about distinguished triangles commutes in $\cK_2(\cP)$
\begin{diagram}[midshaft,size=2em]
L' &\rTo^{\begin{pmatrix}\begin{smallmatrix}f'\\-j'\end{smallmatrix}\end{pmatrix}} &M\oplus Y &\rTo^{\begin{smallmatrix}\begin{pmatrix}j &f\end{pmatrix}\end{smallmatrix}} &L &\rTo^{h'g} &L'^*\\
\dDashto^{\lambda'} & &\dTo^{\begin{pmatrix}\begin{smallmatrix}1 &\\& \beta\end{smallmatrix}\end{pmatrix}} & &\dTo^{\lambda} & &\dDashto^{\lambda'^*}\\
L'' &\rTo^{\begin{pmatrix}\begin{smallmatrix}f''\\-j''\end{smallmatrix}\end{pmatrix}} &M\oplus\beta.Y &\rTo^{\begin{pmatrix}\begin{smallmatrix}\lambda j&\lambda f\beta^{-1}\end{smallmatrix}\end{pmatrix}} &\lambda.L &\rTo^{h''\alpha g\lambda^{-1}} &L''^*.
\end{diagram}
Since $\begin{pmatrix}\begin{smallmatrix}1 &\\& \beta\end{smallmatrix}\end{pmatrix}$ and $\lambda$ are isomorphisms, $\lambda'$ is a homotopy equivalence. Moreover, since $L',L''$ are radical complexes, by Lemma \ref{isomorphism and homotopy equivalence}, $\lambda'$ is an isomorphism, and so $\ue'=\ue'',\lambda'\in \rG_{\ue'}$ and $L''=\lambda'.L'$. Again, by the axiom (TR3) of $\cK_2(\cP)$, there exists $\alpha'_r\in \Hom_{\cC_2(\cP)}(X,\alpha.X), \gamma'_r\in \Hom_{\cC_2(\cP)}(Z,\gamma.Z)$ such that the following diagrams about distinguished triangles commute in $\cK_2(\cP)$
\begin{equation}\label{two commutative diagram}
\begin{diagram}[midshaft,size=2em]
L' &\rTo^{f'} &M &\rTo^{g'} &X &\rTo^{h'} &L'^* & &Z &\rTo^{i'} &L' &\rTo^{j'} &Y &\rTo^{k'} &Z^*\\
\dTo^{\lambda'} & &\vEq & &\dDashto^{\alpha'_r}_{\alpha'} & &\dTo^{\lambda'^*} & &\dDashto^{\gamma'_r}_{\gamma'} & &\dTo^{\lambda'} & &\dTo^{\beta}& &\dTo^{\gamma'^*_r}\\
\lambda'.L' &\rTo^{f''} &M &\rTo^{g''} &\alpha.X &\rTo^{h''} &(\lambda'.L')^*, & &\gamma.Z &\rTo^{i''} &\lambda'.L' &\rTo^{j''} &\beta.Y &\rTo^{k''} &(\gamma.Z)^*.
\end{diagram}
\end{equation}
Since $\lambda',\beta$ are isomorphisms, $\alpha'_r,\gamma'_r$ are homotopy equivalences. Notice that $X,\alpha.X$ have the same projective dimension vector pair. By Lemma \ref{isomorphism and homotopy equivalence}, $\alpha'_r$ is an isomorphism and so $\alpha'_r\in \rG_{(\ue_1)_r}$, where $(\ue_1)_r$ is the projective dimension vector pair of $X_r,(\alpha.X)_r$. Then there exists $\alpha'=f_{(\ue_1)_r\ue_1}(\alpha'_r)=\alpha'_r\oplus 1\oplus 1\in \rG_{\ue_1}$, see Subsection \ref{Moduli K_2(P)}, such that $\alpha.X=\alpha'.X$ and the left diagram in (\ref{two commutative diagram}) commutes. Similarly, there exists $\gamma'\in \rG_{\ue_3}$ such that $\gamma'.Z=\gamma.Z$ and the right diagram in (\ref{two commutative diagram}) commutes. Hence
\begin{align*}
((f'',g'',h''),(i'',j'',k''))=&((f'\lambda'^{-1},\alpha'g',\lambda'^*h'\alpha'^{-1}),(\lambda'i'\gamma'^{-1},\beta j'\lambda'^{-1},\gamma'^*k'\beta^{-1}))\\
=&(\alpha',\beta,\gamma',\lambda').((f',g',h'),(i',j',k')),
\end{align*}
and so there is a well-defined map sending an orbit to an orbit
$$\rG_{\ue_1}\times \rG_{\ue_2}\times \rG_{\ue_3}\times \rG_{\ue}.((f,g,h),(i,j,k))\mapsto \rG_{\ue_1}\times \rG_{\ue_2}\times \rG_{\ue_3}\times \rG_{\ue'}.((f',g',h'),(i',j',k'))$$
Moreover, by (\ref{algebraic 1})-(\ref{algebraic 2}), the map is algebraic and it induces a morphism $\varphi$.

Dually, by the symmetry, there is a morphism
$\psi:V^{\tilde{M}}_{\hat{\cO}_1(\hat{\cO}_2\hat{\cO}_3)}\rightarrow V^{\tilde{M}}_{(\hat{\cO}_1\hat{\cO}_2)\hat{\cO}_3}$
which is the inverse of $\varphi$, as desired.
\end{proof}

\subsection{Preliminary result about the structure constants}\label{Preliminary result about the structure constants}\

In this subsection, we deal with the structure constants $F$ in (\ref{constant F}), and their relations with $\chi(V^{\tilde{M}}_{(\hat{\cO}_1\hat{\cO}_2)\hat{\cO}_3})=\chi(V^{\tilde{M}}_{\hat{\cO}_1(\hat{\cO}_2\hat{\cO}_3)})$ in Proposition \ref{octahedral Euler characteristic} in preparation for the proof of Jacobi identity. 

With the same notations in Subsection \ref{An application of octahedral axiom}.  Recall that 
\begin{align*}
W^{\tilde{M}}_{(\hat{\cO}_1\hat{\cO}_2)\hat{\cO}_3}=\bigsqcup_{\ue\leqslant \ue_1+\ue_2}\bigsqcup_{L\in \cL^{\ue}_{\ue_1\ue_2}}W(\cO_{\ue_1},\cO_{\ue_2};L)\times W(L,\cO_{\ue_3};M),
\end{align*}
see (\ref{union 1}), where $\cL^{\ue}_{\ue_1\ue_2}$ consists of radical complexes, and for any $L\in \cL^{\ue}_{\ue_1\ue_2}$, the subset $W(\cO_{\ue_1},\cO_{\ue_2};L)\times W(L,\cO_{\ue_3};M)$ consists of $((f,g,h),(i,j,k))$ such that
\begin{equation}
Y\xrightarrow{f}L\xrightarrow{g}X\xrightarrow{h}Y^*,\ Z\xrightarrow{i}M\xrightarrow{j}L\xrightarrow{k}Z^* \tag{$\triangle\triangle$}
\end{equation}
are distinguished triangles for some $X\in \cO_{\ue_1},Y\in \cO_{\ue_2},Z\in \cO_{\ue_3}$. For any $\ue\leqslant \ue_1+\ue_2$, we divide 
\begin{equation}\label{partion W}
\bigsqcup_{L\in \cL^{\ue}_{\ue_1\ue_2}}W(\cO_{\ue_1},\cO_{\ue_2};L)\times W(L,\cO_{\ue_3};M)=\bigsqcup_{t=1}^4W^{\ue}_{(12)3}(t)
\end{equation}
into disjoint union of $\rG_{\ue_1}\times \rG_{\ue_2}\times \rG_{\ue_3}\times \rG_{\ue}$-invariant constructible subsets, where 
\begin{align*}
W^{\ue}_{(12)3}(1)&=\{(\triangle\triangle)|L\not\simeq M\oplus Z^*\},\\
W^{\ue}_{(12)3}(2)&=\{(\triangle\triangle)|L\simeq M\oplus Z^*,L\not\simeq X\oplus Y\},\\
W^{\ue}_{(12)3}(3)&=\{(\triangle\triangle)|L\simeq M\oplus Z^*,L\simeq X\oplus Y,X\not\simeq Y\},\\
W^{\ue}_{(12)3}(4)&=\{(\triangle\triangle)|L\simeq M\oplus Z^*,L\simeq X\oplus Y,X\simeq Y\}.
\end{align*}
Recall that the symbol $\simeq$ means homotopy equivalence, see Subsection \ref{The categories cC_2(cP) and cK_2(cP)}. We denote by $V^{\ue}_{(12)3}(t)=[W^{\ue}_{(12)3}(t)/\rG_{\ue_1}\times \rG_{\ue_2}\times \rG_{\ue_3}\times \rG_{\ue}]$ the corresponding quotient stack, and so $V^{\tilde{M}}_{(\hat{\cO}_1\hat{\cO}_2)\hat{\cO}_3}=\bigsqcup_{\ue\leqslant \ue_1+\ue_2}\bigsqcup^4_{t=1}V^{\ue}_{(12)3}(t)$. Dually, for 
$$W^{\hat{M}}_{\hat{\cO}_1(\hat{\cO}_2\hat{\cO}_3)}=\bigsqcup_{\ue'\leqslant \ue_2+\ue_3}\bigsqcup_{L'\in \cL^{\ue'}_{\ue_2\ue_3}}W(\cO_{\ue_1},L';M)\times W(\cO_{\ue_2},\cO_{\ue_3};L'),$$
see (\ref{union 2}), and for any $\ue'\leqslant \ue_2+\ue_3$, we divide 
$$\bigsqcup_{L'\in \cL^{\ue'}_{\ue_2\ue_3}}W(\cO_{\ue_1},L';M)\times W(\cO_{\ue_2},\cO_{\ue_3};L')=\bigsqcup^4_{t=1}W^{\ue'}_{1(23)}(t)$$
into disjoint union of $\rG_{\ue_1}\times \rG_{\ue_2}\times \rG_{\ue_3}\times \rG_{\ue'}$-invariant constructible subsets, where 
\begin{align*}
W^{\ue'}_{1(23)}(1)&=\{(\triangle\triangle')|L'\not\simeq M\oplus X^*\},\\
W^{\ue'}_{1(23)}(2)&=\{(\triangle\triangle')|L'\simeq M\oplus X^*,L'\not\simeq Y\oplus Z\},\\
W^{\ue'}_{1(23)}(3)&=\{(\triangle\triangle')|L'\simeq M\oplus X^*,L'\simeq Y\oplus Z,Y\not\simeq Z\},\\
W^{\ue'}_{1(23)}(4)&=\{(\triangle\triangle')|L'\simeq M\oplus X^*,L'\simeq Y\oplus Z,Y\simeq Z\}.
\end{align*}
We denote by $V^{\ue'}_{1(23)}(t)=[W^{\ue'}_{1(23)}(t)/\rG_{\ue_1}\times \rG_{\ue_2}\times \rG_{\ue_3}\times \rG_{\ue'}]$ the corresponding quotient stack, and so $V^{\tilde{M}}_{\hat{\cO}_1(\hat{\cO}_2\hat{\cO}_3)}=\bigsqcup_{\ue'\leqslant \ue_2+\ue_3}\bigsqcup^4_{t=1}V^{\ue'}_{1(23)}(t)$.
\begin{lemma}\label{finite subset I}
For any $\ue\leqslant \ue_1+\ue_2$, there exists a finite subset $\cI^{\ue}_{\ue_1\ue_2}\subset \bbC\times \bbC$ such that $\cL^{\ue}_{\ue_1\ue_2}$ can be divided into disjoint union of $\rG_{\ue}$-invariant constructible subsets
$$\cL^{\ue}_{\ue_1\ue_2}=\bigsqcup_{(r,s)\in \cI^{\ue}_{\ue_1\ue_2}}\cL^{\ue}_{\ue_1\ue_2}(r,s)$$
where $\cL^{\ue}_{\ue_1\ue_2}(r,s)=\{L\in \cL^{\ue}_{\ue_1\ue_2}|F^L_{\cO_{\ue_1}\cO_{\ue_2}}=r,F^L_{\cO_{\ue_2}\cO_{\ue_1}}=s\}$.
\end{lemma}
\begin{proof}
By Lemma \ref{convolution constructible}, both $1_{\cO_{\ue_1}}*1_{\cO_{\ue_2}}$ and $1_{\cO_{\ue_2}}*1_{\cO_{\ue_1}}$
are $\rG_{\ue}$-invariant constructible functions.  It is clear that they are supported on $\cL^{\ue}_{\ue_1\ue_2}$. Hence the subset $\cI^{\ue}_{\ue_1\ue_2}=(\im 1_{\cO_{\ue_1}}*1_{\cO_{\ue_2}})\times (\im 1_{\cO_{\ue_2}}*1_{\cO_{\ue_1}})\subset \bbC\times \bbC$ is finite, and for any $(r,s)\in \cI^{\ue}_{\ue_1\ue_2}$, the subset $\cL^{\ue}_{\ue_1\ue_2}(r,s)=(1_{\cO_{\ue_1}}*1_{\cO_{\ue_2}})^{-1}(r)\cap (1_{\cO_{\ue_2}}*1_{\cO_{\ue_1}})^{-1}(s)$ is $\rG_{\ue}$-invariant constructible.
\end{proof}

For any $(r,s)\in \cI^{\ue}_{\ue_1\ue_2}$, we fix a complex $L_{r,s}\in \cL^{\ue}_{\ue_1\ue_2}(r,s)$, and then simply denote by $\langle L_{r,s}\rangle=\cL^{\ue}_{\ue_1\ue_2}(r,s)$. We divide 
$$\langle L_{r,s}\rangle=\langle L_{r,s}\rangle_1\sqcup\langle L_{r,s}\rangle_2$$
into disjoint union of $\rG_{\ue}$-invariant constructible subsets, where
\begin{align*}
\langle L_{r,s}\rangle_1&=\{L\in \langle L_{r,s}\rangle|L\not\simeq M\oplus Z^*\ \textrm{for any}\ Z\in \cO_{\ue_3}\},\\
\langle L_{r,s}\rangle_2&=\{L\in \langle L_{r,s}\rangle|L\simeq M\oplus Z^*\ \textrm{for some}\ Z\in \cO_{\ue_3}\}.
\end{align*}
Then $V(\langle L_{r,s}\rangle,\cO_{\ue_3};M),V(\langle L_{r,s}\rangle_1,\cO_{\ue_3};M), V(\langle L_{r,s}\rangle_2,\cO_{\ue_3};M)$ have been defined in Subsection \ref{convolution}. Moreover, we define 
\begin{align*}
\tilde{V}^{\ue}_{(12)3}&=\bigsqcup_{(r,s)\in \cI^{\ue}_{\ue_1\ue_2}}V(\cO_{\ue_1},\cO_{\ue_2};L_{r,s})\times V(\langle L_{r,s}\rangle,\cO_{\ue_3};M),\\
\tilde{V}^{\ue}_{(12)3}(1)&=\bigsqcup_{(r,s)\in \cI^{\ue}_{\ue_1\ue_2}}V(\cO_{\ue_1},\cO_{\ue_2};L_{r,s})\times V(\langle L_{r,s}\rangle_1,\cO_{\ue_3};M),\\
\tilde{V}^{\ue}_{(12)3}(2)&=\bigsqcup_{(r,s)\in \cI^{\ue}_{\ue_1\ue_2}}V(\cO_{\ue_1},\cO_{\ue_2};L_{r,s})\times V(\langle L_{r,s}\rangle_2,\cO_{\ue_3};M),
\end{align*}
and so $\tilde{V}^{\ue}_{(12)3}=\tilde{V}^{\ue}_{(12)3}(1)\sqcup \tilde{V}^{\ue}_{(12)3}(2)$. Dually, for any $(r,s)\in \cI^{\ue'}_{\ue_2\ue_3}$, we fix a complex $L'_{r,s}\in \cL^{\ue'}_{\ue_2\ue_3}(r,s)=\{L'\in \cL^{\ue'}_{\ue_2\ue_3}|F^{L'}_{\cO_{\ue_2}\cO_{\ue_3}}=r,F^{L'}_{\cO_{\ue_3}\cO_{\ue_2}}=s\}$, and then simply denote by $\langle L'_{r,s}\rangle=\cL^{\ue'}_{\ue_2\ue_3}(r,s)$. We divide 
$$\langle L'_{r,s}\rangle=\langle L'_{r,s}\rangle_1\sqcup\langle L'_{r,s}\rangle_2$$
into disjoint union of $\rG_{\ue'}$-invariant constructible subsets, where
\begin{align*}
\langle L'_{r,s}\rangle_1&=\{L'\in \langle L'_{r,s}\rangle|L'\not\simeq M\oplus X^*\ \textrm{for any}\ X\in \cO_{\ue_1}\},\\
\langle L'_{r,s}\rangle_2&=\{L'\in \langle L'_{r,s}\rangle|L'\simeq M\oplus X^*\ \textrm{for some}\ X\in \cO_{\ue_1}\}.
\end{align*}
Moreover, we define
\begin{align*}
\tilde{V}^{\ue'}_{1(23)}&=\bigsqcup_{(r,s)\in \cI^{\ue'}_{\ue_2\ue_3}}V(\cO_{\ue_1},\langle L'_{r,s}\rangle;M)\times V(\cO_{\ue_2},\cO_{\ue_3},L'_{r,s}),\\
\tilde{V}^{\ue'}_{1(23)}(1)&=\bigsqcup_{(r,s)\in \cI^{\ue'}_{\ue_2\ue_3}}V(\cO_{\ue_1},\langle L'_{r,s}\rangle_1;M)\times V(\cO_{\ue_2},\cO_{\ue_3},L'_{r,s}),\\
\tilde{V}^{\ue'}_{1(23)}(2)&=\bigsqcup_{(r,s)\in \cI^{\ue'}_{\ue_2\ue_3}}V(\cO_{\ue_1},\langle L'_{r,s}\rangle_2;M)\times V(\cO_{\ue_2},\cO_{\ue_3},L'_{r,s}),
\end{align*}
and so $\tilde{V}^{\ue'}_{1(23)}=\tilde{V}^{\ue'}_{1(23)}(1)\sqcup\tilde{V}^{\ue'}_{1(23)}(2)$.

\subsubsection{\textbf{Case} $t=1$}\

For $(r,s)\in \cI^{\ue}_{\ue_1\ue_2}$, let $W^{\ue,\langle L_{r,s}\rangle_1}_{(12)3}(1)\subset W^{\ue}_{(12)3}(1)$ be the $\rG_{\ue_1}\times \rG_{\ue_2}\times \rG_{\ue_3}\times \rG_{\ue}$-invariant constructible subset consisting of $(\triangle\triangle)$ such that $L\in \langle L_{r,s}\rangle_1$, and 
$$V^{\ue,\langle L_{r,s}\rangle_1}_{(12)3}(1)=[W^{\ue,\langle L_{r,s}\rangle_1}_{(12)3}(1)/\rG_{\ue_1}\times \rG_{\ue_2}\times \rG_{\ue_3}\times \rG_{\ue}]$$
be the corresponding quotient stack. So
\begin{align*}
W^{\ue}_{(12)3}(1)=\bigsqcup_{(r,s)\in \cI^{\ue}_{\ue_1\ue_2}}W^{\ue,\langle L_{r,s}\rangle_1}_{(12)3}(1),\ V^{\ue}_{(12)3}(1)=\bigsqcup_{(r,s)\in \cI^{\ue}_{\ue_1\ue_2}}V^{\ue,\langle L_{r,s}\rangle_1}_{(12)3}(1).
\end{align*}
Dually, for $(r,s)\in \cI^{\ue'}_{\ue_2\ue_3}$, let $W^{\ue',\langle L'_{r,s}\rangle_1}_{1(23)}(1)\subset W^{\ue'}_{1(23)}(1)$ be the $\rG_{\ue_1}\times \rG_{\ue_2}\times \rG_{\ue_3}\times \rG_{\ue'}$-invariant constructible subset consisting of $(\triangle\triangle')$ such that $L'\in \langle L'_{r,s}\rangle_1$ and then
\begin{align*}
&V^{\ue',\langle L'_{r,s}\rangle_1}_{1(23)}(1)=[W^{\ue',\langle L'_{r,s}\rangle_1}_{1(23)}(1)/\rG_{\ue_1}\times \rG_{\ue_2}\times \rG_{\ue_3}\times \rG_{\ue'}],\\
W^{\ue'}_{1(23)}(1)&=\bigsqcup_{(r,s)\in \cI^{\ue'}_{\ue_2\ue_3}}W^{\ue',\langle L'_{r,s}\rangle_1}_{1(23)}(1),\ V^{\ue'}_{1(23)}(1)=\bigsqcup_{(r,s)\in \cI^{\ue'}_{\ue_2\ue_3}}V^{\ue',\langle L'_{r,s}\rangle_1}_{1(23)}(1).
\end{align*}

\begin{proposition}\label{Case 1}
For any $\ue\leqslant \ue_1+\ue_2$ and $(r,s)\in \cI^{\ue}_{\ue_1\ue_2}$, we have
$$\chi(V^{\ue,\langle L_{r,s}\rangle_1}_{(12)3}(1))=F^{L_{r,s}}_{\cO_{\ue_1}\cO_{\ue_2}}F^M_{\langle L_{r,s}\rangle_1 \cO_{\ue_3}}.$$
As a consequence, we have
$$\chi(V^{\ue}_{(12)3}(1))=\chi(\tilde{V}^{\ue}_{(12)3}(1)).$$
Dually, for any $\ue'\leqslant \ue_2+\ue_3$ and $(r,s)\in \cI^{\ue'}_{\ue_2\ue_3}$, we have
\begin{align*}
\chi(V^{\ue',\langle L'_{r,s}\rangle_1}_{1(23)}(1))&=F^M_{\cO_{\ue_1}\langle L'_{r,s}\rangle_1}F^{L'_{r,s}}_{\cO_{\ue_2} \cO_{\ue_3}},\\ 
\chi(V^{\ue'}_{1(23)}(1))&=\chi(\tilde{V}^{\ue'}_{1(23)}(1)).
\end{align*}
\end{proposition}
\begin{proof}
By definition, there is a natural surjective projection
\begin{align*}
&\pi:W^{\ue,\langle L_{r,s}\rangle_1}_{(12)3}(1)\rightarrow W(\langle L_{r,s}\rangle_1,\cO_{\ue_3};M)\\
(Y\xrightarrow{f}L\xrightarrow{g}X&\xrightarrow{h}Y^*,\ Z\xrightarrow{i}M\xrightarrow{j}L\xrightarrow{k}Z^*)\mapsto (Z\xrightarrow{i}M\xrightarrow{j}L\xrightarrow{k}Z^*)
\end{align*}
which induces a morphism
$$\overline{\pi}:V^{\ue,\langle L_{r,s}\rangle_1}_{(12)3}(1)\rightarrow [W(\langle L_{r,s}\rangle_1,\cO_{\ue_3};M)/\rG_{\ue_3}\times \rG_{\ue}]=V(\langle L_{r,s}\rangle_1,\cO_{\ue_3};M),$$
where $\chi(V(\langle L_{r,s}\rangle_1,\cO_{\ue_3};M))=F^M_{\langle L_{r,s}\rangle_1 \cO_{\ue_3}}$. Consider the fiber of $\overline{\pi}$ at the geometric point corresponding to the orbit $\rG_{\ue_3}\times \rG_{\ue}.(Z\xrightarrow{i}M\xrightarrow{j}L\xrightarrow{k}Z^*)$, by definition, the fiber is isomorphic to $[W(\cO_{\ue_1},\cO_{\ue_2};L)/\rG_{\ue_1}\times \rG_{\ue_2}\times \Stab(i,j,k)]$, where 
$$\Stab(i,j,k)=\{(\gamma,\lambda)\in \rG_{\ue_3}\times \rG_{\ue}|\gamma.Z=Z,\lambda.L=L, i\gamma^{-1}=i,\lambda j=j,\gamma^*k\lambda^{-1}=k\}$$ 
is the stabilizer and $\rG_{\ue_1}\times \rG_{\ue_2}\times \Stab(i,j,k)$ acts on $W(\cO_{\ue_1},\cO_{\ue_2};L)$ via 
$$(\alpha,\beta,\gamma,\lambda).(Y\xrightarrow{f}L\xrightarrow{g}X\xrightarrow{h}Y^*)=(\beta.Y\xrightarrow{\lambda f\beta^{-1}}L\xrightarrow{\alpha g\lambda^{-1}}\alpha.X\xrightarrow{\beta^*h\alpha^{-1}}(\beta.Y)^*).$$
We claim that 
\begin{equation}\label{claim}
\begin{aligned}
&\chi([W(\cO_{\ue_1},\cO_{\ue_2};L)/\rG_{\ue_1}\times \rG_{\ue_2}\times \Stab(i,j,k)])\\
=&\chi([W(\cO_{\ue_1},\cO_{\ue_2};L)/\rG_{\ue_1}\times \rG_{\ue_2}])=F^L_{\cO_{\ue_1}\cO_{\ue_2}}=F^{L_{r,s}}_{\cO_{\ue_1}\cO_{\ue_2}}
\end{aligned}
\end{equation}
which is constant. Once the claim holds, applying Corollary \ref{fibre naive Euler characteristic} for $\overline{\pi}$, we obtain $\chi(V^{\ue,\langle L_{r,s}\rangle}_{(12)3})=F^{L_{r,s}}_{\cO_{\ue_1}\cO_{\ue_2}}F^M_{\langle L_{r,s}\rangle_1 \cO_{\ue_3}}$, moreover, 
\begin{align*}
\chi(V^{\ue}_{(12)3}(1))=&\sum_{(r,s)\in \cI^{\ue}_{\ue_1\ue_2}}\chi(V^{\ue,\langle L_{r,s}\rangle_1}_{(12)3}(1))=\sum_{(r,s)\in \cI^{\ue}_{\ue_1\ue_2}}F^{L_{r,s}}_{\cO_{\ue_1}\cO_{\ue_2}}F^M_{\langle L_{r,s}\rangle_1 \cO_{\ue_3}}\\
=&\sum_{(r,s)\in \cI^{\ue}_{\ue_1\ue_2}}\chi(V(\cO_{\ue_1},\cO_{\ue_2};L_{r,s})\times V(\langle L_{r,s}\rangle_1,\cO_{\ue_3};M))=\chi(\tilde{V}^{\ue}_{(12)3}(1)),
\end{align*}
as desired. We prove (\ref{claim}) as follows.

The natural morphism $\rG_{\ue_1}\times \rG_{\ue_2}.(f,g,h)\mapsto \rG_{\ue_1}\times \rG_{\ue_2}\times \Stab(i,j,k).(f,g,h)$ induces a morphism
\begin{align*}
\rho:[W(\cO_{\ue_1},\cO_{\ue_2};L)/\rG_{\ue_1}\times \rG_{\ue_2}]\rightarrow [W(\cO_{\ue_1},\cO_{\ue_2};L)/\rG_{\ue_1}\times \rG_{\ue_2}\times \Stab(i,j,k)]
\end{align*}
whose fiber at the geometric point corresponding to $\rG_{\ue_1}\times \rG_{\ue_2}\times \Stab(i,j,k).(f,g,h)$ is isomorphic to
$$\Stab(i,j,k).(f,g,h)\cong \Stab(i,j,k)/\Stab_{(i,j,k)}(f,g,h),$$
where $\Stab_{(i,j,k)}(f,g,h)=\{(\gamma,\lambda) \in \Stab(i,j,k)|\lambda f=f,g\lambda^{-1}=g\}$ is the stabilizer. We calculate $\chi(\Stab(i,j,k)/\Stab_{(i,j,k)}(f,g,h))$ as follows. By similar argument as in Remark \ref{reduce to radical}, the orbit $\Stab(i,j,k).(f,g,h)$ does not change while we replace $Z$ by $Z_r$, and so we may assume $Z$ is a radical complex. By definition, $L\in \cL^{\ue}_{\ue_1\ue_2}$ is already a radical complex. Consider the conditions in $\Stab(i,j,k)$, note that $\gamma.Z=Z,\lambda.L=L$ imply that $\gamma\in\Stab_{\rG_{\ue_3}}(Z)\cong \Aut_{\cC_2(\cP)}(Z),\lambda\in \Stab_{\rG_{\ue}}(L)\cong \Aut_{\cC_2(\cP)}(L)$, and note that $i\gamma^{-1}=i,\lambda j=j,\gamma^*k\lambda^{-1}=k$ are about morphisms in $\cK_2(\cP)$. By Corollary \ref{automorphism groups coincide}, there is a surjective group homomorphism
\begin{align*}
p:&\Stab(i,j,k)=\{(\gamma,\lambda)\in \Aut_{\cC_2(\cP)}(Z)\times \Aut_{\cC_2(\cP)}(L)|i\gamma^{-1}=i,\lambda j=j,\gamma^*k\lambda^{-1}=k\}\\
\twoheadrightarrow &\overline{\Stab}(i,j,k)
=\{(\gamma,\lambda)\in \Aut_{\cK_2(\cP)}(Z)\times \Aut_{\cK_2(\cP)}(L)|i\gamma^{-1}=i,\lambda j=j,\gamma^*k\lambda^{-1}=k\}
\end{align*}
with kernel $\{(1_Z+\xi,Z_L+\zeta)|\xi\in \Htp(Z,Z),\zeta\in \Htp(L,L)\}$ which is bijective to the vector space $\Htp(Z,Z)\oplus \Htp(L,L)$. There is a natural bijection
\begin{align*}
\theta:&\overline{\Stab}(i,j,k)\\
\rightarrow &\overline{S}(i,j,k)=\{(\gamma-1_Z,\lambda-1_L)\in \End_{\cK_2(\cP)}(Z)\times \End_{\cK_2(\cP)}(L)|(\gamma,\lambda)\in \overline{\Stab}(i,j,k)\}.
\end{align*}
Next, we prove that 
\begin{equation}
\overline{S}(i,j,k)\subset \End_{\cK_2(\cP)}(Z)\times \End_{\cK_2(\cP)}(L)\ \textrm{is a vector subspace.} \tag{$\ast$}
\end{equation}
It is clear that $(0,0)\in \overline{S}(i,j,k)$. For any $(\gamma_1,\lambda_1),(\gamma_2,\lambda_2)\in \overline{\Stab}(i,j,k)$ and $k\in \bbC$, consider the element
\begin{align*}
(\gamma,\lambda)=(1_Z,1_L)+(\gamma_1-1_Z,\lambda_1-1_L)+k(\gamma_2-1_Z,\lambda_2-1_L),
\end{align*}
it is easy to check that $i=i\gamma,\lambda j=j,\gamma^*k=k\lambda$. Since $Z\in \cO_{\ue_3}$ is indecomposable as an object in $\cK_2(\cP)$, the endomorphisms $\gamma_1-1_Z, \gamma_2-1_Z\in \End_{\cK_2(\cP)}(Z)$ are either homotopy equivalences or nilpotent morphisms. Since $L\not\simeq M\oplus Z^*$, the distinguished triangle $Z\xrightarrow{i}M\xrightarrow{j}L\xrightarrow{k}Z^*$ does not split, thus $i\not=0$, and so $i(\gamma_1-1_Z)=i(\gamma_2-1_Z)=0$ imply that $\gamma_1-1_Z, \gamma_2-1_Z$ are nilpotent, and then so is $\gamma_1-1_Z+k(\gamma_2-1_Z)$. Hence $\gamma=1_Z+(\gamma_1-1_Z)+k(\gamma_2-1_Z)\in \Aut_{\cK_2(\cP)}(Z)$, then by the following commutative diagram about distinguished triangles
\begin{diagram}[midshaft,size=2em]
Z &\rTo^{i} &M &\rTo^{j} &L &\rTo^{k} &Z^*\\
\dTo^{\gamma} & &\vEq & &\dTo^{\lambda} & &\dTo^{\gamma^*}\\
Z &\rTo^{i} &M &\rTo^{j} &L &\rTo^{k} &Z^*,
\end{diagram}
we have $\lambda\in \Aut_{\cK_2(\cP)}(L)$ and finish the proof of $(\ast)$. Under the group morphism $p$ and isomorphism $\theta$, the image $\theta p(\Stab_{(i,j,k)}(f,g,h))\subset \overline{S}(i,j,k)$ is a vector subspace. Hence,
$$\chi(\Stab(i,j,k)/\Stab_{(i,j,k)}(f,g,h))=\chi(\overline{S}(i,j,k)/\theta p(\Stab_{(i,j,k)}(f,g,h)))=1$$
is constant. Applying Corollary \ref{fibre naive Euler characteristic} for the morphism $\rho$, the claim (\ref{claim}) holds.

The dual statement can be proved similarly.
\end{proof}

\subsubsection{\textbf{Case} $t=2$}\

In the case $L\simeq M\oplus Z^*\not\simeq X\oplus Y$ for some $Z\in \cO_{\ue_3}$ and any $X\in \cO_{\ue_1},Y\in \cO_{\ue_2}$. Without loss of generality, we assume $L=M\oplus Z^*$ and denote by $\cO_Z=\rG_{\ue_3}.Z\subset \cO_{\ue_3}$.

Consider the constructible set $W(M\oplus Z^*,\cO_Z;M)$ which consists of $(i,j,k)$ such that $\gamma.Z\xrightarrow{i}M\xrightarrow{j}M\oplus Z^*\xrightarrow{k}(\gamma.Z)^*$ is a distinguished triangle, by similar argument as in Remark \ref{reduce to radical}, we may assume $M,Z$ are radical complexes. Then the group $\rG_{\ue_3}\times \Aut_{\cK_2(\cP)}(M\oplus Z^*)$ acts on $W(M\oplus Z^*,\cO_Z;M)$ via
\begin{align*}
&(\gamma',\lambda).(\gamma.Z\xrightarrow{i}M\xrightarrow{j}M\oplus Z^*\xrightarrow{k}(\gamma.Z)^*)\\
=&(\gamma'\gamma.Z\xrightarrow{i\gamma'^{-1}}M\xrightarrow{\lambda j}M\oplus Z^*\xrightarrow{\gamma'^*k\lambda^{-1}}(\gamma'\gamma.Z)^*).
\end{align*}
Note that any distinguished triangle $\gamma.Z\xrightarrow{i}M\xrightarrow{j}M\oplus Z^*\xrightarrow{k}(\gamma.Z)^*$ splits, that is, there exists homotopy equivalences $\gamma',\zeta,\lambda$ such that the following diagram commutes
\begin{diagram}[midshaft,size=2em]
\gamma.Z &\rTo^{i} &M &\rTo^{j=\begin{pmatrix}\begin{smallmatrix}j_1\\j_2\end{smallmatrix}\end{pmatrix}} &M\oplus Z^* &\rTo^{k=\begin{pmatrix}\begin{smallmatrix}k_1&k_2\end{smallmatrix}\end{pmatrix}} &(\gamma.Z)^*\\
\dTo^{\gamma'} & &\dTo^{\zeta} & &\dTo^{\lambda=\begin{pmatrix}\begin{smallmatrix}\lambda_{11}&\lambda_{12}\\\lambda_{21}&\lambda_{22}\end{smallmatrix}\end{pmatrix}} & &\dTo^{\gamma'^*}\\
Z &\rTo^{0} &M &\rTo^{\begin{pmatrix}\begin{smallmatrix}1\\0\end{smallmatrix}\end{pmatrix}} &M\oplus Z^* &\rTo^{\begin{pmatrix}\begin{smallmatrix}0&1\end{smallmatrix}\end{pmatrix}} &Z^*,
\end{diagram}
then 
\begin{align*}
&(\gamma',\begin{pmatrix}\begin{smallmatrix}\zeta^{-1} &\\ &1_{Z^*}\end{smallmatrix}\end{pmatrix}\lambda).(\gamma.Z\xrightarrow{i}M\xrightarrow{j}M\oplus Z^*\xrightarrow{k}(\gamma.Z)^*)\\
=&(Z\xrightarrow{0}M\xrightarrow{\begin{pmatrix}\begin{smallmatrix}1\\0\end{smallmatrix}\end{pmatrix}}M\oplus Z^*\xrightarrow{\begin{pmatrix}\begin{smallmatrix}0 &1\end{smallmatrix}\end{pmatrix}}Z^*).
\end{align*}
Since $\gamma':\gamma.Z\rightarrow Z$ is a homotopy equivalence between radical complexes, by Lemma \ref{isomorphism and homotopy equivalence}, it is an isomorphism, and so $\gamma'\in \rG_{\ue_3}$. Thus there is a unique orbit 
$$\rG_{\ue_3}\times \Aut_{\cK_2(\cP)}(M\oplus Z^*).(Z\xrightarrow{0}M\xrightarrow{\begin{pmatrix}\begin{smallmatrix}1\\0\end{smallmatrix}\end{pmatrix}}M\oplus Z^*\xrightarrow{\begin{pmatrix}\begin{smallmatrix}0 &1\end{smallmatrix}\end{pmatrix}}Z^*).$$ Its stabilizer is  
$$S=\{(\gamma,\begin{pmatrix}\begin{smallmatrix}1_{M} &\eta\\0 &\gamma^*\end{smallmatrix}\end{pmatrix})\in\Aut_{\cC_2(\cP)}(Z)\times \Aut_{\cK_2(\cP)}(M\oplus Z^*)|\eta\in \Hom_{\cK_2(\cP)}(Z^*,M)\}.$$
By Corollary \ref{automorphism groups coincide}, there is a surjective group homomorphism
\begin{align*}
p:S\twoheadrightarrow\! \overline{S}\!=\!\{(\gamma,\begin{pmatrix}\begin{smallmatrix}1_{M} &\eta\\0 &\gamma^*\end{smallmatrix}\end{pmatrix})\in\Aut_{\cK_2(\cP)}(Z)\!\times\! \Aut_{\cK_2(\cP)}(M\oplus Z^*)|\eta\in \Hom_{\cK_2(\cP)}(Z^*,M)\}
\end{align*}
with kernel which is bijective to the vector space $\Htp(Z,Z)$.

Analogue to $V^{\tilde{M}}_{(\hat{\cO}_1\hat{\cO}_2)\hat{\cO}_3}$, the subset $W(\cO_{\ue_1},\cO_{\ue_2};M\oplus Z^*)\times W(M\oplus Z^*,\cO_Z;M)$ has a $\rG_{\ue_1}\times \rG_{\ue_2}\times \rG_{\ue_3}\times \Aut_{\cK_2(\cP)}(M\oplus Z^*)$-action given by
\begin{align*}
&(\alpha,\beta,\gamma',\lambda).(Y\xrightarrow{f}M\oplus Z^*\xrightarrow{g}X\xrightarrow{h}Y^*,\gamma.Z\xrightarrow{i}M\xrightarrow{j}M\oplus Z^*\xrightarrow{k}(\gamma.Z)^*)\\
=&(\beta.Y\xrightarrow{\lambda f\beta^{-1}}M\oplus Z^*\xrightarrow{\alpha g\lambda^{-1}}\alpha.X\xrightarrow{\beta^*h\alpha^{-1}}(\beta.Y)^*,\\
&\ \gamma'\gamma.Z\xrightarrow{i\gamma'^{-1}}M\xrightarrow{\lambda j}M\oplus Z^*\!\xrightarrow{\gamma'^*k\lambda^{-1}}(\gamma'\gamma.Z)^*).
\end{align*}
We denote by $V^M_{(\cO_{\ue_1}\cO_{\ue_2})\cO_Z}$ the corresponding quotient stack.  Dually, the constructible set $W(\cO_Z,M\oplus Z^*;M)$ also has a unique $\rG_{\ue_3}\times \Aut_{\cK_2(\cP)}(M\oplus Z^*)$-orbit with stabilizer $S$. Analogue to $V^{\tilde{M}}_{\hat{\cO}_3(\hat{\cO}_1\hat{\cO}_2)}$, the subset $W(\cO_Z,M\oplus Z^*;M)\times W(\cO_{\ue_1},\cO_{\ue_2};M\oplus Z^*)$ has a $\rG_{\ue_1}\times \rG_{\ue_2}\times \rG_{\ue_3}\times \Aut_{\cK_2(\cP)}(M\oplus Z^*)$-action with quotient stack denoted by $V^M_{\cO_Z(\cO_{\ue_1}\cO_{\ue_2})}$.

\begin{proposition}\label{Case 2}
For any $\ue\leqslant \ue_1+\ue_2$, we have
\begin{align*}
\chi(V^M_{(\cO_{\ue_1}\cO_{\ue_2})\cO_Z})=\chi([W(\cO_{\ue_1},\cO_{\ue_2};M\oplus Z^*)/\rG_{\ue_1}\times \rG_{\ue_2}\times S])=\chi(V^M_{\cO_Z(\cO_{\ue_1}\cO_{\ue_2})}),
\end{align*}
where $\rG_{\ue_1}\times \rG_{\ue_2}\times S$ acts on $W(\cO_{\ue_1},\cO_{\ue_2};M\oplus Z^*)$ via
\begin{align*}
&(\alpha,\beta,\gamma,\begin{pmatrix}\begin{smallmatrix}1_{M} &\eta\\0 &\gamma^*\end{smallmatrix}\end{pmatrix}).(Y\xrightarrow{f}M\oplus Z^*\xrightarrow{g}X\xrightarrow{h}Y^*)\\
=&(\beta.Y\xrightarrow{\begin{pmatrix}\begin{smallmatrix}1_{M} &\eta\\0 &\gamma^*\end{smallmatrix}\end{pmatrix}f\beta^{-1}}M\oplus Z^*\xrightarrow{\alpha g\begin{pmatrix}\begin{smallmatrix}1_{M} &\eta\\0 &\gamma^*\end{smallmatrix}\end{pmatrix}^{-1}}\alpha.X\xrightarrow{\beta^*h\alpha^{-1}}(\beta.Y)^*).
\end{align*}
As a consequence, we have
\begin{equation}\label{V(2)}
\begin{aligned}  
\chi(V^{\ue}_{(12)3}(2))&=\chi(V^{\ue}_{3(12)}(2)),\\
F^{M\oplus Z^*}_{\cO_{\ue_1}\cO_{\ue_2}}&=0.
\end{aligned}
\end{equation}
\end{proposition}
\begin{proof}
By definition, there is a natural surjective projection
\begin{align*}
\pi:W(\cO_{\ue_1},\cO_{\ue_2};M\oplus Z^*)\times W(M\oplus Z^*,\cO_Z;M)&\rightarrow W(M\oplus Z^*,\cO_Z;M)\\
((f,g,h),(i,j,k))&\mapsto (i,j,k)
\end{align*} 
which induces a morphism
$$\overline{\pi}:V^M_{(\cO_{\ue_1}\cO_{\ue_2})\cO_Z}\rightarrow [W(M\oplus Z^*,\cO_Z;M)/\rG_{\ue_3}\times \Aut_{\cK_2(\cP)}(M\oplus Z^*)]=\{\textrm{point}\}$$
whose fiber at the unique point is isomorphic to $[W(\cO_{\ue_1},\cO_{\ue_2};M\oplus Z^*)/\rG_{\ue_1}\times \rG_{\ue_2}\times S]$.
Applying Corollary \ref{fibre naive Euler characteristic}, we obtain 
$$\chi(V^M_{(\cO_{\ue_1}\cO_{\ue_2})\cO_Z})=\chi([W(\cO_{\ue_1},\cO_{\ue_2};M\oplus Z^*)/\rG_{\ue_1}\times \rG_{\ue_2}\times S]).$$
Dually, consider the natural surjective projection
$$\pi':W(\cO_Z,M\oplus Z^*;M)\times W(\cO_{\ue_1},\cO_{\ue_2};M\oplus Z^*)\rightarrow W(\cO_Z,M\oplus Z^*;M),$$
by similar argument, we obtain
$$\chi(V^M_{\cO_Z(\cO_{\ue_1}\cO_{\ue_2})})=\chi([W(\cO_{\ue_1},\cO_{\ue_2};M\oplus Z^*)/\rG_{\ue_1}\times \rG_{\ue_2}\times S]).$$
Therefore, $Z\mapsto \chi(V^M_{(\cO_{\ue_1}\cO_{\ue_2})\cO_Z})=\chi(V^M_{\cO_Z(\cO_{\ue_1}\cO_{\ue_2})})$ defines a $\rG_{\ue_3}$-invariant constructible function on $\cO_{\ue_3}$, and so
\begin{align*}
\chi(V^{\ue}_{(12)3}(2))=\int_{[\cO_{\ue_3}/\rG_{\ue_3}]}\chi(V^M_{(\cO_{\ue_1}\cO_{\ue_2})\cO_Z})=\int_{[\cO_{\ue_3}/\rG_{\ue_3}]}\chi(V^M_{\cO_Z(\cO_{\ue_1}\cO_{\ue_2})})=\chi(V^{\ue}_{3(12)}(2)).
\end{align*}

The natural morphism $\rG_{\ue_1}\times \rG_{\ue_2}.(f,g,h)\mapsto \rG_{\ue_1}\times\rG_{\ue_2}\times S.(f,g,h)$ induces a morphism 
$$\rho:[W(\cO_{\ue_1},\cO_{\ue_2};M\oplus Z^*)/\rG_{\ue_1}\times \rG_{\ue_2}]\rightarrow [W(\cO_{\ue_1},\cO_{\ue_2};M\oplus Z^*)/\rG_{\ue_1}\times \rG_{\ue_2}\times S]$$
whose fiber at the geometric point corresponding to the orbit $\rG_{\ue_1}\times \rG_{\ue_2}\times S.(f,g,h)$ is isomorphic to
$$S.(f,g,h)\cong S/\Stab_S(f,g,h),$$
where 
\begin{align*}
\Stab_S(f,g,h)=&\{(\gamma,\begin{pmatrix}\begin{smallmatrix}1_{M} &\eta\\0 &\gamma^*\end{smallmatrix}\end{pmatrix})\in S|\begin{pmatrix}\begin{smallmatrix}1_{M} &\eta\\0 &\gamma^*\end{smallmatrix}\end{pmatrix}f=f,g\begin{pmatrix}\begin{smallmatrix}1_{M} &\eta\\0 &\gamma^*\end{smallmatrix}\end{pmatrix}^{-1}=g\}\\
=&\{(\gamma,\begin{pmatrix}\begin{smallmatrix}1_{M} &\eta\\0 &\gamma^*\end{smallmatrix}\end{pmatrix})\in S|\eta f_2=0,\gamma^*f_2=f_2,g_1\eta+g_2\gamma^*=g_2\}                                                                                                                                                                                                                                                                                                                                                                                                                                                                                                                                                                                                                                                                                                                                                                                                                                                                                                                                                                                                                                                                                                        
\end{align*}
is the stabilizer and $f=\begin{pmatrix}\begin{smallmatrix}f_1\\f_2\end{smallmatrix}\end{pmatrix},g=\begin{pmatrix}\begin{smallmatrix}g_1&g_2\end{smallmatrix}\end{pmatrix}$. We calculate $\chi(\Stab_S(f,g,h))$ as follows. Under the surjective group homomorphism $p:S\twoheadrightarrow \overline{S}$, we denote by the image $p(\Stab_S(f,g,h))=\overline{\Stab}_{\overline{S}}(f,g,h)$, and then $\chi(\Stab_S(f,g,h))=\chi(\overline{\Stab}_{\overline{S}}(f,g,h))$, since the kernel of $p$ is bijective to a vector space. There is a natural bijection
\begin{align*}
\theta:\overline{\Stab}_{\overline{S}}(f,g,h)\rightarrow \overline{S}(f,g,h)
=\{(\gamma-1_{Z},\begin{pmatrix}\begin{smallmatrix}0 &\eta\\0 &\gamma^*-1_{Z^*}\end{smallmatrix}\end{pmatrix})|(\gamma,\begin{pmatrix}\begin{smallmatrix}1_{M} &\eta\\0 &\gamma^*\end{smallmatrix}\end{pmatrix})\in \overline{\Stab}_{\overline{S}}(f,g,h)\}.
\end{align*}
Next, we prove that 
\begin{equation}
\overline{S}(f,g,h)\subset \End_{\cK_2(\cP)}(Z)\times \End_{\cK_2(\cP)}(M\oplus Z^*)\ \textrm{is a vector subspace.} \tag{$\ast\ast$}
\end{equation}
It is clear that $(0,0)\in \overline{S}(f,g,h)$. For any $(\gamma_1,\begin{pmatrix}\begin{smallmatrix}1 &\eta_1\\0 &\gamma_1^*\end{smallmatrix}\end{pmatrix}),(\gamma_2,\begin{pmatrix}\begin{smallmatrix}1 &\eta_2\\0 &\gamma_2^*\end{smallmatrix}\end{pmatrix})\in \overline{\Stab}_{\overline{S}}(f,g,h)$ and $k\in \bbC$, consider the element
$$(\gamma,\begin{pmatrix}\begin{smallmatrix}1 &\eta\\0 &\gamma^*\end{smallmatrix}\end{pmatrix})=(1_{Z},\begin{pmatrix}\begin{smallmatrix}1_{M} &0\\0 &1_{Z^*}\end{smallmatrix}\end{pmatrix})+(\gamma_1-1_{Z},\begin{pmatrix}\begin{smallmatrix}0 &\eta_1\\0 &\gamma_1^*-1_{Z^*}\end{smallmatrix}\end{pmatrix})+k(\gamma_2-1_{Z},\begin{pmatrix}\begin{smallmatrix}0 &\eta_2\\0 &\gamma_2^*-1_{Z^*}\end{smallmatrix}\end{pmatrix}),$$
it is easy to check $\eta f_2=0,\gamma^*f_2=f_2,g_1\eta+g_2\gamma^*=g_2$. Since $Z$ is indecomposable as an object in $\cK_2(\cP)$, the endomorphisms $\gamma_1-1_{Z},\gamma_2-1_{Z}\in \End_{\cK_2(\cP)}(Z)$ are either homotopy equivalences or nilpotent morphisms. If $\gamma_1-1_{Z}$ is an isomorphism, then so is $\gamma_1^*-1_{Z^*}$, and $(\gamma_1^*-1_{Z^*})f_2=0$ implies that $f_2=0$. By \cite[Lemma 2]{Peng-Xiao-2000}, the distinguished triangle $Y\xrightarrow{\begin{pmatrix}\begin{smallmatrix}f_1\\0\end{smallmatrix}\end{pmatrix}}M\oplus Z^*\xrightarrow{\begin{pmatrix}\begin{smallmatrix}g_1&g_2\end{smallmatrix}\end{pmatrix}}X\xrightarrow{h}Y^*$ is isomorphic to the direct sum of distinguished triangles $Y\xrightarrow{f_1}M\xrightarrow{g'}X'\xrightarrow {h'}Y^*$ and $0\rightarrow Z^*\xrightarrow{1_{Z^*}}Z^*\rightarrow 0$ for some $X'$. In particular, $X\simeq X'\oplus Z^*$. Since $X,Z$ are indecomposable as objects in $\cK_2(\cP)$, we have $X\simeq Z^*,X'\simeq 0$, and then $f_1$ is a homotopy equivalence and $Y\simeq M$, which is a contradiction to $L\simeq M\oplus Z^*\not\simeq X\oplus Y$. Thus the endomorphism $\gamma_1-1_{Z}$ is nilpotent, similarly, so is $\gamma_2-1_{Z}$, and then so is $(\gamma_1-1_{Z})+k(\gamma_2-1_{Z})$. Hence we have $\gamma=1_{Z}+(\gamma_1-1_{Z})+k(\gamma_2-1_{Z})\in\Aut_{\cK_2(\cP)}(Z)$ and finish the proof of $(\ast\ast)$, so
$$\chi(\Stab_S(f,g,h))=\chi(\overline{\Stab}_{\overline{S}}(f,g,h))=\chi(\overline{S}(f,g,h))=1$$
is constant. Applying Corollary \ref{fibre naive Euler characteristic} for the natural projection $S\rightarrow S/\Stab_S(f,g,h)$, we obtain 
$$\chi(S.(f,g,h))=\chi(S/\Stab_S(f,g,h))=\chi(S).$$
By the fact that $\chi(\Aut_{\cC_2(\cP)}(Z))=0$ and $\chi(\Hom_{\cK_2(\cP)}(Z^*,M))=1$, see Lemma \ref{Upsilon(Aut)} and Proposition \ref{Euler characteristic}, we have
$$\chi(S)=\chi(\Aut_{\cC_2(\cP)}(Z)\times \Hom_{\cK_2(\cP)}(Z^*,M))=0$$
which is constant. Applying Corollary \ref{fibre naive Euler characteristic} for the morphism $\rho$, we obtain
\begin{align*}
F^{M\oplus Z^*}_{\cO_{\ue_1}\cO_{\ue_2}}=\chi([W(\cO_{\ue_1},\cO_{\ue_2};M\oplus Z^*)/\rG_{\ue_1}\times \rG_{\ue_2}])=0,
\end{align*}
as desired.
\end{proof}

\subsubsection{\textbf{Case} $t=3$}\

In the case $L\simeq M\oplus Z^*\simeq X\oplus Y, X\not\simeq Y$ for some $X\in \cO_{\ue_1},Y\in \cO_{\ue_2},Z\in \cO_{\ue_3}$. Without loss of generality, we assume that $L=M\oplus Z^*$, and denote by $\cO_X=\rG_{\ue_1}.X\subset \cO_{\ue_1},\cO_Y=\rG_{\ue_2}.Y\subset \cO_{\ue_2},\cO_Z=\rG_{\ue_3}.Z\subset \cO_{\ue_3}$.

The constructible set $W(M\oplus Z^*,\cO_Z;M)$ has been studied in 7.5.2. Consider the constructible set $W(\cO_X,\cO_Y;X\oplus Y)$ which has a $\rG_{\ue_1}\times \rG_{\ue_2}$-action, by similar argument as in Remark \ref{reduce to radical}, we may assume $X,Y,Z,M$ are radical complexes. Note that any distinguished triangle $\beta.Y\xrightarrow{f}X\oplus Y\xrightarrow{g}\alpha.X\xrightarrow{h}(\beta.Y)^*$ splits, thus $f=\begin{pmatrix}\begin{smallmatrix}f_1\\f_2\end{smallmatrix}\end{pmatrix}$ is a splitting monomorphism,  $g=\begin{pmatrix}\begin{smallmatrix}g_1 &g_2\end{smallmatrix}\end{pmatrix}$ is a splitting epimorphism and $h=0$. Since $X\not\simeq Y$ are indecomposable, $f_2,g_1$ must be homotopy equivalences, then by Lemma \ref{isomorphism and homotopy equivalence}, they are isomorphisms. Hence, there exists $g_1^{-1}\in \rG_{\ue_1},f_2\in \rG_{\ue_2}$ such that  
\begin{align*}
&(g_1^{-1},f_2).(\beta.Y \xrightarrow{f=\begin{pmatrix}\begin{smallmatrix}f_1\\f_2\end{smallmatrix}\end{pmatrix}} X\oplus Y\xrightarrow{g=\begin{pmatrix}\begin{smallmatrix}g_1 &g_2\end{smallmatrix}\end{pmatrix}} \alpha.X\xrightarrow{0} (\beta.Y)^*)\\
=&(Y \xrightarrow{\begin{pmatrix}\begin{smallmatrix}f_1f_2^{-1}\\1\end{smallmatrix}\end{pmatrix}} X\oplus Y\xrightarrow{\begin{pmatrix}\begin{smallmatrix}1 &g_1^{-1}g_2\end{smallmatrix}\end{pmatrix}} X\xrightarrow{0} Y^*),
\end{align*} 
where $f_1f_2^{-1}+g_1^{-1}g_2=0$ follows from $gf=0$. Therefore, any $\rG_{\ue_1}\times \rG_{\ue_2}$-orbit in $W(\cO_X,\cO_Y;X\oplus Y)$ is of the form
$$\rG_{\ue_1}\times \rG_{\ue_2}.(Y\xrightarrow{\begin{pmatrix}\begin{smallmatrix}\theta\\ 1\end{smallmatrix}\end{pmatrix}} X\oplus Y\xrightarrow{\begin{pmatrix}\begin{smallmatrix}1&-\theta \end{smallmatrix}\end{pmatrix}}X\xrightarrow{0}Y^*),$$
where $\theta\in \Hom_{\cK_2(\cP)}(Y,X)$. It is clear that each stabilizer is trivial. Hence
\begin{equation}\label{V(Xoplus Y)=1}
\begin{aligned}
&V(\cO_X,\cO_Y;X\oplus Y)\cong \Hom_{\cK_2(\cP)}(Y,X),\\
&F^{X\oplus Y}_{\cO_X\cO_Y}=\chi(V(\cO_X,\cO_Y;X\oplus Y))=1.
\end{aligned}
\end{equation}
Similarly, since any distinguished triangle $\beta.Y\xrightarrow{\begin{pmatrix}\begin{smallmatrix}f_1\\ f_2\end{smallmatrix}\end{pmatrix}} M\oplus Z^*\xrightarrow{\begin{pmatrix}\begin{smallmatrix}g_1&g_2 \end{smallmatrix}\end{pmatrix}}\alpha.X\xrightarrow{h}(\beta.Y)^*$ in $W(\cO_X,\cO_Y;M\oplus Z^*)$ splits, we have $h=0$ and either $f_1,g_2$ or $f_2,g_1$ are homotopy equivalences, and so either $Y\simeq M,X\simeq Z^*$ or $Y\simeq Z^*,X\simeq M$.
 
Analogue to Subsection 7.5.2, the subset $W(\cO_X,\cO_Y;M\oplus Z^*)\times W(M\oplus Z^*,\cO_Z;M)$ has a $\rG_{\ue_1}\times \rG_{\ue_2}\times \rG_{\ue_3}\times \Aut_{\cK_2(\cP)}(M\oplus Z^*)$-action with quotient stack denote by $V^M_{(\cO_X\cO_Y)\cO_Z}$. Dually, $W(\cO_Z,M\oplus Z^*;M)\times W(\cO_X,\cO_Y;M\oplus Z^*)$ has a $\rG_{\ue_1}\times \rG_{\ue_2}\times \rG_{\ue_3}\times \Aut_{\cK_2(\cP)}(M\oplus Z^*)$-action with quotient stack denoted by $V^M_{\cO_Z(\cO_X\cO_Y)}$.

\begin{proposition}\label{Case 3}
For any $\ue\leqslant \ue_1+\ue_2$, we have
\begin{align*}
\chi(V^M_{(\cO_X\cO_Y)\cO_Z})=\begin{cases}1, &X\simeq M,Y\simeq Z^*,\\1+\dim_{\bbC}\Hom_{\cK_2(\cP)}(Y,X), &X\simeq Z^*,Y\simeq M.\end{cases}
\end{align*}
Dually, we have
\begin{align*}
\chi(V^M_{\cO_Z(\cO_X\cO_Y)})=\begin{cases}1, &Y\simeq M,X\simeq Z^*,\\1+\dim_{\bbC}\Hom_{\cK_2(\cP)}(Y,X), &Y\simeq Z^*,X\simeq M.\end{cases}
\end{align*}
As a consequence, we have
\begin{equation}
\begin{aligned}\label{chi(V(3))-chi(V(3))}
\sum_{\ue\leqslant \ue_1+\ue_2}(\chi(V^{\ue}_{(12)3}(3))-\chi(V^{\ue}_{3(12)}(3)))=&\int_{[(\hat{\cO}_1^*\cap\hat{\cO}_3)/\rG_{\bd_3}]}\dim_{\bbC}\Hom_{\cK_2(\cP)}(M,Z^*)\cdot 1_{\hat{\cO}_2}(\tilde{M})\\
-&\int_{[(\hat{\cO}_2^*\cap\hat{\cO}_3)/\rG_{\bd_3}]}\dim_{\bbC}\Hom_{\cK_2(\cP)}(Z^*,M)\cdot 1_{\hat{\cO}_1}(\tilde{M}),
\end{aligned}
\end{equation}
where we view $\tilde{Z}\mapsto \dim_{\bbC}\Hom_{\cK_2(\cP)}(M,Z^*)$ and $\tilde{Z}\mapsto dim_{\bbC}\Hom_{\cK_2(\cP)}(Z^*,M)$ as $\rG_{\bd_3}$-invariant support-bounded constructible functions.
\end{proposition}
\begin{proof}
By definition, there is a natural surjective projection
\begin{align*}
\pi:W(\cO_X,\cO_Y;M\oplus Z^*)\times W(M\oplus Z^*,\cO_Z;M)&\rightarrow W(M\oplus Z^*,\cO_Z;M)\\
((f,g,h),(i,j,k))&\mapsto (i,j,k)
\end{align*}
which induces a morphism 
$$\overline{\pi}:V^M_{(\cO_X\cO_Y)\cO_Z}\rightarrow [W(M\oplus Z^*,\cO_Z;M)/\rG_{\ue_3}\times \Aut_{\cK_2(\cP)}(M\oplus Z^*)]=\{\textrm{point}\}$$
whose fiber at the unique point is isomorphic to $[W(\cO_X,\cO_Y;M\oplus Z^*)/\rG_{\ue_1}\times \rG_{\ue_2}\times S]$, where 
$$S=\{(\gamma,\begin{pmatrix}\begin{smallmatrix}1_{M} &\eta\\0 &\gamma^*\end{smallmatrix}\end{pmatrix})\in\Aut_{\cC_2(\cP)}(Z)\times \Aut_{\cK_2(\cP)}(M\oplus Z^*)|\eta\in \Hom_{\cK_2(\cP)}(Z^*,M)\}.$$
Hence $\chi(V^M_{(\cO_X\cO_Y)\cO_Z})=\chi([W(\cO_X,\cO_Y;M\oplus Z^*)/\rG_{\ue_1}\times \rG_{\ue_2}\times S])$.

If $X\simeq M,Y\simeq Z^*$, we fix homotopy equivalences $\varphi:M\rightarrow X,\psi:Y\rightarrow Z^*$. Then for any distinguished triangle $\beta.Y\xrightarrow{\begin{pmatrix}\begin{smallmatrix}f_1\\ f_2\end{smallmatrix}\end{pmatrix}} M\oplus Z^*\xrightarrow{\begin{pmatrix}\begin{smallmatrix}g_1&g_2 \end{smallmatrix}\end{pmatrix}}\alpha.X\xrightarrow{0}(\beta.Y)^*$ in $W(\cO_X,\cO_Y;M\oplus Z^*)$, the morphisms $f_2,g_1$ are homotopy equivalences. By Lemma \ref{isomorphism and homotopy equivalence}, $\varphi,\psi,f_2,g_1$ are isomorphisms. Hence there exists $\varphi g_1^{-1}\in \rG_{\ue_1},\beta^{-1}\in \rG_{\ue_2}$ and $\gamma=\psi^*{\beta^*}^{-1}{f_2^*}^{-1}\in \Aut_{\cC_2(\cP)}(Z),\eta=-f_1f_2^{-1}\in \Hom_{\cK_2(\cP)}(Z^*,M)$ such that
\begin{align*}
&(\varphi g_1^{-1},\beta^{-1},\gamma,\begin{pmatrix}\begin{smallmatrix}1_{M} &\eta\\0 &\gamma^*\end{smallmatrix}\end{pmatrix}).(\beta.Y\xrightarrow{\begin{pmatrix}\begin{smallmatrix}f_1\\ f_2\end{smallmatrix}\end{pmatrix}} M\oplus Z^*\xrightarrow{\begin{pmatrix}\begin{smallmatrix}g_1&g_2 \end{smallmatrix}\end{pmatrix}}\alpha.X\xrightarrow{0}(\beta.Y)^*)\\
=&(Y\xrightarrow{\begin{pmatrix}\begin{smallmatrix}0\\ \psi\end{smallmatrix}\end{pmatrix}}M\oplus Z^*\xrightarrow{\begin{pmatrix}\begin{smallmatrix}\varphi&0 \end{smallmatrix}\end{pmatrix}}X\xrightarrow{0} Y^*).
\end{align*}
Therefore, there is a unique $\rG_{\ue_1}\times \rG_{\ue_2}\times S$-orbit in $W(\cO_X,\cO_Y;M\oplus Z^*)$, and so
$$\chi(V^M_{(\cO_X\cO_Y)\cO_Z})=1.$$

If $X\simeq Z^*,Y\simeq M$, we fix homotopy equivalences $\zeta:Z^*\rightarrow X,\xi:Y\rightarrow M$. Then for any distinguished triangle $\beta.Y\xrightarrow{\begin{pmatrix}\begin{smallmatrix}f_1\\ f_2\end{smallmatrix}\end{pmatrix}} M\oplus Z^*\xrightarrow{\begin{pmatrix}\begin{smallmatrix}g_1&g_2 \end{smallmatrix}\end{pmatrix}}\alpha.X\xrightarrow{0}(\beta.Y)^*$ in $W(\cO_X,\cO_Y;M\oplus Z^*)$, the morphisms $f_1,g_2$ are homotopy equivalences. By Lemma \ref{isomorphism and homotopy equivalence}, $\zeta,\xi,f_1,g_2$ are isomorphisms. Hence there exists $\zeta g_2^{-1}\in \rG_{\ue_1},\xi^{-1}f_1\in \rG_{\ue_2}$ such that  
\begin{align*}
&(\zeta g_2^{-1},\xi^{-1}f_1).(\beta.Y\xrightarrow{\begin{pmatrix}\begin{smallmatrix}f_1\\ f_2\end{smallmatrix}\end{pmatrix}} M\oplus Z^*\xrightarrow{\begin{pmatrix}\begin{smallmatrix}g_1&g_2 \end{smallmatrix}\end{pmatrix}}\alpha.X\xrightarrow{0}(\beta.Y)^*)\\
=&(Y\xrightarrow{\begin{pmatrix}\begin{smallmatrix}\xi\\ f_2f_1^{-1}\xi\end{smallmatrix}\end{pmatrix}} M\oplus Z^*\xrightarrow{\begin{pmatrix}\begin{smallmatrix}\zeta g_2^{-1}g_1&\zeta \end{smallmatrix}\end{pmatrix}}X\xrightarrow{0}Y^*),
\end{align*}
where $\zeta(f_2f_1^{-1}\xi)+(\zeta g_2^{-1}g_1)\xi=0$ follows from $g_1f_1+g_2f_2=0$. Thus the morphism $\theta\mapsto \rG_{\ue_1}\times \rG_{\ue_2}.(Y\xrightarrow{\begin{pmatrix}\begin{smallmatrix}\xi\\ \zeta^{-1}\theta\end{smallmatrix}\end{pmatrix}} M\oplus Z^*\xrightarrow{\begin{pmatrix}\begin{smallmatrix}-\theta\xi^{-1}&\zeta \end{smallmatrix}\end{pmatrix}}X\xrightarrow{0}Y^*)$ induces an isomorphism
\begin{align*}
\rho:\Hom_{\cK_2(\cP)}(Y,X)&\xrightarrow{\cong} [W(\cO_X,\cO_Y;M\oplus Z^*)/\rG_{\ue_1}\times \rG_{\ue_2}].
\end{align*}
Consider the subgroup $S_0=\{(\gamma,\begin{pmatrix}\begin{smallmatrix}1_{M} &0\\0 &\gamma^*\end{smallmatrix}\end{pmatrix})\in S|\gamma\in \Aut_{\cC_2(\cP)}(Z)\}$ of $S$, then $S_0$ acts on $[W(\cO_X,\cO_Y;M\oplus Z^*)/\rG_{\ue_1}\times \rG_{\ue_2}]$ via
\begin{align*}
&(\gamma,\begin{pmatrix}\begin{smallmatrix}1_{M} &0\\0 &\gamma^*\end{smallmatrix}\end{pmatrix}).(\rG_{\ue_1}\times \rG_{\ue_2}.(Y\xrightarrow{\begin{pmatrix}\begin{smallmatrix}\xi\\ \zeta^{-1}\theta\end{smallmatrix}\end{pmatrix}} M\oplus Z^*\xrightarrow{\begin{pmatrix}\begin{smallmatrix}-\theta\xi^{-1}&\zeta \end{smallmatrix}\end{pmatrix}}X\xrightarrow{0}Y^*))\\
=&(\rG_{\ue_1}\times \rG_{\ue_2}.(Y\xrightarrow{\begin{pmatrix}\begin{smallmatrix}\xi\\ \gamma^*\zeta^{-1}\theta\end{smallmatrix}\end{pmatrix}} M\oplus Z^*\xrightarrow{\begin{pmatrix}\begin{smallmatrix}-\theta\xi^{-1}&\zeta\gamma^*\end{smallmatrix}\end{pmatrix}}X\xrightarrow{0}Y^*))
\end{align*}
which implies that $S_0$ acts freely on $\rho(\Hom_{\cK_2(\cP)}(Y,X)\setminus\{0\})$. Under the isomorphism $\rho$, $\Aut_{\cC_2(\cP)}(Z)\cong S_0$ acts freely on $\Hom_{\cK_2(\cP)}(Y,X)\setminus\{0\}$. Since $Z$ is indecomposable, $\Aut_{\cC_2(\cP)}(Z)=\bbC^*\ltimes (1+\rad \End_{\cC_2(\cP)}(Z))$, consider the natural morphism 
$$[(\Hom_{\cK_2(\cP)}(Y,X)\setminus\{0\})/\bbC^*]\rightarrow[(\Hom_{\cK_2(\cP)}(Y,X)\setminus\{0\})/\Aut_{\cC_2(\cP)}(Z)]$$
whose fiber at the geometric point corresponding to $\Aut_{\cC_2(\cP)}(Z).\theta$ is isomorphic to 
$$(1+\rad \End_{\cC_2(\cP)}(Z)).\theta\cong (1+\rad \End_{\cK_2(\cP)}(Z))/\{1\}.$$
It is clear that $(1+\rad \End_{\cK_2(\cP)}(Z))$ is isomorphic to the vector space $\rad \End_{\cK_2(\cP)}(Z)$, and so $\chi((1+\rad \End_{\cK_2(\cP)}(Z)).\theta)=1$ is constant. Applying Corollary \ref{fibre naive Euler characteristic}, we obtain
\begin{align*}
\chi([(\Hom_{\cK_2(\cP)}(Y,X)\setminus\{0\})/\bbC^*])=\chi([(\Hom_{\cK_2(\cP)}(Y,X)\setminus\{0\})/\Aut_{\cC_2(\cP)}(Z)]),
\end{align*}
and so 
\begin{align*}
&\chi([(W(\cO_X,\cO_Y;M\oplus Z^*)\setminus \rho(0))/\rG_{\ue_1}\times \rG_{\ue_2}\times S_0])\\
=&\chi((\Hom_{\cK_2(\cP)}(Y,X)\setminus\{0\})/\Aut_{\cC_2(\cP)}(Z))
=\chi((\Hom_{\cK_2(\cP)}(Y,X)\setminus\{0\})/\bbC^*)\\
=&\chi(\bbC\mathbb{P}^{\dim_{\bbC}\Hom_{\cK_2(\cP)}(Y,X)-1})
=\dim_{\bbC}\Hom_{\cK_2(\cP)}(Y,X),\\
&\chi([W(\cO_X,\cO_Y;M\oplus Z^*)/\rG_{\ue_1}\times \rG_{\ue_2}\times S_0])=1+\dim_{\bbC}\Hom_{\cK_2(\cP)}(Y,X).
\end{align*}
Finally, there is a natural morphism
$$[W(\cO_X,\cO_Y;M\oplus Z^*)/\rG_{\ue_1}\times \rG_{\ue_2}\times S_0]\rightarrow [W(\cO_X,\cO_Y;M\oplus Z^*)/\rG_{\ue_1}\times \rG_{\ue_2}\times S]$$
whose fiber at the geometric point corresponding to the orbit $\rG_{\ue_1}\times \rG_{\ue_2}\times S.(f,g,h)$ is isomorphic to 
$$S_1.(f,g,h)\cong S_1/\Stab_{S_1}(f,g,h),$$
where $S_1=S/S_0=\{(1_Z,\begin{pmatrix}\begin{smallmatrix}1_{M} &\eta\\0 &1_{Z^*}\end{smallmatrix}\end{pmatrix})|\eta\in \Hom_{\cK_2(\cP)}(Z^*,M)\}$ and $\Stab_{S_1}(f,g,h)$ is the stabilizer. It is clear that $S_1$ is isomorphic to the vector space $\Hom_{\cK_2(\cP)}(Z^*,M)$ and $\Stab_{S_1}(f,g,h)$ is isomorphic to a vector subspace, thus $\chi(S_1.(f,g,h))=1$ is constant. Applying Corollary \ref{fibre naive Euler characteristic}, we obtain 
\begin{align*}
\chi(V^M_{(\cO_X\cO_Y)\cO_Z})=&\chi([W(\cO_X,\cO_Y;M\oplus Z^*)/\rG_{\ue_1}\times \rG_{\ue_2}\times S])\\
=&\chi([W(\cO_X,\cO_Y;M\oplus Z^*)/\rG_{\ue_1}\times \rG_{\ue_2}\times S_0])\\
=&1+\dim_{\bbC}\Hom_{\cK_2(\cP)}(Y,X),
\end{align*}
as desired. The dual statement for $\chi(V^M_{\cO_Z(\cO_X\cO_Y)})$ can be proved similarly, and the result for $\sum_{\ue\leqslant \ue_1+\ue_2}\chi(V^{\ue}_{(12)3}(3))-\chi(V^{\ue}_{3(12)}(3))$ follows directly.
\end{proof}

\subsubsection{\textbf{Case} $t=4$}\

In the case $L\simeq M\oplus Z^*\simeq X\oplus Y,X\simeq Y$ for some $X\in \cO_{\ue_1},Y\in \cO_{\ue_2},Z\in \cO_{\ue_3}$. By similar argument as ($\ref{V(2)}$) in Proposition \ref{Case 2}, we have the following proposition.

\begin{proposition}\label{Case 4}
For any $\ue\leqslant \ue_1+\ue_2$, we have 
$$V^M_{(\cO_X\cO_Y)\cO_Z}=V^M_{\cO_Z(\cO_X\cO_Y)}.$$
As a consequence, we have
$$\chi(V^{\ue}_{(12)3}(4))=\chi(V^{\ue}_{3(12)}(4)).$$
\end{proposition}

\subsubsection{Application}\

\begin{proof}[Proof of Lemma \ref{Lie bracket supported indecomposable}]
Note that  for any distinguished triangle $Y\rightarrow L \rightarrow X\xrightarrow{h} Y^*$, the complex $L\simeq 0$ if and only if $h$ is a homotopy equivalence, and so if $L\simeq 0$, then
\begin{align*}
F^L_{\cO_{\ue_1}\cO_{\ue_2}}-F^L_{\cO_{\ue_2}\cO_{\ue_1}}=\ \ &\chi([(\rG_{\bd_1}.t_{\ue_1}(\cO_{\ue_1})\cap(\rG_{\bd_2}.t_{\ue_2}(\cO_{\ue_2}))^*)/\rG_{\bd_1}])\\
-&\chi([(\rG_{\bd_2}.t_{\ue_2}(\cO_{\ue_2})\cap(\rG_{\bd_1}.t_{\ue_1}(\cO_{\ue_1}))^*)/\rG_{\bd_2}])=0.
\end{align*}

If $L$ is decomposable as an object in $\cK_2(\cP)$, suppose $L\simeq M\oplus Z^*$, where $M,Z\not\simeq 0$. If $L\not\simeq X\oplus Y$ for any $X\in \cO_{\ue_1},Y\in \cO_{\ue_2}$, by Proposition \ref{Case 2}, then
$$F^L_{\cO_{\ue_1}\cO_{\ue_2}}=F^L_{\cO_{\ue_2}\cO_{\ue_1}}=0.$$ 
If $L\simeq\! X\oplus Y,X\not\simeq Y$ for some $X\in \cO_{\ue_1},Y\in \cO_{\ue_2}$, by (\ref{V(Xoplus Y)=1}) and Proposition \ref{Case 2}, then
\begin{align*}
F^L_{\cO_X\cO_Y}&=F^L_{\cO_Y\cO_X}=1,\\
F^L_{\cO_{\ue_1}\cO_{\ue_2}}=F^L_{\cO_X\cO_{\ue_2}}+F^L_{(\cO_{\ue_1}\setminus\cO_X)\cO_{\ue_2}}&=(F^L_{\cO_X\cO_Y}+F^L_{\cO_X(\cO_{\ue_2}\setminus\cO_Y)})+0=1,\\
F^L_{\cO_{\ue_2}\cO_{\ue_1}}=F^L_{\cO_Y\cO_{\ue_1}}+F^L_{(\cO_{\ue_2}\setminus\cO_Y)\cO_{\ue_1}}&=(F^L_{\cO_Y\cO_X}+F^L_{\cO_Y(\cO_{\ue_1}\setminus\cO_X)})+0=1.
\end{align*}
If $L\!\simeq X\oplus Y,X\simeq Y$ for some $X\in \cO_{\ue_1},Y\in \cO_{\ue_2}$, it is clear that $F^L_{\cO_X\cO_Y}=F^L_{\cO_Y\cO_X}$. By Proposition \ref{Case 2}, we have
\begin{align*}
F^L_{\cO_{\ue_1}\cO_{\ue_2}}=F^L_{\cO_X\cO_{\ue_2}}+F^L_{(\cO_{\ue_1}\setminus\cO_X)\cO_{\ue_2}}&=(F^L_{\cO_X\cO_Y}+F^L_{\cO_X(\cO_{\ue_2}\setminus\cO_Y)})+0=F^L_{\cO_X\cO_Y},\\
F^L_{\cO_{\ue_2}\cO_{\ue_1}}=F^L_{\cO_Y\cO_{\ue_1}}+F^L_{(\cO_{\ue_2}\setminus\cO_Y)\cO_{\ue_1}}&=(F^L_{\cO_Y\cO_X}+F^L_{\cO_Y(\cO_{\ue_1}\setminus\cO_X)})+0=F^L_{\cO_Y\cO_X}.
\end{align*}
Thus if $L$ is decomposable or $0$ as an object in $\cK_2(\cP)$, then $F^L_{\cO_{\ue_1}\cO_{\ue_2}}-F^L_{\cO_{\ue_2}\cO_{\ue_1}}=0$.
\end{proof}

\begin{corollary}\label{chi(tilde(V)(2))=0}
For any $\ue\leqslant \ue_1+\ue_2$, we have 
$$\chi(\tilde{V}^{\ue}_{(12)3}(2))=\chi(\tilde{V}^{\ue}_{(21)3}(2)).$$
Dually, for any $\ue'\leqslant \ue_2+\ue_3$, we have
$$\chi(\tilde{V}^{\ue'}_{1(23)}(2))=\chi(\tilde{V}^{\ue'}_{1(32)}(2)).$$
\end{corollary}
\begin{proof}
Recall that $\tilde{V}^{\ue}_{(12)3}=\bigsqcup_{(r,s)\in \cI^{\ue}_{\ue_1\ue_2}}V(\cO_{\ue_1},\cO_{\ue_2};L_{r,s})\times V(\langle L_{r,s}\rangle_2,\cO_{\ue_3};M)$. For any $(r,s)\in \cI^{\ue}_{\ue_1\ue_2}$, if $\langle L_{r,s}\rangle_2\not=\varnothing$, there exists $L\simeq M\oplus Z^*\in \langle L_{r,s}\rangle$ for some $Z\in \cO_{\ue_3}$. Since $L$ is decomposable, by above proof, we have 
$$r=F^{L_{r,s}}_{\cO_{\ue_1}\cO_{\ue_2}}=F^L_{\cO_{\ue_1}\cO_{\ue_2}}=F^L_{\cO_{\ue_2}\cO_{\ue_1}}=F^{L_{r,s}}_{\cO_{\ue_1}\cO_{\ue_2}}=s,$$
and so $$\chi(\tilde{V}^{\ue}_{(12)3})=\sum_{(r,r)\in \cI^{\ue}_{\ue_1\ue_2}}r\chi(V(\langle L_{r,r}\rangle_2,\cO_{\ue_3};M)).$$
Similarly, 
$$\chi(\tilde{V}^{\ue}_{(21)3})=\sum_{(r,r)\in \cI^{\ue}_{\ue_2\ue_1}}r\chi(V(\langle L_{r,r}\rangle_2,\cO_{\ue_3};M)),$$
as desired. The dual statement can be proved similarly.
\end{proof}

\subsection{Jacobi identity}\label{Jacobi identity}\

In this subsection, we prove Theorem \ref{Lie algebra g_2}. We use the same notations in Subsection \ref{An application of octahedral axiom} and \ref{Preliminary result about the structure constants}. Moreover, for convenience, we simply denote by $\cO_i=\cO_{\ue_i}$ for $i=1,2,3$.

\begin{proof}[Proof of Theorem \ref{Lie algebra g_2}] 
Firstly, we prove
\begin{equation}
[[1_{\hat{\cO}_1},1_{\hat{\cO}_2}],1_{\hat{\cO}_3}]-[[1_{\hat{\cO}_1},1_{\hat{\cO}_3}],1_{\hat{\cO}_2}]-[[1_{\hat{\cO}_3},1_{\hat{\cO}_2}],1_{\hat{\cO}_1}]=0.\tag{I}
\end{equation}
By definition, see (\ref{constant F}), Lemma \ref{finite subset I}, (\ref{Lie algebra g_2 formula}) and Subsection \ref{Preliminary result about the structure constants}, we have
\begin{align*}
[1_{\hat{\cO}_1},1_{\hat{\cO}_2}]_{\tilde{\fn}}=&\sum_{\ue\leqslant \ue_1+\ue_2}\sum_{(r,s)\in \cI^{\ue}_{\ue_1\ue_2}}(F^{L_{r,s}}_{\cO_1\cO_2}-F^{L_{r,s}}_{\cO_2\cO_1})1_{\langle L_{r,s}\rangle},\\
[[1_{\hat{\cO}_1},1_{\hat{\cO}_2}]_{\tilde{\fn}},1_{\hat{\cO}_3}]=&[[1_{\hat{\cO}_1},1_{\hat{\cO}_2}]_{\tilde{\fn}},1_{\hat{\cO}_3}]_{\tilde{\fn}}-(F^{\hat{\cO}_3^*}_{\hat{\cO}_1\hat{\cO}_2}-F^{\hat{\cO}_3^*}_{\hat{\cO}_2\hat{\cO}_1})\tilde{h}_{\bd_1+\bd_2},
\end{align*}
where 
\begin{align}
[[1_{\hat{\cO}_1},1_{\hat{\cO}_2}]_{\tilde{\fn}},1_{\hat{\cO}_3}]_{\tilde{\fn}}(\tilde{M})=&\sum_{\ue\leqslant \ue_1+\ue_2}\sum_{(r,s)\in \cI^{\ue}_{\ue_1\ue_2}}(F^{L_{r,s}}_{\cO_1\cO_2}-F^{L_{r,s}}_{\cO_2\cO_1})(F^M_{\langle L_{r,s}\rangle\cO_3}-F^M_{\cO_3\langle L_{r,s}\rangle})\notag\\
\overset{\ref{Preliminary result about the structure constants}}{=}&\sum_{\ue\leqslant\ue_1+\ue_2}(\chi(\tilde{V}^{\ue}_{(12)3})-\chi(\tilde{V}^{\ue}_{(21)3})-\chi(\tilde{V}^{\ue}_{3(12)})+\chi(\tilde{V}^{\ue}_{3(21)})),\notag\\
F^{\hat{\cO}_3^*}_{\hat{\cO}_1\hat{\cO}_2}=&\sum_{\ue\leqslant \ue_1+\ue_2}\sum_{(r,s)\in \cI^{\ue}_{\ue_1\ue_2}}F^{L_{r,s}}_{\cO_1\cO_2}\chi([(\rG_{\bd_1+\bd_2}.t_{\ue}(\langle L_{r,s}\rangle)\cap \hat{\cO}^*_3)/\rG_{\bd_1+\bd_2}]) \label{F=Fchi}
\end{align}
for any $\tilde{M}\in \rP_2(A,\bd_1+\bd_2+\bd_3)$, and $F^{\hat{\cO}_3^*}_{\hat{\cO}_2\hat{\cO}_1}$ is similar. Thus 
\begin{align*}
&[[1_{\hat{\cO}_1},1_{\hat{\cO}_2}],1_{\hat{\cO}_3}]=[([1_{\hat{\cO}_1},1_{\hat{\cO}_2}]_{\tilde{\fn}}-\chi([(\hat{\cO}_1\cap\hat{\cO}_2^*)/\rG_{\bd_1}])\tilde{h}_{\bd_1}),1_{\hat{\cO}_3}]\\
=&[[1_{\hat{\cO}_1},1_{\hat{\cO}_2}]_{\tilde{\fn}},1_{\hat{\cO}_3}]_{\tilde{\fn}}-(F^{\cO_3^*}_{\cO_1\cO_2}-F^{\cO_3^*}_{\cO_2\cO_1})\tilde{h}_{\bd_1+\bd_2}-\chi([(\hat{\cO}_1\cap\hat{\cO}_2^*)/\rG_{\bd_1}])(\tilde{h}_{\bd_1}|\tilde{h}_{\bd_3})1_{\hat{\cO}_3}.
\end{align*}
We denote by
\begin{align*}
A_{\tilde{M}}=\ \ &\sum_{\ue\leqslant\ue_1+\ue_2}(\chi(\tilde{V}^{\ue}_{(12)3})-\chi(\tilde{V}^{\ue}_{(21)3})-\chi(\tilde{V}^{\ue}_{3(12)})+\chi(\tilde{V}^{\ue}_{3(21)}))\\
-&\sum_{\ue'\leqslant\ue_1+\ue_3}(\chi(\tilde{V}^{\ue'}_{(13)2})-\chi(\tilde{V}^{\ue'}_{(31)2})-\chi(\tilde{V}^{\ue'}_{2(13)})+\chi(\tilde{V}^{\ue'}_{2(31)}))\\
-&\sum_{\ue''\leqslant\ue_3+\ue_2}(\chi(\tilde{V}^{\ue''}_{(32)1})-\chi(\tilde{V}^{\ue''}_{(23)1})-\chi(\tilde{V}^{\ue''}_{1(32)})+\chi(\tilde{V}^{\ue''}_{1(23)}))\\
=\ \ &\triangle^{\tilde{M}}_{123}+\triangle^{\tilde{M}}_{231}+\triangle^{\tilde{M}}_{312}-\triangle^{\tilde{M}}_{213}-\triangle^{\tilde{M}}_{321}-\triangle^{\tilde{M}}_{132},\\
B_{\tilde{M}}=\ \ &\chi([(\hat{\cO}_1\cap\hat{\cO}_2^*)/\rG_{\bd_1}])(\tilde{h}_{\bd_1}|\tilde{h}_{\bd_3})1_{\hat{\cO}_3}(\tilde{M})\\
-&\chi([(\hat{\cO}_1\cap\hat{\cO}_3^*)/\rG_{\bd_1}])(\tilde{h}_{\bd_1}|\tilde{h}_{\bd_2})1_{\hat{\cO}_2}(\tilde{M})\\
-&\chi([(\hat{\cO}_3\cap\hat{\cO}_2^*)/\rG_{\bd_3}])(\tilde{h}_{\bd_3}|\tilde{h}_{\bd_1})1_{\hat{\cO}_1}(\tilde{M}),\\
C=\ \ &(F^{\hat{\cO}_3^*}_{\hat{\cO}_1\hat{\cO}_2}-F^{\hat{\cO}_3^*}_{\hat{\cO}_2\hat{\cO}_1})\tilde{h}_{\bd_1+\bd_2}\\
-&(F^{\hat{\cO}_2^*}_{\hat{\cO}_1\hat{\cO}_3}-F^{\hat{\cO}_2^*}_{\hat{\cO}_3\hat{\cO}_1})\tilde{h}_{\bd_1+\bd_3}\\
-&(F^{\hat{\cO}_1^*}_{\hat{\cO}_3\hat{\cO}_2}-F^{\hat{\cO}_1^*}_{\hat{\cO}_2\hat{\cO}_3})\tilde{h}_{\bd_3+\bd_2},
\end{align*}
where 
$$\triangle^{\tilde{M}}_{123}=\sum_{\ue\leqslant\ue_1+\ue_2}\chi(\tilde{V}^{\ue}_{(12)3})-\sum_{\ue''\leqslant\ue_3+\ue_2}\chi(\tilde{V}^{\ue''}_{1(23)})$$
and the other $\triangle^{\tilde{M}}_{pqr}$ are similar, then (I) is equivalent to $A_{\tilde{M}}-B_{\tilde{M}}-C=0$ for any $\tilde{M}\in \rP_2(A,\bd_1+\bd_2+\bd_3)$. By Proposition \ref{Case 1} and Proposition \ref{octahedral Euler characteristic}, we have
\begin{align*}
\triangle^{\tilde{M}}_{123}=\ \ &\sum_{\ue\leqslant\ue_1+\ue_2}(\chi(\tilde{V}^{\ue}_{(12)3}(1))+\chi(\tilde{V}^{\ue}_{(12)3}(2)))-\sum_{\ue''\leqslant\ue_3+\ue_2}(\chi(\tilde{V}^{\ue''}_{1(23)}(1))+\chi(\tilde{V}^{\ue''}_{1(23)}(2)))\\
\overset{\ref{Case 1}}{=}\ \ &\sum_{\ue\leqslant\ue_1+\ue_2}(\chi(V^{\ue}_{(12)3}(1))+\chi(\tilde{V}^{\ue}_{(12)3}(2)))-\sum_{\ue''\leqslant\ue_3+\ue_2}(\chi(V^{\ue''}_{1(23)}(1))+\chi(\tilde{V}^{\ue''}_{1(23)}(2)))\\
\overset{\ref{octahedral Euler characteristic}}{=}\ \ &\sum_{\ue\leqslant\ue_1+\ue_2}\chi(\tilde{V}^{\ue}_{(12)3}(2))-\sum_{\ue''\leqslant\ue_3+\ue_2}\chi(\tilde{V}^{\ue''}_{1(23)}(2)) \tag{$\triangle^1$} \\
+&\ \ \sum^4_{t=2}(\sum_{\ue''\leqslant\ue_3+\ue_2}\chi(V^{\ue''}_{1(23)}(t))-\sum_{\ue\leqslant\ue_1+\ue_2}\chi(V^{\ue}_{(12)3}(t))) \tag{$\triangle^2$}.
\end{align*}
Then we divide $A_{\tilde{M}}=A_{\tilde{M}}^1+A_{\tilde{M}}^2$, where $A_{\tilde{M}}^1$ is the sum of terms similar to $(\triangle^1)$ appearing in all $\triangle^{\tilde{M}}_{pqr}$, and $A_{\tilde{M}}^2$ is the sum of terms similar to $(\triangle^2)$ appearing in all $\triangle^{\tilde{M}}_{pqr}$, that is, 
\begin{align*}
A_{\tilde{M}}^1=&\!\!\!\sum_{\ue\leqslant\ue_1+\ue_2}\!\!\chi(\tilde{V}^{\ue}_{(12)3}(2))-\!\!\sum_{\ue''\leqslant\ue_3+\ue_2}\!\!\chi(\tilde{V}^{\ue''}_{1(23)}(2))
+\!\!\sum_{\ue''\leqslant\ue_2+\ue_3}\!\!\chi(\tilde{V}^{\ue''}_{(23)1}(2))-\!\!\sum_{\ue'\leqslant\ue_1+\ue_3}\!\!\chi(\tilde{V}^{\ue'}_{2(31)}(2))\\
+&\!\!\!\sum_{\ue'\leqslant\ue_3+\ue_1}\!\!\chi(\tilde{V}^{\ue'}_{(31)2}(2))-\!\!\sum_{\ue\leqslant\ue_2+\ue_1}\chi(\tilde{V}^{\ue}_{3(12)}(2))
-\!\!\sum_{\ue\leqslant\ue_2+\ue_1}\chi(\tilde{V}^{\ue}_{(21)3}(2))+\!\!\sum_{\ue'\leqslant\ue_3+\ue_1}\!\!\chi(\tilde{V}^{\ue'}_{2(13)}(2))\\
-&\!\!\!\!\sum_{\ue''\leqslant\ue_3+\ue_2}\!\!\chi(\tilde{V}^{\ue''}_{(32)1}(2))+\!\!\sum_{\ue\leqslant\ue_1+\ue_2}\!\!\chi(\tilde{V}^{\ue}_{3(21)}(2))
-\!\!\sum_{\ue'\leqslant\ue_1+\ue_3}\!\!\chi(\tilde{V}^{\ue'}_{(13)2}(2))+\!\!\sum_{\ue''\leqslant\ue_2+\ue_3}\!\!\chi(\tilde{V}^{\ue''}_{1(32)}(2)),\\
A_{\tilde{M}}^2=&\sum^4_{t=2}\\
((&\!\!\!\!\!\sum_{\ue''\leqslant\ue_3+\ue_2}\!\!\chi(V^{\ue''}_{1(23)}(t))-\!\!\sum_{\ue\leqslant\ue_1+\ue_2}\!\!\chi(V^{\ue}_{(12)3}(t)))
+(\!\!\!\sum_{\ue'\leqslant\ue_1+\ue_3}\!\!\chi(V^{\ue'}_{2(31)}(t))-\!\!\!\sum_{\ue''\leqslant\ue_2+\ue_3}\!\!\chi(V^{\ue''}_{(23)1}(t)))\\
+(&\!\!\!\!\!\sum_{\ue\leqslant\ue_2+\ue_1}\!\!\chi(V^{\ue}_{3(12)}(t))-\!\!\sum_{\ue'\leqslant\ue_3+\ue_1}\!\!\chi(V^{\ue'}_{(31)2}(t)))
-(\!\!\sum_{\ue'\leqslant\ue_3+\ue_1}\!\!\chi(V^{\ue'}_{2(13)}(t))-\!\!\sum_{\ue\leqslant\ue_2+\ue_1}\!\!\chi(V^{\ue}_{(21)3}(t)))\\
-(&\!\!\!\!\!\sum_{\ue\leqslant\ue_1+\ue_2}\!\!\chi(V^{\ue}_{3(21)}(t))-\!\!\sum_{\ue''\leqslant\ue_3+\ue_2}\!\!\chi(V^{\ue''}_{(32)1}(t)))
-(\!\!\!\!\sum_{\ue''\leqslant\ue_2+\ue_3}\!\!\chi(V^{\ue''}_{1(32)}(t))-\!\!\!\!\sum_{\ue'\leqslant\ue_1+\ue_3}\!\!\!\!\chi(V^{\ue'}_{(13)2}(t)))).
\end{align*}
By Corollary \ref{chi(tilde(V)(2))=0}, we have $A_{\tilde{M}}^1=0$. By Proposition \ref{Case 2} and \ref{Case 4}, the sums in $A_{\tilde{M}}^2$ taking over $t=2$ or $t=4$ vanish, and so it remains to deal with the sum  taking over $t=3$. We use (\ref{chi(V(3))-chi(V(3))}) in Proposition \ref{Case 3} to substitute terms in $A_{\tilde{M}}^2$, then we will obtain $A_{\tilde{M}}^2=B_{\tilde{M}}$. More precisely, consider the term $-\chi([(\hat{\cO}_3\cap\hat{\cO}_2^*)/\rG_{\bd_3}])(\tilde{h}_{\bd_3}|\tilde{h}_{\bd_1})1_{\hat{\cO}_1}(\tilde{M})$ in $B_{\tilde{M}}$, notice that if $\hat{\cO}_3\cap\hat{\cO}_2^*=\varnothing$ or $1_{\hat{\cO}_1}(\tilde{M})=0$, this term vanishes. On the other hand, consider those terms in $A^2_{\tilde{M}}$ which contribute $1_{\hat{\cO}_1}(\tilde{M})$, we have the following result and we denote it by $a_11_{\hat{\cO}_1}(\tilde{M})$,
\begin{align*}
&\ \sum_{\ue'\leqslant\ue_1+\ue_3}\!\!\chi(V^{\ue'}_{2(31)}(3))-\!\!\sum_{\ue'\leqslant\ue_3+\ue_1}\!\!\chi(V^{\ue'}_{(31)2}(3))
+\!\!\sum_{\ue\leqslant\ue_2+\ue_1}\!\!\chi(V^{\ue}_{3(12)}(3))-\!\!\sum_{\ue\leqslant\ue_1+\ue_2}\!\!\chi(V^{\ue}_{(12)3}(3))\\
&+\!\!\sum_{\ue'\leqslant\ue_1+\ue_3}\!\!\chi(V^{\ue'}_{(13)2}(3))-\!\!\sum_{\ue'\leqslant\ue_3+\ue_1}\!\!\chi(V^{\ue'}_{2(13)}(3))+\!\!\sum_{\ue\leqslant\ue_2+\ue_1}\!\!\chi(V^{\ue}_{(21)3}(3))-\!\!\sum_{\ue\leqslant\ue_1+\ue_2}\!\!\chi(V^{\ue}_{3(21)}(3))\\
=&(-\!\!\int_{[(\hat{\cO}_3^*\cap\hat{\cO}_2)/\rG_{\bd_2}]}\!\dim_{\bbC}\Hom_{\cK_2(\cP)}(M,Z^*)\!+\!\!\int_{[(\hat{\cO}_2^*\cap\hat{\cO}_3)/\rG_{\bd_3}]}\!\dim_{\bbC}\Hom_{\cK_2(\cP)}(Z^*,M)\\
-\!\!&\int_{[(\hat{\cO}_3^*\cap\hat{\cO}_2)/\rG_{\bd_2}]}\!\!\dim_{\bbC}\Hom_{\cK_2(\cP)}(Z^*,M)\!+\!\!\int_{[(\hat{\cO}_2^*\cap\hat{\cO}_3)/\rG_{\bd_3}]}\!\!\dim_{\bbC}\Hom_{\cK_2(\cP)}(M,Z^*))1_{\hat{\cO}_1}(\tilde{M}),
\end{align*}
where all integrations are about functions in $\tilde{Z}$, see Proposition \ref{Case 3}. Notice that if $\hat{\cO}_3\cap\hat{\cO}_2^*=\varnothing$ or $1_{\hat{\cO}_1}(\tilde{M})=0$, we have $a_11_{\hat{\cO}_1}(\tilde{M})=0$; otherwise, we have $\bd_2+\bd_3=0$, the image of $M$ in the Grothendieck group $K_0$ is $\bd_1+\bd_2+\bd_3=\bd_1$, and $a_11_{\hat{\cO}_1}(\tilde{M})$ equals to
\begin{align*}
&(-\!\!\int_{[(\hat{\cO}_3\cap\hat{\cO}^*_2)/\rG_{\bd_3}]}\!\dim_{\bbC}\Hom_{\cK_2(\cP)}(M,Z)\!+\!\!\int_{[(\hat{\cO}_2^*\cap\hat{\cO}_3)/\rG_{\bd_3}]}\!\dim_{\bbC}\Hom_{\cK_2(\cP)}(Z^*,M)\\
-\!\!&\int_{[(\hat{\cO}_3\cap\hat{\cO}_2^*)/\rG_{\bd_3}]}\!\dim_{\bbC}\Hom_{\cK_2(\cP)}(Z,M)\!+\!\!\int_{[(\hat{\cO}_2^*\cap\hat{\cO}_3)/\rG_{\bd_3}]}\!\dim_{\bbC}\Hom_{\cK_2(\cP)}(M,Z^*))1_{\hat{\cO}_1}(\tilde{M})\\
\overset{(\ref{symmetric bilinear form})}{=}&-\int_{[(\hat{\cO}_2^*\cap\hat{\cO}_3)/\rG_{\bd_3}]}(\tilde{h}_{\bd_1}|\tilde{h}_{\bd_3})1_{\hat{\cO}_1}(\tilde{M})=-\chi([(\hat{\cO}_2^*\cap\hat{\cO}_3)/\rG_{\bd_3}])(\tilde{h}_{\bd_1}|\tilde{h}_{\bd_3})1_{\hat{\cO}_1}(\tilde{M}).
\end{align*}
The other terms in $A_{\tilde{M}}^2$ and $B_{\tilde{M}}$ cancel out in a similar way. It remains to prove $C=0$. Notice that if $\bd_1+\bd_2+\bd_3\not=0$, all $F^{\hat{\cO}_r^*}_{\hat{\cO}_p\hat{\cO}_q}$ appearing in $C$ vanish, then so does $C$. Next, we assume $\bd_1+\bd_2+\bd_3=0$. By definition, we have
\begin{align*}
F^{\hat{\cO}_3^*}_{\hat{\cO}_1\hat{\cO}_2}=&\int_{[\hat{\cO}^*_3/\rG_{\bd_3}]}1_{\hat{\cO}_1}*1_{\hat{\cO}_2}(\tilde{Z}^*)
\overset{(\ref{integration form})}{=}\int_{[\hat{\cO}_3/\rG_{\bd_3}]}\int_{V(\hat{\cO}_1,\hat{\cO}_2;\tilde{Z})}1\\=&\int_{V(\hat{\cO}_1,\hat{\cO}_2;\hat{\cO}_3^*)}1=\chi(V(\hat{\cO}_1,\hat{\cO}_2;\hat{\cO}_3^*)),
\end{align*}
where $V(\hat{\cO}_1,\hat{\cO}_2;\hat{\cO}_3^*)=[W(\hat{\cO}_1,\hat{\cO}_2;\hat{\cO}_3^*)/\rG_{\bd_1}\times \rG_{\bd_2}\times \rG_{\bd_3}]$, and $W(\hat{\cO}_1,\hat{\cO}_2;\hat{\cO}_3^*)$ is a constructible set consisting of distinguished triangles $Y\xrightarrow{f}Z^*\xrightarrow{g}X\xrightarrow{h}Y^*$ for $\tilde{X}\in \hat{\cO}_1,\tilde{Y}\in \hat{\cO}_2,\tilde{Z}\in \hat{\cO}_3$, the group $\rG_{\bd_1}\times \rG_{\bd_2}\times \rG_{\bd_3}$ acts on it via
$$(\alpha,\beta,\gamma).(Y\xrightarrow{f}Z^*\xrightarrow{g}X\xrightarrow{h}Y^*)=(\beta.Y\xrightarrow{\gamma^*f\beta^{-1}}(\gamma.Z)^*\xrightarrow{\alpha g{\gamma^*}^{-1}}\alpha.X\xrightarrow{\beta^*h\alpha^{-1}}(\beta.Y)^*).$$
Similarly, we have $F^{\hat{\cO}_1^*}_{\hat{\cO}_2\hat{\cO}_3}=\chi(V(\hat{\cO}_2,\hat{\cO}_3;\hat{\cO}_1^*))$. The rotation of distinguished triangles
$$(Y\xrightarrow{f}Z^*\xrightarrow{g}X\xrightarrow{h}Y^*)\mapsto (Z\xrightarrow{-g^*}X^*\xrightarrow{-h^*}Y\xrightarrow{f}Z^*)$$
defines an isomorphism $W(\hat{\cO}_1,\hat{\cO}_2;\hat{\cO}_3^*)\cong W(\hat{\cO}_2,\hat{\cO}_3;\hat{\cO}_1^*)$, and then it induces an isomorphism $V(\hat{\cO}_1,\hat{\cO}_2;\hat{\cO}_3^*)\cong V(\hat{\cO}_2,\hat{\cO}_3;\hat{\cO}_1^*)$, and so $F^{\hat{\cO}_3^*}_{\hat{\cO}_1\hat{\cO}_2}=F^{\hat{\cO}_1^*}_{\hat{\cO}_2\hat{\cO}_3}$. Similarly, we have 
\begin{align}\label{F=F}
F^{\hat{\cO}_3^*}_{\hat{\cO}_1\hat{\cO}_2}=F^{\hat{\cO}_1^*}_{\hat{\cO}_2\hat{\cO}_3}=F^{\hat{\cO}_2^*}_{\hat{\cO}_3\hat{\cO}_1},F^{\hat{\cO}_3^*}_{\hat{\cO}_2\hat{\cO}_1}=F^{\hat{\cO}_2^*}_{\hat{\cO}_1\hat{\cO}_3}=F^{\hat{\cO}_1^*}_{\hat{\cO}_3\hat{\cO}_2},
\end{align} 
and so
\begin{align*}
C=F^{\hat{\cO}_3^*}_{\hat{\cO}_1\hat{\cO}_2}(\tilde{h}_{\bd_1+\bd_2}+\tilde{h}_{\bd_1+\bd_3}+\tilde{h}_{\bd_3+\bd_2})-F^{\hat{\cO}_3^*}_{\hat{\cO}_2\hat{\cO}_1}(\tilde{h}_{\bd_1+\bd_2}+\tilde{h}_{\bd_1+\bd_3}+\tilde{h}_{\bd_3+\bd_2})=0,
\end{align*}
as desired. This finishes the proof of (I).

Secondly, we prove 
\begin{align}
[[\tilde{h}_{\bd_1},1_{\hat{\cO}_2}],1_{\hat{\cO}_3}]-[[\tilde{h}_{\bd_1},1_{\hat{\cO}_3}],1_{\hat{\cO}_2}]-[[1_{\hat{\cO}_3},1_{\hat{\cO}_2}],\tilde{h}_{\bd_1}]=0. \tag{II}
\end{align}
By definition, see (\ref{Lie algebra g_2 formula}), the left hand side of (II) equals to 
\begin{align}
(\tilde{h}_{\bd_1}|\tilde{h}_{\bd_2}+\tilde{h}_{\bd_3})[1_{\hat{\cO}_2},1_{\hat{\cO}_3}]-(\tilde{h}_{\bd_1}|\tilde{h}_{\bd_2}+\tilde{h}_{\bd_3})[1_{\hat{\cO}_2},1_{\hat{\cO}_3}]_{\tilde{\fn}}, \tag{II'}
\end{align}
Notice that if $\hat{\cO}_2\cap \hat{\cO}_3^*=\varnothing$, we have $[1_{\hat{\cO}_2},1_{\hat{\cO}_3}]=[1_{\hat{\cO}_2},1_{\hat{\cO}_3}]_{\tilde{\fn}}$; otherwise, we have $\bd_2+\bd_3=0$, and then $\tilde{h}_{\bd_2}+\tilde{h}_{\bd_3}=0$. Hence, (II') always vanishes, as desired. 

Finally,  we prove
\begin{align}
[[\tilde{h}_{\bd_1},\tilde{h}_{\bd_2}],1_{\hat{\cO}_3}]-[[\tilde{h}_{\bd_1},1_{\hat{\cO}_3}],\tilde{h}_{\bd_2}]-[[1_{\hat{\cO}_3},\tilde{h}_{\bd_2}],\tilde{h}_{\bd_1}]&=0, \tag{III}\\
[[\tilde{h}_{\bd_1},\tilde{h}_{\bd_2}],\tilde{h}_{\bd_3}]-[[\tilde{h}_{\bd_1},\tilde{h}_{\bd_3}],\tilde{h}_{\bd_2}]-[[\tilde{h}_{\bd_3},\tilde{h}_{\bd_2}],\tilde{h}_{\bd_1}]&=0. \tag{IV}
\end{align}
By definition, see (\ref{Lie algebra g_2 formula}), they are trivial.
\end{proof}

\section{The isomorphism of two Lie algebras}\label{The isomorphism of two Lie algebras}\

There is a natural functor $\Phi:\cC_2(\cP)\rightarrow \cK_2(\cP)$ which preserves objects and maps morphisms in $\cC_2(\cP)$ to their homotopy classes. In this section, we prove that $\Phi$ together with the isomorphism between Grothendieck groups $\cK(A)\cong K_0$ induces a Lie algebra isomorphism $\varphi:\fg\rightarrow \tilde{\fg}$, where $\fg=\fn\oplus \fh$ is defined in Subsection \ref{Lie algebra spanned by contractible complexes and indecomposable radical complexes} and $\tilde{\fg}=\tilde{\fn}\oplus {\fh}$ is defined in Subsection \ref{Lie algebra spanned by supported-indecomposable function and the Grothendieck group}.

By definition, $\fh$ is a $\bbC$-subspace spanned by $\{h_\alpha|\alpha\in K(\cA)\}$. By Lemma \ref{property of h}, we have $\fh\cong\bbC\otimes_\bbZ K(\cA)$ and 
\begin{equation}\label{[h,h]=0}
[h_\alpha,h_{\alpha'}]=0,
\end{equation}
for any $\alpha,\alpha'\in K(\cA)$, where $K(\cA)$ is the Grothendieck group of $\cA$. By definition, $\tilde{\fh}=\bbC\otimes_\bbZ K_0$ and
\begin{equation}\label{[h,h]'=0}
[\tilde{h}_{\bd},\tilde{h}_{\bd'}]=0,
\end{equation}
for any $\bd,\bd'\in K_0$, where $K_0$ is the Grothendieck group of $\cK_2(\cP)$.

By \cite[Proposition 2.11]{Fu-2012}, there is an isomorphism $K_0\cong K(\cA)$ preserving their symmetric bilinear forms. More precisely, the map
$$X=(X^1,X^0,d^1,d^0)\mapsto \hat{X^0}-\hat{X^1}$$
induces an isomorphism $\kappa^{-1}:K_0\rightarrow K(\cA)$ and a Lie algebra isomorphism 
\begin{align*}
\varphi_{\fh}:\fh&\rightarrow \tilde{\fh}\\
h_{\alpha}&\mapsto \tilde{h}_{\kappa(\alpha)}.
\end{align*}
The symmetric Euler form $(-,-)$ on $K(\cA)$, see (\ref{Euler form and the symmetric Euler form}), can be extended to a bilinear form on $\fh\cong\bbC\otimes_\bbZ K(\cA)$, still denoted by $(-,-)$. Then the isomorphism $\varphi_{\fh}$ preserves bilinear forms $(-,-)$ on $\fh$ and $(-|-)$ on $\tilde{\fh}$, see \ref{symmetric bilinear form}, that is, 
\begin{align}\label{(-,-)=(-,-)}
(h_\alpha,h_{\alpha'})=(\tilde{h}_{\kappa(\alpha)}|\tilde{h}_{\kappa(\alpha')}).
\end{align}

For $\ue\in \bbN I\times \bbN I$, we denote by $P^j=\bigoplus_{i\in I}e^j_iP_i, \alpha=\hat{P}^0-\hat{P}^1\in K(\cA)$ and $\bd=\udim(\ue)\in K_0$, then we have $\bd=\kappa(\alpha)$. Let $\cO\subset \rP_2^{\rad}(A,\ue)$ be a $\rG_{\ue}$-invariant constructible subset consisting of points corresponding to indecomposable radical complexes. The characteristic stack function of the quotient stack $[\cO/\rG_{\ue}]$ and its image in the classical limit are
\begin{align}\label{overline{1}}
&1_{[\cO/\rG_{\ue}]}=[[\cO/\rG_{\ue}]\hookrightarrow \cN^{\rad}\hookrightarrow \cM],\notag\\
&\overline{1}_{[\cO/\rG_{\ue}]}=1\otimes 1_{[\cO/\rG_{\ue}]}\in \fn, 
\end{align}
see Subsection \ref{Lie algebra spanned by contractible complexes and indecomposable radical complexes}. Note that $\cO\subset \rP_2^*(A,\ue)$, thus $\hat{\cO}=\rG_{\bd}.t_{\ue}(\cO)\subset \cP_2(A,\bd)$ is a $\rG_{\bd}$-invariant support-bounded constructible subset consisting of points corresponding to equivalence classes $\tilde{X}=\{X\oplus K_P\oplus K_P^*|P\in \cP\}$ of indecomposable objects in $\cK_2(\cP)$, and so $1_{\hat{\cO}}\in \tilde{\fn}$. Extending the map $\overline{1}_{[\cO/\rG_{\ue}]}\mapsto 1_{\hat{\cO}}$ linearly, we obtain a linear map
$$\varphi_{\fn}:\fn\rightarrow \tilde{\fn}.$$

Hence, we obtain a linear map
\begin{align*}
\varphi=\begin{pmatrix}\varphi_{\fn} &\\ &\varphi_{\fh}\end{pmatrix}:\fg=\fn\oplus \fh\rightarrow \tilde{\fg}=\tilde{\fn}\oplus \tilde{\fh}.
\end{align*}

\begin{theorem}\label{two Lie algebra isomorphism}
The linear map $\varphi:\fg\rightarrow \tilde{\fg}$ is a Lie algebra isomorphism.
\end{theorem}
\begin{proof}
Firstly, we prove that $\varphi$ is a Lie algebra homomorphism. For any $h_\beta,h_{\beta'}\in \fh$, where $\beta,\beta'\in K(\cA)$, we have 
$$\varphi([h_\beta,h_{\beta'}])\overset{(\ref{[h,h]=0})}{=}0\overset{(\ref{[h,h]'=0})}{=}[\tilde{h}_{\kappa(\beta)},\tilde{h}_{\kappa(\beta')}]=[\varphi(h_\beta),\varphi(h_{\beta'})].$$
For any $\ue\in \bbN I\times \bbN I$, we denote by $P^j=\bigoplus_{i\in I}e^j_iP_i$ and $\alpha=\hat{P}^0-\hat{P}^1\in K(\cA)$. For any $\rG_{\ue}$-invariant constructible subset $\cO\subset \rP_2^{\rad}(A,\ue)$ which consists of points corresponding to indecomposable radical complexes, by Lemma \ref{property of h}, we have
\begin{align*}
&\varphi([h_\beta,\overline{1}_{[\cO/\rG_{\ue}]}])\overset{\ref{property of h}}{=}\varphi((\beta,\alpha)\overline{1}_{[\cO/\rG_{\ue}]})=(\beta,\alpha)1_{\hat{\cO}}\\
\overset{(\ref{(-,-)=(-,-)})}{=}&(\tilde{h}_{\kappa(\beta)}|\tilde{h}_{\kappa(\alpha)})1_{\hat{\cO}}\overset{(\ref{Lie algebra g_2 formula})}{=}[\tilde{h}_{\kappa(\beta)},1_{\hat{\cO}}]=[\varphi(h_\beta),\varphi(\overline{1}_{[\cO/\rG_{\ue}]})].
\end{align*} 
Moreover, for any other $\ue'\in \bbN I\times \bbN I$, we denote by $P'^j=\bigoplus_{i\in I}e'^j_iP_i$ and $\alpha'=\hat{P}'^0-\hat{P}'^1\in K(\cA)$. For any $\rG_{\ue'}$-invariant constructible subset $\cO'\subset \rP_2^{\rad}(A,\ue')$ which consists of points corresponding to indecomposable radical complexes, we have
\begin{align*}
&\varphi([\overline{1}_{[\cO/\rG_{\ue}]},\overline{1}_{[\cO'/\rG_{\ue'}]}])\overset{(\ref{Lie algebra g formula})}{=}\varphi([\overline{1}_{[\cO/\rG_{\ue}]},\overline{1}_{[\cO'/\rG_{\ue'}]}]_{\fn}-\chi((\cO\cap\cO'^*)/\rG_{\ue})h_\alpha)\\
=&\varphi([\overline{1}_{[\cO/\rG_{\ue}]},\overline{1}_{[\cO'/\rG_{\ue'}]}]_{\fn})-\chi([(\cO\cap\cO'^*)/\rG_{\ue}])\tilde{h}_{\kappa(\alpha)},\\
&[\varphi(\overline{1}_{[\cO/\rG_{\ue}]}),\varphi(\overline{1}_{[\cO'/\rG_{\ue'}]})]=[1_{\hat{\cO}},1_{\hat{\cO}'}]
\overset{(\ref{Lie algebra g formula})}{=}[1_{\hat{\cO}},1_{\hat{\cO}'}]_{\tilde{\fn}}-\chi([(\hat{\cO}\cap\hat{\cO}'^*)/\rG_{\kappa(\alpha)}])\tilde{h}_{\kappa(\alpha)}.
\end{align*}
By definition, it is clear that $\chi([(\cO\cap\cO'^*)/\rG_{\ue}])=\chi([(\hat{\cO}\cap\hat{\cO}'^*)/\rG_{\kappa(\alpha)}])$, and so 
\begin{align*}
&\varphi([\overline{1}_{[\cO/\rG_{\ue}]},\overline{1}_{[\cO'/\rG_{\ue'}]}])=[\varphi(\overline{1}_{[\cO/\rG_{\ue}]}),\varphi(\overline{1}_{[\cO'/\rG_{\ue'}]})]\\
\Leftrightarrow &\varphi([\overline{1}_{[\cO/\rG_{\ue}]},\overline{1}_{[\cO'/\rG_{\ue'}]}]_{\fn})=[1_{\hat{\cO}},1_{\hat{\cO}'}]_{\tilde{\fn}},
\end{align*}
where $\varphi([\overline{1}_{[\cO/\rG_{\ue}]},\overline{1}_{[\cO'/\rG_{\ue'}]}]_{\fn})$ and $[1_{\hat{\cO}},1_{\hat{\cO}'}]_{\tilde{\fn}}$ are $\rG_{\kappa(\alpha)+\kappa(\alpha')}$-invariant support-bounded constructible functions on $\rP_2(A,\kappa(\alpha)+\kappa(\alpha'))$. By Theorem \ref{Lie algebra g} and Lemma \ref{Lie bracket supported indecomposable}, both of them are support-indecomposable, so it remains to prove 
\begin{equation}\label{varphi(Z)=Z}
\varphi([\overline{1}_{[\cO/\rG_{\ue}]},\overline{1}_{[\cO'/\rG_{\ue'}]}]_{\fn})(\tilde{Z})=[1_{\hat{\cO}},1_{\hat{\cO}'}]_{\tilde{\fn}}(\tilde{Z}),
\end{equation}
where $\tilde{Z}=\{Z_r\oplus K_P\oplus K_P^*|P\in \cP\}$ and $Z_r$ is an indecomposable radical complex. It is enough to prove
\begin{equation}\label{varphi(Z)=F^Z}
\varphi(\overline{1}_{[\cO/\rG_{\ue}]}*\overline{1}_{[\cO'/\rG_{\ue'}]})(\tilde{Z})=1_{\hat{\cO}}*1_{\hat{\cO}'}(\tilde{Z}).
\end{equation}
Once it holds, we have $\varphi(\overline{1}_{[\cO'/\rG_{\ue}]}*\overline{1}_{[\cO/\rG_{\ue}]})(\tilde{Z})=F^{Z_r}_{\cO'\cO}$ similarly, and then (\ref{varphi(Z)=Z}) follows. 

On the one hand, by \textbf{Case $(i=1)$'} in the proof of Theorem \ref{Lie algebra g}, relation (iii) in Definition \ref{relation3}, and $\Upsilon|_{t=-1}=\chi$, the left hand side of (\ref{varphi(Z)=F^Z}) is equal to 
\begin{equation}
\begin{aligned}\label{left constant}
&\sum_{m\in M}\sum_{P,Q\in \cP}\chi([(V_m\times E^1_m)_{Z_r\oplus K_P\oplus K_Q^*}\times \Aut_{\cC_2(\cP)}(Z_r)/G_m\ltimes E^0_m])\\
=&\sum_{m\in M}\sum_{P,Q\in \cP}\chi([(V_m\times E^1_m)_{Z_r\oplus K_P\oplus K_Q^*}\times \Aut_{\cC_2(\cP)}(Z_r)/G_m])
\end{aligned}
\end{equation}
where $(V_m\times E^1_m)_{Z_r\oplus K_P\oplus K_Q^*}\!\subset\!\! (V_m\times E^1_m)_{1,P,Q}$ is a constructible subset consisting of $(v,\xi)$ such that if $v\in V_m\subset \cO\times \cO'$ corresponds to indecomposable radical complexes $X,Y$, then $\Stab_{G_m}(v)\cong \Aut_{\cC_2(\cP)}(X)\times \Aut_{\cC_2(\cP)}(Y)$, the dimensions of $E^1_m\cong \Ext^1_{\cC_2(\cA)}(X,Y)$, $E^0_m\cong \Hom_{\cC_2(\cA)}(X,Y)$ are constant, and the middle term of $\xi\in E^1_m\!\cong\! \Ext^1_{\cC_2(\cA)}(X,Y)$ is isomorphic to $Z_r\oplus K_P\oplus K_Q^*$. Since $E^0_m$ acts trivially, see Proposition \ref{Motivic Riedtmann-Peng}, we can omit it and the condition that the dimension of $E^0_m$ is constant. On the other hand, by definition, see (\ref{constant F}), the right hand side of (\ref{varphi(Z)=F^Z}) is equal to 
\begin{equation}\label{right constant}
F^{Z_r}_{\cO'\cO''}=\chi([W(\cO,\cO';Z_r)/\rG_{\ue}\times \rG_{\ue'}]).
\end{equation}
Define a constructible set and divide it into the disjoint union of constructible subsets
\begin{align*}
\Hom_{\cK_2(\cP)}(\cO,\cO')_{Z_r^*}=&\{(X,Y,h)|X\in \cO,Y\in \cO',h\in \Hom_{\cK_2(\cP)}(X,Y^*)_{Z_r^*}\}\\
=&\bigsqcup_{n\in N}\Hom_{\cK_2(\cP)}(\cO,\cO')_{Z_r^*,n},
\end{align*}
where $\Hom_{\cK_2(\cP)}(\cO,\cO')_{Z_r^*}$ is a bundle over $\cO\times \cO'$ whose fiber at $(X,Y)\in \cO\times \cO'$ is a constructible subset $\Hom_{\cK_2(\cP)}(X,Y^*)_{Z_r^*}\subset\Hom_{\cK_2(\cP)}(X,Y^*)$, and each subset $\Hom_{\cK_2(\cP)}(\cO,\cO')_{Z_r^*,n}$ consists of $(X,Y,h)$ such that the dimension of $\Hom_{\cK_2(\cP)}(X,Y^*)$ is constant. Consider the natural surjective projection
\begin{align*}
\pi:W(\cO,\cO';Z_r)&\rightarrow \Hom_{\cK_2(\cP)}(\cO,\cO')_{Z_r^*}\\
(Y\xrightarrow{f}Z_r\xrightarrow{g}X\xrightarrow{h}Y^*)&\mapsto (X,Y,h),
\end{align*}
we denote by $W(X,Y;Z_r)_h$ its fiber at $(X,Y,h)$. Note that $W(X,Y;Z_r)_h$ is a subset of $\Hom_{\cK_2(\cP)}(Y,Z_r)\times \Hom_{\cK_2(\cP)}(Z_r\times X)$, and $\Aut_{\cK_2(\cP)}(Z_r)$ acts on it via
$$\gamma.(f,g)=(\gamma f,g\gamma^{-1}).$$
Indeed, there is a commutative diagram 
\begin{diagram}[midshaft,size=2em]
Y &\rTo^{f} &Z_r &\rTo^{g} &X &\rTo^{h} &Y^*\\
\vEq & &\dTo^{\gamma} & &\vEq & &\vEq\\
Y &\rTo^{\gamma f} &Z_r &\rTo^{g\gamma^{-1}} &X &\rTo^{h} &Y^*.
\end{diagram}
By the axiom (TR3), the action is transitive, and so there is a unique orbit
$$W(X,Y;Z_r)_h=\Aut_{\cK_2(\cP)}(Z_r).(f,g)\cong \Aut_{\cK_2(\cP)}(Z_r)/\Stab_{\Aut_{\cK_2(\cP)}(Z_r)}(f,g),$$
where $\Stab_{\Aut_{\cK_2(\cP)}(Z_r)}(f,g)=\{\gamma\in \Aut_{\cK_2(\cP)}(Z_r)|\gamma f=f,g\gamma^{-1}=g\}$ is the stabilizer. There is a natural isomorphism
\begin{align*}
\theta:\Stab_{\Aut_{\cK_2(\cP)}(Z_r)}(f,g)\rightarrow &S(f,g)\\
=&\{\gamma-1_{Z_r}\in \End_{\cK_2(\cP)}(Z_r)|\gamma\in \Stab_{\Aut_{\cK_2(\cP)}(Z_r)}(f,g)\}.
\end{align*}
Next, we prove that 
\begin{equation}
S(f,g)\subset \End_{\cK_2(\cP)}(Z_r)\ \textrm{is a vector subspace}. \tag{$***$}
\end{equation}
It is clear that $0\in S(f,g)$. For any $\gamma_1,\gamma_2\in \Stab_{\Aut_{\cK_2(\cP)}(Z_r)}(f,g)$ and $k\in \bbC$, consider the element
$$\gamma=1_{Z_r}+(\gamma_1-1_{Z_r})+k(\gamma_2-1_{Z_r}),$$
it is easy to check that $\gamma f=f,g\gamma=\gamma$. Since $Z_r$ is indecomposable as an object in $\cK_2(\cP)$, the endomorphisms $\gamma_1-1_{Z_r},\gamma_2-1_{Z_r}$ are either homotopy equivalences or nilpotent morphisms. If $\gamma_1-1_{Z_r}$ is a homotopy equivalence, then by $(\gamma_1-1_{Z_r})f=g(\gamma_1-1_{Z_r})=0$, we have $f=g=0$, and so $X\simeq Y^*,Z_r\simeq 0$, a contradiction. Thus $\gamma_1-1_{Z_r}$ is nilpotent, similarly, $\gamma_2-1_{Z_r}$ is nilpotent, and so is $(\gamma_1-1_{Z_r})+k(\gamma_2-1_{Z_r})$. Hence $\gamma\in \Aut_{\cK_2(\cP)}(Z_r)$ and finish the proof of ($***$). Thus we can regard $W(\cO,\cO';Z_r)$ as a bundle on $\Hom_{\cK_2(\cP)}(\cO,\cO')_{Z_r^*}$, whose fibers are isomorphic to quotients of $\Aut_{\cK_2(\cP)}(Z_r)$ by subgroups isomorphic to vector subspaces. 

Now, we can relate (\ref{left constant}) and (\ref{right constant}) as follows. 

\textbf{Step $1$.} Recall that there is a bijection $\Ext^1_{\cC_2(\cA)}(X,Y)\cong \Hom_{\cK_2(\cP)}(X,Y^*)$, see \cite[Lemma 3.3]{Bridgeland-2013}. Moreover, by Lemma \ref{bijection between Ext and Hom}, if $\Ext^1_{\cC_2(\cA)}(X,Y)_{Z_r\oplus K_P\oplus K_Q^*}\not=\varnothing$, then there is a bijection $\Ext^1_{\cC_2(\cA)}(X,Y)_{Z_r\oplus K_P\oplus K_Q^*}\cong \Hom_{\cK_2(\cP)}(X,Y^*)_{Z_r^*}$. Thus we can take $M=N$ such that for each $m\in M$, there is an isomorphism
$$\zeta_0:\bigsqcup_{P,Q\in \cP}(V_m\times E^1_m)_{Z_r\oplus K_P\oplus K_Q^*}\xrightarrow{\cong} \Hom_{\cK_2(\cP)}(\cO,\cO')_{Z_r^*,m},$$
and there is a commutative diagram
\begin{diagram}[midshaft]
\bigsqcup_{m\in M}\bigsqcup_{P,Q\in \cP}(V_m\times E^1_m)_{Z_r\oplus K_P\oplus K_Q^*}\times \Aut_{\cC_2(\cP)}(Z_r) &\rTo^{\pi'} &\bigsqcup_{m\in M}\bigsqcup_{P,Q\in \cP}(V_m\times E^1_m)_{Z_r\oplus K_P\oplus K_Q^*}\\
\dDashto^{\zeta} & &\dTo^{\zeta_0}_{\cong}\\
W(\cO,\cO';Z_r) &\rTo^{\pi} &\Hom_{\cK_2(\cP)}(\cO,\cO')_{Z_r^*},
\end{diagram}
where $\pi'$ is the natural projection which is a principal $\Aut_{\cC_2(\cP)}(Z_r)$-bundle, and 
$$\zeta((v,\xi),\gamma)=(\zeta_0(v,\xi),\overline{p(\gamma)})$$
for any $(v,\xi)\in \bigsqcup_{m\in M}\bigsqcup_{P,Q\in \cP}(V_m\times E^1_m)_{Z_r\oplus K_P\oplus K_Q^*}$ and $\gamma\in \Aut_{\cC_2(\cP)}(Z_r)$, where $p(\gamma)\in \Aut_{\cK_2(\cP)}(Z_r)$ via the homomorphism $p:\Aut_{\cC_2(\cP)}(Z_r)\twoheadrightarrow \Aut_{\cK_2(\cP)}(Z_r)$ in Corollary \ref{automorphism groups coincide}, and $\overline{p(\gamma)}$ is the image of $p(\gamma)$ in the quotient of $\Aut_{\cK_2(\cP)}(Z_r)$ by a subgroup isomorphic to a vector space. Indeed, the quotient is isomorphic to the fiber $\pi^{-1}(\zeta_0(v,\xi))$. Since any fiber of $p$ is also isomorphic to the vector space $\Htp(Z_r,Z_r)$, the fiber of $\zeta$ is isomorphic to a vector space.

\textbf{Step $2$.} It is clear that $G_m\cong \rG_{\ue}\times \rG_{\ue'}$ for any $m\in M$. Moreover, $\zeta$ induces a morphism $G_m.((v,\xi),\gamma)\mapsto \rG_{\ue}\times \rG_{\ue}.\zeta(((v,\xi),\gamma))$, and then induces a morphism
$$\overline{\zeta}:\bigsqcup_{m\in M}\bigsqcup_{P,Q\in \cP}[(V_m\times E^1_m)_{Z_r\oplus K_P\oplus K_Q^*}\times \Aut_{\cC_2(\cP)}(Z_r)/G_m]\rightarrow [W(\cO,\cO';Z_r)/\rG_{\ue}\times \rG_{\ue'}].$$

\textbf{Step $3$.} By Proposition \ref{Motivic Riedtmann-Peng}, the  $\Stab_{G_m}(v)$-action on $E^1_m$ coincide with the $\Aut_{\cC_2(\cP)}(X)\times \Aut_{\cC_2(\cP)}(Y)$-action on $\Ext^1_{\cC_2(\cA)}(X,Y)$ which is free, and so $G_m$ acts freely on $\bigsqcup_{m\in M}\bigsqcup_{P,Q\in \cP}(V_m\times E^1_m)_{Z_r\oplus K_P\oplus K_Q^*}\times \Aut_{\cC_2(\cP)}(Z_r)$. The fiber of $\overline{\zeta}$ at the geometric point corresponding to the orbit $\rG_{\ue}\times\rG_{\ue'}.(Y\xrightarrow{f}Z_r\xrightarrow{g}X\xrightarrow{h}Y^*)$ is isomorphic to the product of a vector space (contributed by the fiber of $\zeta$) and the stabilizer
\begin{align*}
\Stab(f,g,h)\!=&\{(\alpha,\beta)\in\! \rG_{\ue}\times \rG_{\ue'}|\alpha.X=X,\beta.Y=Y,f\beta^{-1}=f,\alpha g=g, \beta^*h\alpha^{-1}=h\}\\
=&\{(\alpha,\beta)\in \!\Aut_{\cC_2(\cP)}(X)\times \Aut_{\cC_2(\cP)}(Y)|f\beta^{-1}=f,\alpha g=g, \beta^*h\alpha^{-1}=h\},
\end{align*}
where we use the fact that $\alpha.X=X$ and $\beta.Y=Y$ imply that $\alpha\in \Stab_{\rG_{\ue}}(X)\cong \Aut_{\cC_2(\cP)}(X)$ and $\beta\in\Stab_{\rG_{\ue'}}(Y)\cong \Aut_{\cC_2(\cP)}(Y)$. Note that the conditions that $f\beta^{-1}=f,\alpha g=g, \beta^*h\alpha^{-1}=h$ are about morphisms in $\cK_2(\cP)$, then by Corollary \ref{automorphism groups coincide}, there is a surjective group homomorphism
\begin{align*}
p:&\Stab(f,g,h)\twoheadrightarrow\\
&\overline{\Stab}(f,g,h)\!=\!\{(\alpha,\beta)\in \!\Aut_{\cK_2(\cP)}(X)\!\times\!\Aut_{\cK_2(\cP)}(Y)|f\beta^{-1}\!=\! f,\alpha g=\!g, \beta^*h\alpha^{-1}=\!h\}
\end{align*}
with kernel which is bijective to the vector space $\Htp(X,X)\oplus \Htp(Y,Y)$. Applying Corollary \ref{fibre naive Euler characteristic}, we obtain
$$\chi(\Stab(f,g,h))=\chi(\overline{\Stab}(f,g,h)).$$
There is a natural bijection from $\overline{\Stab}(f,g,h)$ to
$$
\overline{S}(f,g,h)=\{(\alpha-1_X,\beta-1_Y)\in \End_{\cK_2(\cP)}(X)\times \End_{\cK_2(\cP)}(Y)|(\alpha,\beta)\in \overline{\Stab}(f,g,h)\}.
$$
Next, we prove that
\begin{equation}
\overline{S}(f,g,h)\subset \End_{\cK_2(\cP)}(X)\times \End_{\cK_2(\cP)}(Y)\ \textrm{is a vector subspace.} \tag{$***\ *$}
\end{equation}
It is clear that $(0,0)\in \overline{S}(f,g,h)$. For any $(\alpha_1,\beta_1),(\alpha_2,\beta_2)\in \overline{\Stab}(f,g,h)$ and $k\in \bbC$, consider the element
$$(\alpha,\beta)=(1_X,1_Y)+(\alpha_1-1_X,\beta_1-1_Y)+k(\alpha_2-1_X,\beta_2-1_Y),$$
it is easy to check that $f=f\beta,\alpha g=g, \beta^*h=h\alpha$. Since $X,Y$ are indecomposable as objects in $\cK_2(\cP)$, the endomorphisms $\alpha_1-1_X,\beta_1-1_Y$ are either homotopy equivalences or nilpotent morphisms. If $\alpha_1-1_X$ is a homotopy equivalence, then $(\alpha_1-1_X)g=0$ implies $g=0$, and then $Y\simeq X^*\oplus Z_r$, a contradiction. If $\beta_1-1_Y$ is a homotopy equivalence, then $f(\beta_1-1_Y)=0$ implies $f=0$, and then $X\simeq Y^*\oplus Z_r$, a contradiction. Thus $\alpha_1-1_X,\beta_1-1_Y$ are nilpotent, similarly, $\alpha_2-1_X,\beta_2-1_Y$ are nilpotent, and then so is $(\alpha_1-1_X,\beta_1-1_Y)+k(\alpha_2-1_X,\beta_2-1_Y)$. Hence $(\alpha,\beta)\in \Aut_{\cK_2(\cP)}(X)\times\Aut_{\cK_2(\cP)}(Y)$ and finish the proof of ($***\ *$). So 
$$\chi(\overline{\Stab}(f,g,h))=\chi(\overline{S}(f,g,h))=1.$$
Hence, any fiber of $\overline{\zeta}$ has the same naive Euler characteristic $1$. Applying Corollary \ref{fibre naive Euler characteristic} for $\overline{\zeta}$, we obtain 
\begin{equation}
\begin{aligned}\label{g=F}
&\sum_{m\in M}\sum_{P,Q\in \cP}\chi([(V_m\times E^1_m)_{Z_r\oplus K_P\oplus K_Q^*}\times \Aut_{\cC_2(\cP)}(Z_r)/G_m])\\
=&\chi([W(\cO,\cO';Z_r)/\rG_{\ue}\times \rG_{\ue'}]),
\end{aligned}
\end{equation}
and finish the proof of (\ref{varphi(Z)=F^Z}). Therefore, $\varphi:\fg\rightarrow \tilde{\fg}$ is a Lie algebra homomorphism.

Finally, we have already known that $\varphi_{\fh}:\fh\rightarrow \tilde{\fh}$ is an isomorphism. Note that an object $X\in \cK_2(\cP)$ is indecomposable if and only if its radical part $X_r\in \cC_2(\cP)$ is an indecomposable radical complex, and so $\varphi_{\fn}:\fh\rightarrow \tilde{\fn}$ is also an isomorphism. Therefore, $\varphi:\fg\rightarrow \tilde{\fg}$ is an isomorphism.
\end{proof}

\begin{proposition}
(a) The symmetric bilinear form $(-|-)$ on $\tilde{\fh}$ can be extended to a symmetric bilinear form $(-\mid-)$ on $\tilde{\fg}$ such that\\
(i) $(-|-)$ is invariant, that is, 
\begin{align*}
([\tilde{h}_\bd,1_{\hat{\cO}_1}]\mid1_{\hat{\cO}_2})&=(\tilde{h}_{\bd}\mid[1_{\hat{\cO}_1},1_{\hat{\cO}_2}]),\\
([1_{\hat{\cO}_1},1_{\hat{\cO}_2}]\mid1_{\hat{\cO}_3})&=(1_{\hat{\cO}_1}\mid[1_{\hat{\cO}_2},1_{\hat{\cO}_3}]);
\end{align*}
(ii) $(1_{\hat{\cO}_1}|1_{\hat{\cO}_2})=-\chi([(\hat{\cO}_1\cap\hat{\cO}_2^*)/\rG_{\bd_1}]);$\\
(iii) $(\tilde{\fh}\mid\tilde{\fn})=0$;\\
(iv) $(-|-)|_{\tilde{\fn}\times \tilde{\fn}}$ is non-degenerate\\
for any $\bd\in K_0$ and $\rG_{\bd_i}$-invariant support-bounded constructible subset $\hat{\cO}_i\subset \rP_2(A,\bd_i)$ consisting of points corresponding to equivalence classes of indecomposable objects in $\cK_2(\cP)$, $i=1,2,3$.\\
(b) The symmetric Euler form $(-,-)$ on $\fh\cong \bbC\otimes K(\cA)$ can be extended to a symmetric bilinear form $(-,-)$ on $\fg$ such that\\
(i) $(-,-)$ is invariant, that is, 
\begin{align*}
([h_\alpha, \overline{1}_{[\cO_1/\rG_{\ue_1}]}],\overline{1}_{[\cO_2/\rG_{\ue_2}]})&=(h_\alpha,[\overline{1}_{[\cO_1/\rG_{\ue_1}]},\overline{1}_{[\cO_2/\rG_{\ue_2}]}]),\\
([\overline{1}_{[\cO_1/\rG_{\ue_1}]},\overline{1}_{[\cO_2/\rG_{\ue_2}]}],\overline{1}_{[\cO_3/\rG_{\ue_3}]})&=(\overline{1}_{[\cO_1/\rG_{\ue_1}]},[\overline{1}_{[\cO_2/\rG_{\ue_2}]},\overline{1}_{[\cO_3/\rG_{\ue_3}]}]);
\end{align*}
(ii) $(\overline{1}_{[\cO_1/\rG_{\ue_1}]},\overline{1}_{[\cO_2/\rG_{\ue_2}]})=-\chi([(\cO_1\cap\cO_2^*)/\rG_{\ue_1}]);$\\
(iii) $(\fh,\fn)=0$;\\
(iv) $(-,-)|_{\fn\times \fn}$ is non-degenerate\\
for any $\alpha\in K(\cA)$ and $\rG_{\ue_i}$-invariant constructible subset $\cO_i\subset \rP_2^{\rad}(A,\ue_i)$ consisting of points corresponding to indecomposable radical complexes, $i=1,2,3$.\\
(c) The Lie algebra isomorphism $\varphi:\fg\rightarrow \tilde{\fg}$ is compatible with these two symmetric bilinear forms, that is, $(x,y)=(\varphi(x)\mid\varphi(y))$ for any $x,y\in \fg$.
\end{proposition}
\begin{proof}
(a) We only need to prove the symmetric bilinear form extended linearly by (ii) and (iii) satisfies (i) and (iv). By definition, see (\ref{Lie algebra g_2 formula}), we have
\begin{align*}
([\tilde{h}_{\bd}, 1_{\hat{\cO}_1}]\mid 1_{\hat{\cO}_2})
=&(\tilde{h}_{\bd}\mid\tilde{h}_{\bd_1})(1_{\hat{\cO}_1}\mid 1_{\hat{\cO}_2})\\
=&-(\tilde{h}_{\bd}\mid\tilde{h}_{\bd_1})\chi([(\hat{\cO}_1\cap\hat{\cO}_2^*)/\rG_{\bd_1}]),\\
(\tilde{h}_{\bd}\mid[1_{\hat{\cO}_1},1_{\hat{\cO}_2}])
=&(\tilde{h}_{\bd}\mid[1_{\hat{\cO}_1},1_{\hat{\cO}_2}]_{\tilde{\fn}}-\chi([(\hat{\cO}_1\cap\hat{\cO}_2^*)/\rG_{\bd_1}])\tilde{h}_{\bd_1})\\
=&-(\tilde{h}_{\bd}\mid\tilde{h}_{\bd_1})\chi([(\hat{\cO}_1\cap\hat{\cO}_2^*)/\rG_{\bd_1}]).
\end{align*}
Moreover, by definition and (\ref{F=Fchi}), we have
\begin{align*}
([1_{\hat{\cO}_1},1_{\hat{\cO}_2}]\mid1_{\hat{\cO}_3})
=&([1_{\hat{\cO}_1},1_{\hat{\cO}_2}]_{\tilde{\fn}}-\chi([(\hat{\cO}_1\cap\hat{\cO}_2^*)/\rG_{\bd_1}])\tilde{h}_{\bd_1}\mid1_{\hat{\cO}_3})\\
=&(\sum_{\ue\leqslant \ue_1+\ue_2}\sum_{(r,s)\in \cI^{\ue}_{\ue_1\ue_2}}(F^{L_{r,s}}_{\cO_1\cO_2}-F^{L_{r,s}}_{\cO_2\cO_1})1_{\langle L_{r,s}\rangle}\mid 1_{\hat{\cO}_3})\\
=&\sum_{\ue\leqslant \ue_1+\ue_2}\sum_{(r,s)\in \cI^{\ue}_{\ue_1\ue_2}}\!\!\!\!\!\!(F^{L_{r,s}}_{\cO_1\cO_2}\!-\!F^{L_{r,s}}_{\cO_2\cO_1})\chi([(\rG_{\bd_1+\bd_2}.t_{\ue}(\langle L_{r,s}\rangle)\cap \hat{\cO}^*_3)/\rG_{\bd_1+\bd_2}])\\
=&-F^{\hat{\cO}_3^*}_{\hat{\cO}_1\hat{\cO}_2}+F^{\hat{\cO}_3^*}_{\hat{\cO}_2\hat{\cO}_1}
\end{align*}
Similarly, we have $(1_{\hat{\cO}_1}\mid[1_{\hat{\cO}_2},1_{\hat{\cO}_3}])=-F^{\hat{\cO}_1^*}_{\hat{\cO}_2\hat{\cO}_3}+F^{\hat{\cO}_1^*}_{\hat{\cO}_3\hat{\cO}_2}$, and then (i) follows (\ref{F=F}). 

For any nonzero $\hat{f}=\sum^n_{i=1}c_i1_{\hat{\cO}_i}\in \tilde{\fn}_{\bd}$, where $c_i\in \bbC$ and $\hat{\cO}_i\subset\rP_2(A,\bd)$ is a $\rG_{\bd}$-invariant support-bounded constructible subset, with loss of generality, we assume $c_1\not=0$ and $\hat{\cO}_1\cap\hat{\cO}_i=\varnothing$ for $i\not=1$. We take $\tilde{Z}\in \hat{\cO}_1$ and let $\hat{\cO}_{\tilde{Z}^*}\subset \rP_2(A,-\bd)$ be the $\rG_{-\bd}$-orbit of $\tilde{Z}^*$, then $(\hat{f}\mid 1_{\hat{\cO}_{\tilde{Z}^*}})=c_1\not=0$. This finishes the proof of (iv).

(b) can be proved by similar argument, and then (c) is clear.
\end{proof}

\section{Examples}

In this section, we present two examples.

\subsection{Kac-Moody Lie algebra}\

Consider the case $A=\bbC Q$, where $Q$ is a finite acyclic quiver. In this case, the category $\cA$ is hereditary and $\cR_A=\cD^b(A)/[2]\simeq \cK_2(\cP)$, see \cite[Subsection 7]{Peng-Xiao-1997}. There are two equivalent descriptions of indecomposable objects in $\cR_A\simeq \cK_2(\cP)$. The one is given by \cite[Lemma 5.2]{Happel-1988}, we have
$$\ind \cR_A=\ind\cA\cup(\ind\cA)^*,$$
where $\ind\cA,\ind \cR_A$ are the sets of isomorphism classes of indecomposable objects in $\cA,\cR_A$, respectively, and $\ind\cA, (\ind\cA)^*\subset \ind\cR_A$ via the embedding as stalk complexes. The other one is given by \cite[Lemma 4.2]{Bridgeland-2013}, indecomposable objects in $\cK_2(\cP)$, that is, indecomposable radical complexes in $\cC_2(\cP)$, are of the form 
\begin{diagram}[midshaft,size=2em]
C_M=(P &\pile{\rTo^{f}\\ \lTo_0} &Q), & &C_M^*=(Q &\pile{\rTo^{0}\\ \lTo_{-f}} &P)
\end{diagram}
where $M\in \ind\cA$ and $0\rightarrow P\xrightarrow{f}Q\rightarrow M\rightarrow 0$ is a minimal projective resolution. 

This deduces a triangle decomposition of $\fg\cong \tilde{\fg}$ defined in Subsections \ref{Lie algebra spanned by contractible complexes and indecomposable radical complexes} and \ref{Lie algebra spanned by supported-indecomposable function and the Grothendieck group},
$$\tilde{\fg}=\tilde{\fn}^+\oplus\tilde{\fh}\oplus \tilde{\fn}^-,$$
where $\tilde{\fn}^+$ and $\tilde{\fn}^-\subset\tilde{\fn}$ are the subspaces spanned by constructible functions supported on $\{C_M|M\in \ind \cA\}$ and $\{C_M^*|M\in \ind \cA\}$, respectively, so that $\tilde{\fn}=\tilde{\fn}^+\oplus\tilde{\fn}^-$. It is easy to prove that $\tilde{\fn}^+$ and $\tilde{\fn}^-$ are closed under the Lie bracket. Hence we obtain a triangle decomposition of Lie algebra. Moreover, together with the non-degenerate symmetric bilinear form $(-|-)_{\tilde{\fn}\times \tilde{\fn}}$, the Lie algebra $\tilde{\fg}$ is a generalized Kac-Moody Lie algebra in the sense of \cite{Borcherds-1998}. 

Letting $\cC\tilde{\fg}\subset \tilde{\fg}$ be the Lie subalgebra generated by characteristic functions $1_{\hat{\cO}_X}$ of orbits of exceptional objects $X\in \cR_A\simeq \cK_2(\cP)$, one can prove that $\cC\tilde{\fg}$ is isomorphic to the corresponding derived Kac-Moody Lie algebra $\cG$ determined by the quiver $Q$, see \cite[Subsection 4.7]{Peng-Xiao-2000}.

\subsection{Loop algebra}\

Let $\bbC\bbP^1$ be the complex projective line, $\underline{\lambda}=(\lambda_1,...,\lambda_n)$ be a sequence of distinct points of $\bbC\bbP^1$ and $\underline{p}=(p_1,...,p_n)\in \bbZ_{> 0}^n$ be a sequence of positive integers, then the triple $\mathbb{X}=(\bbC\bbP^1,\underline{\lambda},\underline{p})$ is called a wighted projective line. The category $\textrm{Coh} \mathbb{X}$ of coherent sheaves on $\mathbb{X}$ is a hereditary category, so
$$\ind \cR_{\mathbb{X}}=\ind \textrm{Coh} \mathbb{X}\cup (\ind\textrm{Coh} \mathbb{X})[1],$$
where $\ind \cR_{\mathbb{X}}$ is the set of isomorphism classes of indecomposable objects in the root category $\cR_{\mathbb{X}}=\cD^b(\textrm{Coh} \mathbb{X})/[2]$. For any $\mathbb{X}=(\bbC\bbP^1,\underline{\lambda},\underline{p})$, Ringel introduced the canonical algebra $\Lambda=\Lambda(\underline{\lambda},\underline{p})$ in \cite{Ringel-1984} such that 
$$\cD^b(\textrm{Coh} \mathbb{X})\simeq \cD^b(\Lambda),$$
and so
$$\cR_{\mathbb{X}}\simeq \cR_\Lambda\simeq \cK_2(\cP_\Lambda).$$
For the canonical algebra $\Lambda$, there are $\bbC$-Lie algebras $\fg\cong \tilde{\fg}$ defined in Subsections \ref{Lie algebra spanned by contractible complexes and indecomposable radical complexes} and \ref{Lie algebra spanned by supported-indecomposable function and the Grothendieck group}. 

Associating to $\mathbb{X}=(\bbC\bbP^1,\underline{\lambda},\underline{p})$, there is a star-shaped graph
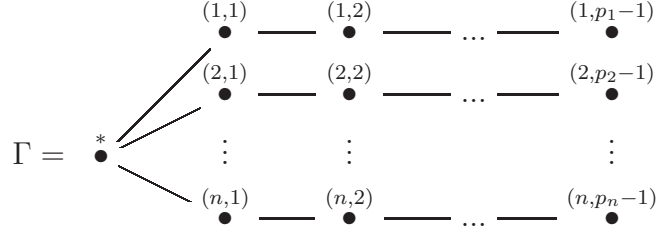
\begin{diagram}[size=2em]
& & &\overset{(1,1)}{\bullet} &\rLine &\overset{(1,2)}{\bullet} &\rLine &... &\rLine&\overset{(1,p_1-1)}{\bullet}\\
& &\ruLine &\overset{(2,1)}{\bullet} &\rLine &\overset{(2,2)}{\bullet} &\rLine &... &\rLine&\overset{(2,p_2-1)}{\bullet}\\
\Gamma=&\overset{*}{\bullet} &\ruLine(2,1) &\vdots & &\vdots & & & &\vdots\\
&&\luLine(2,1) &\overset{(n,1)}{\bullet} &\rLine &\overset{(n,2)}{\bullet} &\rLine &... &\rLine&\overset{(n,p_n-1)}{\bullet}
\end{diagram}
Let $\cG$ be the Kac-Moody Lie algebra corresponding to the graph $\Gamma$, and $\cL\cG$ be the loop algebra of $\cG$. 

In \cite{Crawley-Boevey-2010}, Crawley-Boevey considered Peng-Xiao's Lie algebra $\tilde{\fg}_{(q-1)}$ for the root category $\cR_{\mathbb{X}_{\bbF_q}}$ of coherent sheaves on $\mathbb{X}_{\bbF_q}$ over a finite field $\bbF_q$. He found a collection of elements in $\tilde{\fg}_{(q-1)}$ satisfying the generating relations of $\cL\cG$, and then proved an analogue of Kac's Theorem for weighted projective lines \cite[Theorem 1]{Crawley-Boevey-2010}. Later, in \cite{Dou-Sheng-Xiao-2011}, Dou-Sheng-Xiao considered $\tilde{\fg}$ and gave an alternative proof over $\bbC$. 

Let $\cC\tilde{\fg}\subset \tilde{\fg}$ be the Lie subalgebra generated by characteristic functions of orbits of elements found by Crawley-Boevey in \cite{Crawley-Boevey-2010}. It is isomorphic to the Lie subalgebra of $\tilde{\fg}$ generated by characteristic functions of orbits of exceptional objects, see \cite[Subsection 8.4]{Ruan-Zhang-2020}. There is a canonical surjective Lie algebra homomorphism $\cC\tilde{\fg}\twoheadrightarrow \cL\cG$ such that $\cC\tilde{\fg}$ and $\cL\cG$ have the same root system.\\

\subsection*{Acknowledgements}\

J. Fang is partially supported by National Key R\&D Program of China (Grant No. 2020YFE0204200). Y. Lan is partially supported by National Natural Science Foundation of China (Grant No. 12288201). J. Xiao is partially supported by National Natural Science Foundation of China (Grant No. 12031007).

\bibliography{mybibfile}{}

\end{document}